\PassOptionsToPackage{sort&compress}{natbib} 
\documentclass[english,authoryear,11pt]{elsarticle}

\makeatletter
\def\ps@pprintTitle{%
}
\makeatother

\usepackage{geometry}

\usepackage[utf8]{inputenc}
\usepackage[T1]{fontenc}
\usepackage[english]{babel}
\usepackage{algorithm,algorithmic,amsmath}
\usepackage{amsthm}

\usepackage{amsfonts}
\usepackage{amssymb}
\usepackage{graphicx}
\usepackage{fancyhdr}
\usepackage{float}
\usepackage[font = footnotesize, labelfont = bf, labelsep = period]{caption}
\usepackage{soul}
\usepackage{tabularx}
\usepackage{enumitem}
\usepackage{comment}
\usepackage[hang,flushmargin]{footmisc}

\usepackage{sectsty}
\sectionfont{\large}
\subsubsectionfont{\normalsize \bf}

\usepackage{graphicx}
\graphicspath{ {./Plots_new/} }
\usepackage{subcaption}

\usepackage[usenames, dvipsnames]{xcolor}

\usepackage[colorlinks]{hyperref}
\usepackage{cleveref}
\addto\extrasenglish{												%
}

\usepackage[toc,page]{appendix}
\usepackage{chngcntr}
\usepackage{apptools}

\usepackage{etoolbox}
\makeatletter
\patchcmd{\hyper@makecurrent}{%
    \ifx\Hy@param\Hy@chapterstring
        \let\Hy@param\Hy@chapapp
    \fi
}{%
    \iftoggle{inappendix}{
        \@checkappendixparam{chapter}%
        \@checkappendixparam{section}%
        \@checkappendixparam{subsection}%
        \@checkappendixparam{subsubsection}%
        \@checkappendixparam{paragraph}%
        \@checkappendixparam{subparagraph}%
    }{}%
}{}{\errmessage{failed to patch}}

\newcommand*{\@checkappendixparam}[1]{%
    \def\@checkappendixparamtmp{#1}%
    \ifx\Hy@param\@checkappendixparamtmp
        \let\Hy@param\Hy@appendixstring
    \fi
}
\makeatletter

\newtoggle{inappendix}
\togglefalse{inappendix}

\apptocmd{\appendix}{\toggletrue{inappendix}}{}{\errmessage{failed to patch}}
\apptocmd{\subappendices}{\toggletrue{inappendix}}{}{\errmessage{failed to patch}}

\newcommand*{\fullnameref}[1]{\hyperref[{#1}]{\autoref*{#1}: \nameref*{#1}}}

\newcommand*{\fullref}[1]{\hyperref[{#1}]{\autoref*{#1}}}

\makeatletter
\renewcommand*{\eqref}[1]{%
  \hyperref[{#1}]{\textup{\tagform@{\ref*{#1}}}}%
}
\makeatother
\newtheorem{lemma}{Lemma}
\newtheorem{theorem}{Theorem}
\newtheorem{proposition}{Proposition}

\newcommand{\myparagraph}[1]{\paragraph{#1} \mbox{}\\ \mbox{} \\}

\usepackage{pdfpages}

\begin{document}

\begin{titlepage}
   \begin{center}
       \vspace*{4cm}

       {\Huge \textbf{Frequentist Asymptotics of Variational Laplace} \par}

       \vspace{1cm}

       {\Large \textbf{Janis Keck}}
 
       \vspace{2cm}
            
       {\large Thesis Submitted for the Degree of\\
       Master of Science\\
       Social, Cognitive and Affective Neuroscience\\ \par}

      \vspace{1cm}
      {\large
      Arbeitsbereich Computational Cognitive Neuroscience \\
      Department of Education and Psychology\\
      Freie Universität Berlin \\ \par}

      \vspace{1cm}     
      {\large
      Primary Reviewer: Prof. Dr. Dirk Ostwald \\
      Secondary Reviewer: Prof. Dr. Felix Blankenburg\\ \par}
      \vspace{1cm}
     
      August 2020
            
   \end{center}
\end{titlepage}

\begin{frontmatter}

\begin{abstract}
Variational inference is a general framework to obtain approximations to the posterior distribution in a Bayesian context. In essence, variational inference entails an optimization over a given family of probability distributions to choose the member of this family best approximating the posterior. 
Variational Laplace, an iterative update scheme motivated by this objective, is widely used in different contexts in the cognitive neuroscience community. However, until now, the theoretical properties of this scheme have not been systematically investigated. Here, we study variational Laplace in the light of frequentist asymptotic statistics. Asymptotical frequentist theory enables one to judge the quality of point estimates by their limit behaviour. We apply this framework to find that point estimates generated by variational Laplace enjoy the desirable properties of asymptotic consistency and efficiency in two toy examples. Furthermore, we derive conditions that are sufficient to establish these properties in a general setting. Besides of point estimates, we also study the frequentist convergence of distributions in the sense of total variation distance, which may be useful to relate variational Laplace both to recent findings regarding variational inference as well as to classical frequentist considerations on the Bayesian posterior.
Finally, to illustrate the validity of our theoretical considerations, we conduct simulation experiments in our study examples.
\end{abstract}

\begin{keyword}
Bayesian inference \sep variational Bayes \sep variational inference \sep variational Laplace \sep point estimation \sep asymptotic statistics \sep machine learning
\end{keyword}

\end{frontmatter}

\section{Introduction}

Bayesian approaches to inference enjoy an ongoing popularity in the cognitive neuroscience community. The cornerstone of Bayesian inference is the posterior distribution, which depending on the complexity of the model used may be impossible to evaluate exactly. Variational inference refers to a general strategy to overcome this limitation, yielding an approximate 'variational' posterior. This is achieved by selecting the best approximating member of some restricted class of distributions,
which is more amenable to analysis \citep{blei2017}. In practice, this is done by maximizing a functional of the approximating density, the "evidence-lower-bound" ($\operatorname{ELBO}$).  Hence, variational inference is replacing the calculation of the posterior with an optimization of this functional over a subclass of probability distributions \citep{wainwright2008graphical}.
Importantly, if the subclass can be described by some finite set of parameters, the objective becomes an optimization over a set of real parameters \citep{Lu2017Gaussian,opper2009variational}. Ideally, the resulting problem may be solved by classical numerical optimization methods in a straightforward manner, although in more intriciate models additional approximations may be needed to make the objective tractable \citep{starke2017}. 

Due to practical benefits like scalability and speed, variational inference enjoys wide popularity in different research areas \citep{blei2017}. Prominently, methods falling under the 
 variational inference framework are actively developed in the machine learning and statistical communities \citep {blei2017,hoffman2013stochastic}. Nonetheless, also in the cognitive neuroscience community particular implementations taylored to models commonly used in our field  have been conceived.
To be specific, a great share of the popularity in cognitive neuroscience may be attributed to the implementations in SPM (https://www.fil.ion.ucl.ac.uk/spm/), where such techniques are used in a wide range of contexts \citep{penny2011statistical}. To name just a few, variational inference is used for Bayesian inference in General Linear Model (GLM) analysis of fMRI data \citep{friston2002classical, starke2017}, distributed source reconstruction in M/EEG inverse problems \citep{friston2008multiple}, hidden parameter estimation in dynamic causal models \citep{ostwald2016probabilistic,daunizeau2009variational,frassle2017regression} and in Bayesian treatments of multivariate methods \citep{friston2008Bayesian,friston2019variational}.

 Despite the different nature of these applications, many of the implementations make use of a common scheme to find an approximate solution to the optimization problem of variational inference: 
"Variational Laplace" \citep{friston2007variational} posits the approximating family of densities to be Gaussian and in essence works by replacing the $\operatorname{ELBO}$ in the objective by its Taylor approximation. 
Although this scheme is frequently used for inference in practice, the theoretical research on this method seems to be sparse. The scheme could show satisfying results in terms of parameter recovery in simulation studies  \citep{starke2017,friston2007variational}, but to our knowledge no study has investigated the theoretical properties of point estimates obtained by variational Laplace. 

In this work, we want to contribute to filling this gap by analyzing those point estimates from the standpoint of frequentist asymptotical statistical theory. Well-established quality criteria of point estimators like consistency and asymptotic normality allow to judge whether a point estimator will recover the true parameter in the limit and  also provide a tool to compare them among each other \citep{van2000asymptotic,shao2003mathematical}. Thus, in this work  we investigate variational Laplace in the light of these criteria. As variational inference is naturally situated in a Bayesian context, where inference typically culminates in a distribution rather than a single point estimate,  we also  study the asymptotics of the variational distributions. To this end, we draw from the established literature studying Bayesian quantities from a frequentist perspective \citep{van2000asymptotic,lehmann2006theory}. This will also allow us to compare our results to recent findings concerning those properties in general variational inference. 

The present work has the following structure:
In \autoref{sec:varinf} we introduce the reader to the main ideas lying at the core of variational inference. We also give a brief distinction between free-form approaches (\autoref{sec:free-form}) and fixed-form approaches (\autoref{sec:fixed-form}). The key notions of Variational Laplace, which is an approximate scheme applied in the latter case, will be discussed in \autoref{sec:varlap}. In this section we will also see that in the special case of a single parameter the point estimates of variational Laplace coincide with maximum-a-posteriori estimates, such that all theoretical properties of the latter equally apply to variational Laplace in this setting. We proceed by examining a fairly general model, which encompasses many of the applications of variational Laplace in \autoref{sec:model}. This model, comprising of a known, possibly non-linear transform function and parametric Gaussian noise, will serve as the study example for most parts of this work. We provide details on the application of variational Laplace in this model class and then study two special cases: The case of just one parameter set (\autoref{sec:singleparam}) and the case of a linear transform (\autoref{sec:linear}). In both cases, we give additional details on the iterates that are obtained when applying variational Laplace to these models. In \autoref{sec:freqcriteria}, we introduce the frequentist  asymptotic quality criteria for point estimators that we are studying in this work.
To apply these criteria to variational Laplace, we briefly delineate Bayesian point estimation via the posterior mean and how this naturally transfers to variational approaches in  \autoref{sec:bayescriteria}. Furthermore, we also briefly recapitulate popular results on the frequentist asymptotics of the Bayesian posterior distribution such as the Bernstein-von-Mises theorem. 
Following this introduction, to situate our work in the current state of research, we discuss previous work concerning the frequentist asymptotics of variational inference in \autoref{sec:review}. We will in particular focus on one new finding: In contrast to earlier results that were mostly dealing with specific models  \citep{wang2004lack,wang2012convergence}, in this recent work the consistency and asymptotic normality have been established under a general set of conditions \citep{wang2019frequentist}. These results will also serve as a reference for our findings in variational Laplace, which is why we state them in some detail. After these preparatory chapters, we are ready to present our results: In \autoref{sec:resultsingleparam} we show that in the case of one parameter, variational Laplace is consistent provided some conditions on the transform function. In \autoref{sec:resultlinear} we show that variational Laplace in the linear example model is consistent and asymptotically normal under mild additional assumptions. We furthermore observe that the asymptotics of the variational distributions are equivalent to those of the posterior distribution given by the Bernstein-von-Mises theorem.  Finally, in \autoref{sec:resultsgeneral} we provide general conditions under which consistency and asymptotic normality of variational Laplace will hold and also discuss the limits of the distribution. 
Having presented these results, we illustrate their validity by simulation studies in our example models in \autoref{sec:simulations}. We conclude our work with a discussion of possible limitations of our approach, open questions to be adressed in future work and the implications for practice (\autoref{sec:discussion}).

\section {Variational Inference}\label{sec:varinf}
In the following, we want to condense general ideas behind the variational inference framework. Variational inference is commonly employed in a Bayesian setting. Hence, the starting point of inference is a probabilistic model comprising observable and unobservable random variables, where the respective distributions admit densities, such that the joint probability density takes the form 
\begin{equation}
p(y,\theta) = p(y|\theta) p(\theta)
\end{equation}
where $y \in \mathbb {R}^m$ corresponds to observable random variables and $\theta \in \mathbb{R}^d$ corresponds to unobservable random variables. In our case, $y$ could be data that is acquired in an experiment while $\theta$ would be unknown parameters governing the model. In this setting, having observed $y$, to calculate the posterior $p(\theta|y)$, one would have to evaluate the following fraction
\begin{equation}
p(\theta|y) = \frac{p(y,\theta)}{p(y)} = \frac{p(y,\theta)}{\int p(y,\theta) d\theta}.
\end{equation}
However, the integral in the denominator may be analytically intractable and a numerical evaluation too costly. The rationale in variational inference  is to select a restricted class of "variational" probability density functions $\mathcal{Q}$ , which is more tractable, to approximate the posterior.
We will discuss common choices of this 'variational family' in the next sections.

The desired outcome of variational inference is to choose an optimal member $q^* \in \mathcal{Q}$ in the sense that 
\begin{equation}\label{KL_opt}
q^* = \underset{q  \in \mathcal{Q}}{\arg \min }  \operatorname{D_{KL}}(q || p(\cdot | y)),
\end{equation}
where the Kullback-Leibler-Divergence $\operatorname{D_{KL}}$ for probability densities $q,p$ is defined as
\begin{equation}
\operatorname{D_{KL}}(q || p)  = \int q(x)\ln{\frac{q(x)}{p(x)}} dx.
\end{equation}
The Kullback-Leibler-Divergence has the two important properties 
\begin{equation}
\operatorname{D_{KL}}(q || p) \geq 0
\end{equation}
and
\begin{equation}
\operatorname{D_{KL}}(q || p)  = 0 \iff  q=p \ a.e. ,
\end {equation}
and  is widely used as a (non-symmetric) distance measure between probability densities. Hence, minimizing it corresponds to the variational density $q^*$ being in some sense closest to the posterior \citep{bishop2006pattern,murphy2012machine,van2000asymptotic}.

Unfortunately, the optimization in \eqref{KL_opt} would still necessitate the evaluation of the posterior. To overcome this problem, variational inference decomposes the log-marginal-likelihood of the observed data as
\begin{equation}\label{decomp}
\ln {p(y)} = \text{ELBO}(q) +   \operatorname{D_{KL}}(q || p(\cdot | y)),
\end{equation}
where we have introduced the $\operatorname{ELBO}$ functional
\begin{equation}
\text{ELBO}(q) = \int q(\theta) \ln \left(\frac{p(y,\theta)}{q(\theta)}\right) d\theta = \int q(\theta) \ln p(y,\theta) d\theta + \mathbb{H}[q]
\end{equation}
where $\mathbb{H}[q] := - \int q(\theta) d\theta$  is the differential entropy of the density $q$. This decomposition leverages the aforementioned non-negativity of the Kullback-Leibler-Divergence, as one can now pursue the following objective for optimization: 
\begin{equation}
q^* = \underset{q  \in \mathcal{Q}}{\arg \max } \ \text{ELBO}(q). 
\end{equation}
The idea behind this replacement is the following: Since for fixed $y$, $\ln {p(y)}$ is fixed, the maximization of the $\operatorname{ELBO}$  corresponds to a minimization of the Kullback-Leibler-Divergence. Note that for the evaluation of the $\operatorname{ELBO}$, the posterior is not required.
Also note that always 
\begin{equation}
\ln {p(y)} \geq \operatorname{ELBO}(q),
\end{equation}
which motivates the nomenclature evidence-lower-bound ($\operatorname{ELBO}$). Furthermore, by maximizing the $\operatorname{ELBO}$, it becomes a closer approximation to the true model evidence $\ln p(y)$. This fact is frequently used to obtain an approximate model comparison criterion.
In this work, however, we will not study this property.

\subsection{Free-Form Mean Field Variational Inference}\label{sec:free-form}

A common choice for the variational family $\mathcal{Q}$ are 'mean-field' families. Here, one first fixes a partition of the parameter $\theta$ into $p$ disjoint subsets. Then, the corresponding mean-field family  consists of all densities that factorize over this  partition, that is 
\begin{equation}
\begin{aligned}
&\theta = (\theta_1,...,\theta_p)\\
&\mathcal{Q} = \left\{q: \mathbb{R}^d \to \mathbb{R}_+ \bigg\vert  q(\theta) =\prod \limits_{j=1}^{p} q_j (\theta_j)\right\}.
\end{aligned} 
\end{equation}
Besides of this assumption, resembling an independence of the respective subsets in the posterior, no additional constraints are imposed on the densities.
In principle, the restriction of the variational density class to the mean-field family already yields  an iterative procedure to approximate the posterior:
Maximization of the $\operatorname{ELBO}$ functional can  be achieved by iteratively maximizing the $\operatorname{ELBO}$ with respect to one density  $q_j(\theta_j)$, while treating the other densities as fixed. 
That is, in every iteration step one seeks to pursue the following coordinate-wise update scheme
\begin{equation}\label{coordinate_wise}
\begin{aligned} 
&\forall j \in \{1,...,p\} \\
&\mathcal{Q}_j = \left\{q_j: \mathbb{R}^{d_j} \to \mathbb{R}_+  \right\} \\
&{q_j}^{*}  = \underset{q_j  \in \mathcal{Q}_j}{\arg \max } \ \text{ELBO}\left(q_j  \prod_{k \neq j} {q_k}\right), \\
\end{aligned}
\end{equation}
where $q_j^{*}$ is the updated optimal variational density of the $j$-th parameter-subset given the fixed densities $\{q_k| k \neq j\}$.
Using variational calculus, it can be shown that the maximization of this objective is obtained by setting 
\begin{equation}
{q_j}^{*} (\theta_j)  \propto \exp\left( \int \prod_{k \neq j} q_k (\theta_{k}) \ln{p(y,\theta)} d \theta_{/j}\right),
\end{equation}
where in the integral $d \theta_{/j}$ refers to integrating over all but the components in the $j$-th set \citep{murphy2012machine}. We will refer to this result as the "fundamental lemma of variational inference" or fundamental lemma for short \citep{ostwald2014tutorial}.
Note that in this approach, although it does not specify the forms of the variational densities beforehand, these forms often still will be restricted to a certain family by the update scheme. As an important example, this is the case in conjugate exponential family models \citep{beal2003variational,murphy2012machine}.
This form, however, although depending on the model and the prior,  
will only "fall out" during analytical treatment of the above procedure - motivating the term "free-form variational inference"  \citep{blei2017}.

\subsection{Fixed-Form Mean-Field Variational Inference}\label{sec:fixed-form}

In practice, the above free-form approach may still not be viable, especially in nonconjugate models where an analytic treatment is hard. Thus it is common to further restrict the mean-field family of densities $\mathcal{Q}$.
Namely, the most typical restriction is to admit only Gaussian densities, that is 
\begin {equation}
\begin{aligned}
\mathcal{Q} = \left\{q(\theta)\bigg\vert q(\theta) = \prod_{j=1}^{p} q_{\mu_j,\Sigma_j}(\theta_j) \right\}\\
q_{\mu_j,\Sigma_j}(\theta_j):= \mathcal{N}(\theta_{j}; \mu_j,\Sigma_j).
\end{aligned}
\end {equation}
One advantage of this approach is that instead of optimizing a functional of probability densities, the problem turns into an optimization problem over the manifold of admissible parameters for Gaussian densities.
That is, in essence it is simplified to a numerical optimization problem - provided, one extracts the dependence of the $\operatorname{ELBO}$ function on the parameters.
As this may still be a non-trivial task, \citet{friston2007variational}  have proposed an approximation to the log-joint probability in the model to give a simple closed-form of the $\operatorname{ELBO}$  that can be numerically optimized.
This scheme, commonly referred to as variational Laplace, iteratively updates one pair of expectation and covariance parameters, while keeping the other parameters fixed. In the following paragraphs, we will lay down the basic ideas of this scheme as it has been detailed by \cite{friston2007variational}. Below, we will focus on the main ideas, while deferring most derivations to \autoref{sec:varlapderivation}.

\subsubsection{Variational Laplace}\label{sec:varlap}

To simplify the following notation, for fixed $y$ we define a map
\begin{equation}
L_y(\theta_1,...,\theta_p) = \ln p(y,\theta_1,...,\theta_p).
\end{equation}
Variational Laplace employs a coordinate-wise update scheme, in a similar spirit as in \eqref{coordinate_wise}, situated in the fixed-form setting, with Gaussian variational densities.
Since these densities are fully determined by their expectation and covariance parameters, the functional $\operatorname{ELBO}(q)$ becomes a function $E(\mu,\Sigma)$ of the variational parameters $\mu = (\mu_1,...,\mu_p) ,\Sigma = (\Sigma_1,...,\Sigma_p)$ and  evaluates to
\begin{equation}
\begin{aligned}
E(\mu,\Sigma) &= \int{L_y(\theta_1,...,\theta_p)  \prod_{j=1}^{p} q_{\mu_j,\Sigma_j}(\theta_j) d\theta_j } +\sum_{j=1}^{p}   \mathbb{H} [q_{\mu_j,\Sigma_j}(\theta_j)] \\
&= \int{ L_y(\theta_1,...,\theta_p)  \prod_{j=1}^{p} q_{\mu_j,\Sigma_j}(\theta_j) d\theta_j } + \frac{1}{2}  \sum_{j=1}^{p}  \ln(\det(\Sigma_j)) + C,\\
\end{aligned}
\end{equation}
where we have used \eqref{Gaussianentropy} to evaluate the differential entropy. Here and in what follows, $C$ refers to some (varying) constant that does not matter for optimization purposes.
Variational Laplace now proceeds as follows:
Assume the parameter set $(\mu_j,\Sigma_j)$ is to be updated, while the parameters $(\mu_{/j},\Sigma_{/j})$ are held fixed. Then, an approximated version of the $\operatorname{ELBO}$ can be derived as
\begin{equation}\label{approxelbo}
\widetilde{E}(\mu_j,\Sigma_j) = L_y(\mu_j,\mu_{/j}) + \frac{1}{2} \sum_{k=1}^{p} \operatorname{tr}\left(\Sigma_k H_{L_y}^{\theta_k}(\mu_j,\mu_{/j})\right) + \frac{1}{2} \sum_{k=1}^p \ln \det(\Sigma_k) + C,
\end{equation}
where
$ H_{L_y}^{\theta_k}(\mu_j,\mu_{/j})$
denotes the quadratic  subpart of the Hessian matrix of the full log-joint probability where the derivatives are only taken with respect to coordinates of the $k$-th parameter, consult \autoref{sec:mvnotation} for details on notation.
This approximated version of the $\operatorname{ELBO}$ essentially relies on a second-order Taylor approximation centered on the variational expectation parameters.
Notably, the above formula yields the advantage that the optimal covariance parameter can be solved in an analytic way. 
Using \eqref{derivtrace} and \eqref{derivlog} to calculate the matrix-derivative of this expression with respect to the $j$-th covariance matrix yields
\begin{equation}
\frac{\widetilde{E}(\mu_j,\Sigma_j) }{\partial \Sigma_j} = \frac{1}{2}  \left(H_{L_y}^{\theta_j}(\mu_j,\mu_{/j})  + \Sigma_j^{-1}\right),
\end{equation}
such that for a fixed variational expectation parameter component $\mu_j$ the optimal covariance matrix in this setting will be given by $-H_{L_y}^{\theta_j}(\mu_j,\mu_{/j}) ^{-1}$. 
It is worth noting that this derivation will only yield a valid covariance matrix if $-H_{L_y}^{\theta_j}(\mu_j,\mu_{/j}) ^{-1}$ is a symmetric, positive definite matrix. Symmetry will typically not be an issue as long as the
log joint probability is twice continuously differentiable. However, the second requirement may not be fulfilled in every model. In the next chapter, we will see how by a slight modification of the Hessians, this may at least partly be achieved in typical applications. 
Hence, in the following, the notation  $H_{L_y}^{\theta_j}(\mu_j,\mu_{/j} )^{-1}$ will either refer to the true Hessian or the respective approximation, such that the expressions remain meaningful. Furthermore, our theoretical results will hold valid in both cases. 

Having derived the above optimal covariance matrix, one could proceed by plugging in this expression into equation \eqref{approxelbo} and optimize it with respect to the variational expectation parameter.  
In the literature, however, it is proposed to determine the optimal expectation parameter with respect to the objective 
\begin{equation}\label{varlap}
\begin{aligned}
\mu_j^* &= \underset{\mu_j} {\arg \max}  \ I_j(\mu_j)\\
I_j(\mu_j) &= L_y(\mu_j,\mu_{/j}) + \frac{1}{2} \sum_{k \neq j}  \operatorname{tr}\left(\Sigma_k H_{L_y}^{\theta_k}(\mu_j,\mu_{/j})\right),
\end{aligned}
\end{equation}
we will discuss possible motivations for this objective in \autoref{sec:varlapderivation} \citep{friston2007variational,daunizeau2017variational,buckley2017free,wipf2009unified}.
This remaining optimization can then be achieved via  standard numerical optimization schemes. For example a Newton-scheme as detailed  in \citet{friston2007variational} can be employed.
In such a scheme, the variational expectation parameter $\mu_j$ is iteratively optimized according to
\begin{equation}
\mu_j^{(i+1)} = \mu_j^{(i)} -  H_{I_j}(\mu_j^{(i)})^{-1} J_{I_j} (\mu_j^{(i)}),
\end{equation}
where  $J_{I_j}$ and $H_{I_j}$  denote the Jacobian and Hessian of $I_j$, the superscript $(i)$ refers to the $i$-th iteration step, and the iterations are carried out for every parameter subset until convergence. Naturally, this scheme can be finessed in terms of its convergence properties. As an example, \citet{ostwald2016probabilistic} provide an algorithm based on a global newton scheme with Hessian modification using additional conditions on the step-sizes in the context of DCM-modelling of ERP-data.

At this point we want to note  that variational Laplace in the case of a single parameter set retrieves what is - although somewhat imprecise - commonly referred to as the Laplace approximation \citep{beal2003variational,kass1995bayes}.
 In this approach, the posterior distribution is approximated by a Gaussian centered on a given maximum-a-posteriori (MAP)  estimate (of the true posterior) with covariance given by the  negative inverse Hessian-matrix evaluated at this estimate.
Thus, if we are using a single parameter set, the objective for the corresponding variational parameter $\mu$ is to find
\begin{equation}
\mu^* = \underset{\mu} {\arg \max} \ I(\mu) =  \underset{\mu} {\arg \max} \   \ln p(y,\mu) =  \underset{\mu} {\arg \max} \   \ln p(\mu|y).
\end{equation} 
Furthermore, the covariance parameter $\Sigma$ is then analytically set to 
\begin{equation}
\Sigma^*  = -\frac{\partial^2}{\partial \theta \partial \theta^T} \ln p(y, \theta) \bigg\rvert_{\theta = {\mu^*}}^{-1} =  -\frac{\partial^2}{\partial \theta \partial \theta^T}\ln p(\theta|y) \bigg\rvert_{\theta = {\mu^*}}^{-1} = -H_{L_y}(\mu^*)^{-1},
\end{equation}
such that the resulting optimal density $q_{\mu^*,\Sigma^*}(\theta)$ corresponds to the  Laplace approximation.

\section {Nonlinear Transform Models}\label{sec:model}

In the following, we want to give an important example for the preceding scheme. To this end, consider the model
\begin{equation}
\begin{aligned}
y &= f(\theta) + e \\
e &\sim \mathcal{N}(0, Q(\lambda)).
\end{aligned}
\end{equation}
This is a transform model with two parameters: The data $y$ are interpreted as the (possibly nonlinear) transform $f$ of some hidden parameter $\theta$ under the addition of Gaussian noise. Furthermore, the covariance structure of this noise is determined by another hidden parameter $\lambda$.
We can embed this in a joint model of the form
\begin{equation}
\begin{aligned}
p(y,\theta,\lambda) &= p(y|\theta,\lambda) p(\theta)p(\lambda) \\
p(y|\theta,\lambda) &=  \mathcal{N}(y; f(\theta), Q(\lambda)) \\
p(\theta) &= \mathcal{N}(\theta;m_\theta,S_\theta) \\
p(\lambda) &= \mathcal{N}(\lambda;m_\lambda,S_\lambda).\\
\end{aligned}
\end{equation}
Here, we have assumed Gaussian priors on the parameters $\theta,\lambda$, but other forms of the prior are equally possible. In principle, one could apply free-form variational inference to such models, but as we briefly discuss in \autoref{sec:varlapderivation_model}, this may not yield significant facilitation of inference. Thus, the use of an approximate scheme like variational Laplace may be motivated. Indeed, the above model, in its general form, comprises many of the applications of variational Laplace that may be found in the literature \citep{friston2008multiple,friston2008Bayesian,friston2007variational,daunizeau2009variational}.
In particular, this model contains the important special case of a GLM with a nonspherical covariance matrix parameterized as the linear combination of a given set of known covariance basis matrizes  \citep{starke2017,friston2008multiple}.

To apply variational Laplace to the above model, one assumes Gaussian variational densities $q(\lambda) = \mathcal{N}(\lambda;\mu_\lambda,\Sigma_\lambda), q(\theta) = \mathcal{N}(\theta;\mu_\theta,\Sigma_\theta)$ and updates the variational  parameter according to the scheme discussed in the last section. Again, we will only state the update equations, for their derivation consult  \autoref{sec:varlapderivation_model}.
Using the notation of the previous exposition, the variational Laplace update scheme for the variational expectation parameter $\mu_\theta$ with given $\mu_\lambda^*,\Sigma_\lambda^*$ becomes
\begin{equation}
\mu_\theta^*  = \underset{\mu_\theta}{\arg \max} \  I(\mu_\theta) \\,
\end{equation}
where
\begin{equation}\label{varlaptheta}
\begin{aligned}
I(\mu_\theta) &=  L_y(\mu_ \theta,\mu_\lambda^*) + \frac{1}{2}   \operatorname{tr}\left(\Sigma_\lambda^* H_{L_y}^\lambda \left(\mu_\theta,\mu_\lambda^*\right)\right) \\
 = &-\frac{1}{2} \left(y-f\left(\mu_\theta\right)\right)^T Q\left(\mu_\lambda^*\right)^{-1} \left(y-f\left(\mu_\theta\right)\right) \\
		    & -\frac{1}{2} \left(\mu_\theta - m_\theta\right)^T S_\theta^{-1}\left(\mu_\theta - m_\theta\right) \\
		    & + \frac{1}{2} \operatorname{tr} \left(\Sigma_\lambda^*   \frac{\partial^2}{\partial \lambda \partial \lambda^T} \left(-\frac{1}{2} \left(y-f\left(\mu_\theta\right)\right)^T Q(\lambda)^{-1}\left(y-f\left(\mu_\theta\right)\right)\right)  \bigg\rvert_{\lambda = \mu_\lambda^*}\right) +C,
\end{aligned}
\end{equation}
and we have ignored terms not depending on $\mu_\theta$  as these will not affect the optimization of $I(\mu_\theta)$. After an optimal parameter is found according to this objective, the covariance parameter is set analytically to 
\begin{equation}
\Sigma_\theta^* = - H_{L_y}^\theta(\mu_\theta^*,\mu_\lambda^*)^{-1} \approx \left(J_f\left(\mu_\theta^*\right)^T Q\left(\mu_\lambda^*\right)^{-1} J_f\left(\mu_\theta^*\right)  + S_\theta^{-1}\right)^{-1}.
\end{equation}
This approximated version of the Hessian is obtained by dropping dependencies on higher derivatives of $f$, as is advocated in \citet{friston2007variational} in light of the previously discussed requirement of the covariance matrix being positive definite.
Similarly, the update scheme for $ \mu_\lambda$ is given by
\begin{equation}
\mu_\lambda^*  = \underset{\mu_\lambda}{\arg \max} \  I(\mu_\lambda) \\,
\end{equation}
with
\begin{equation}
\begin{aligned}
I(\mu_\lambda) &=  L_y(\mu_ \theta^*,\mu_\lambda) + \frac{1}{2}   \operatorname{tr}\left(\Sigma_\theta^* H_{L_y}^\theta \left(\mu_\theta^*,\mu_\lambda\right)\right) \\
= &-\frac{1}{2} \left(y-f\left(\mu_\theta^*\right)\right)^T Q\left(\mu_\lambda\right)^{-1} \left(y-f\left(\mu_\theta^*\right)\right) \\
&-\frac{1}{2} \ln \det Q(\mu_\lambda) \\
		    & -\frac{1}{2} \left(\mu_\lambda - m_\lambda\right)^T S_\lambda^{-1}\left(\mu_\lambda - m_\lambda\right) \\
		    & - \frac{1}{2} \operatorname{tr} \left(\Sigma_\theta^*   \left(J_f\left(\mu_\theta^*\right)\right)^T Q\left(\mu_\lambda\right)^{-1} J_f\left(\mu_\theta^*\right)  + S_\theta^{-1}\right) + C,
\end{aligned}
\end{equation}
where the Hessian with respect to $\theta$ was again approximated by dropping second derivatives of $f$.
As before, after identifying the optimal $\mu_\lambda^*$, the covariance parameter is set to
\begin{equation}
\begin{aligned}
\Sigma_\lambda^* &= - H_{L_y}^\lambda(\mu_\theta^*,\mu_\lambda^*)^{-1}  \\&= -\left( \frac{\partial^2}{\partial \lambda \partial \lambda^T} \left(-\frac{1}{2} \left(y-f\left(\mu_\theta^*\right)\right)^T Q\left(\lambda\right)^{-1}
\left(y-f\left(\mu_\theta^*\right)\right) -\frac{1}{2} \ln \det Q(\lambda)\right) \bigg\rvert_{\lambda = \mu_\lambda^*} - S_\lambda^{-1}\right)^{-1}.
\end{aligned}
\end{equation}

The above set of update equations fully describes the variational Laplace scheme. However, we note that in the literature there are examples in which this scheme may be slightly varied, especially in the case that the transform function $f$ is actually linear \citep{starke2017,friston2008Bayesian,friston2008multiple}. In \autoref{sec:disambiguation}, we give a disambiguation on some of the variations this approach naturally provides and also discuss related methods \citep{chappell2008variational}.
Achieving the maximization of the functions $I(\mu_\theta),I(\mu_\lambda)$ commonly referred to as "variational energies" may be achieved by a Newton scheme, as we saw in the last section. The first and second derivatives needed for these updates may be found in \citet{friston2007variational}.

\section {Model Examples}

To simplify the following discussions, we want to examine two special cases of the above model in more detail. Due to their simple structure, these models will facilitate the analytic treatment of variational Laplace in the upcoming sections. 
\subsection{Single Parameter Model}\label{sec:singleparam}
For our first model we consider the case where there is just one unknown, real valued parameter $\theta$. That is, we assume there is a known, possibly nonlinear, transform function \begin{equation}f: \mathbb{R} \to \mathbb{R}\end{equation} that
conveys the mean of a sample from a normal distribution with known variance. Thus, the model can be stated as 
\begin{equation}
\begin{aligned}
& \\
&p(y|\theta) = \mathcal{N}(y;f(\theta) \mathbf {1}_n,s^2 I_n) = \prod_{i=1}^{n} \mathcal{N}(y_i;f(\theta),s^2) \\
&p(\theta)  =  \mathcal{N}(\theta; m_\theta, s_\theta^2).
\end{aligned}
\end{equation}
In this case, a mean-field assumption is unnecessary: We approximate the posterior by a single variational density 
\begin{equation}
p(\theta|y)  \approx q_{\mu_\theta,{\sigma_\theta}}(\theta) = \mathcal{N}(\theta; m_\theta,{s_\theta}^2).
\end{equation}
As  noted before, this is a special case of variational Laplace where it corresponds to the "classical" Laplace approximation. In the notation of the previous section, we here have
\begin{equation}
I(\mu_\theta) = L_y(\theta) =  \ln p(y,\mu_\theta).
\end{equation}
In particular, the variational mean $\mu_\theta$ coincides with a MAP estimate. Maximization of the posterior can then for example be achieved via a Newton scheme. The derivatives that are necessary to this end, after dropping higher-order derivatives of $f$, read
\begin{equation}
\begin{aligned}
&\frac{\partial}{\partial \theta} L_y(\theta)  =  \frac{m_\theta-\theta}{s_\theta^2} - \frac{n}{s^2} \left(f'(\theta)(f(\theta) - \overline{y})\right) \\
& \frac{\partial}{\partial^2 \theta} L_y(\theta)  \approx \frac{-(n f'(\theta)^2s_\theta^2 + s^2)}{s^2 s_\theta^2}, 
\end{aligned}
\end{equation}
with $\overline{y} = \frac{1}{n}\sum\limits_{i=1}^n y_i$. Hence, the approximated update scheme for the variational parameters in a Newton scheme, where $\mu_\theta$ is given from a previous iteration, can be summed up as
\begin{equation}
\begin{aligned}
&\mu_\theta^* = \mu_\theta +\frac{s^2 s_\theta^2}{n f'( \mu_\theta)^2 s_\theta^2 + s^2}\left(\frac{m_\theta-\mu_\theta}{s_\theta^2} - \frac{n}{s^2} (f'(\mu_\theta)(f(\mu_\theta) - \overline{y})\right) \\
& \sigma_\theta^* = \frac{s^2 s_\theta^2}{n f'( \mu_\theta^*)^2 s_\theta^2 + s^2}.
\end{aligned}
\end{equation}
In our treatment of this model, we will assume that the above scheme is indeed successful, in the sense that a maximum of $\ln p(y,\theta)$ is correctly identified.

\subsection{Linear Transform Model with Unknown Noise Parameter}\label{sec:linear}

The second example will serve to give an intuition on how variational Laplace will differ from MAP if there is more than one parameter, and it will also be possible to 
study the asymptotics in this example. 
The model is given by
\begin{equation}
\begin{aligned}
p(y_i|\theta,\lambda)  &= \mathcal{N}(y_i; a\theta, \exp(\lambda)) \\
p(\theta) &= \mathcal{N}(\theta;m_\theta,s_\theta^2) \\
p(\lambda) &= \mathcal{N}(\lambda;m_\lambda,s_\lambda^2).\\ 
\end{aligned}
\end{equation}
Here, the transform is simply a linear function of  $\theta$, and a second parameter $\lambda$ parameterizing the variance is included. Although naturally this model could be treated in a simpler way and the $\operatorname{ELBO}$ could be evaluated directly,
we want to apply the update equations of variational Laplace according to the scheme described previously. To make the scheme amenable to further analysis, we derive the updates in closed-form in \autoref{sec:linearupdates}.
In one iteration step, with $(\mu_\theta,\sigma_\theta,\mu_\lambda,\sigma_\lambda)$ given by the previous iteration, these update equations read
\begin{equation}\label{updatelinear}
\begin{aligned}
&\mu_\theta^* =  \frac{   a (1+ \frac{\sigma_\lambda}{2})  \frac{\sum_{i=1}^{n} y_i}{n}  + \frac{\exp(\mu_\lambda)m_\theta}{n s_\theta^2}}{ a^2   (1+ \frac{\sigma_\lambda}{2}) +\frac{\exp(\mu_\lambda)}{n s_\theta^2}}\\
&\sigma_\theta^* =\frac{\exp(\mu_\lambda)s_\theta^2}{a^2ns_\theta^2 + \exp(\mu_\lambda)} \\
& \mu_\lambda^* = W\left( n\frac{s_\lambda^2}{2} \left( \frac{\sum_{i=1}^{n} (y_i - a \mu_\theta^*)^2}{n}+a^2  \sigma_\theta^*\right) \exp\left( n \frac{s_\lambda^2}{2} - m_\lambda \right) \right) -  ( n \frac{s_\lambda^2}{2} - m_\lambda)\\
& \sigma_\lambda^* =  \frac{2 \exp(\mu_\lambda^*)s_\lambda^2}{\sum_{i=1}^{n}(y_i -a\theta)^2 s_\lambda^2 + 2 \exp(\mu_\lambda^*)},
\end{aligned}
\end{equation}
where $W$ is Lamberts $W$-function that inverts the function $x \mapsto x \ \exp(x)$ on the nonnegative real line, consult \autoref{sec:lambertw} for further details on this function.

\section{Frequentist Asymptotics}
\subsection {Point Estimators}
In the following, we will define quality criteria for point estimators in the frequentist sense \citep{casella2002statistical,ruschendorf2014mathematische}. To do so, we will start with a brief intuition on frequentist point estimation. In a typical practical scenario, given some samples (observed data), 
one would like to make inference about an underlying "true" distribution the samples have been drawn from. 
Often, one has a vague notion of the general form that this distribution should have, and is left to specify a set of parameters such that the distribution matches the data best in some sense. Thus, one has to decide how to determine this unknown parameter, that is, which estimator to use.  In  simple situations, there is often a natural choice for the "best guess" - consider for example the empirical mean of a sample from a Gaussian distribution as an estimator for the mean parameter of the underlying true distribution.
In more contrived situations, however, it may not be immediately clear which estimator should be used. Quality criteria of point estimators  provide the possibility to judge and compare different estimators and serve as theoretical justification of their use.

To formalize these notions, we assume that we observe the realization of a random variable $y$, and, for our concerns, we make the additional assumption that $y$ is made up of $n$  independent, identically distributed random variables, such that $y = (y_1,...y_n)$.
Furthermore, we assume that this distribution is fixed up to a parameter $ \theta \in \mathbb{R}^d$, such that all admissible probability measures are in the family
\begin{equation}
\mathcal{P} = \{\mathbb{P}_\theta | \theta \in \Theta \subset \mathbb{R}^d\}.
\end{equation}
We further assume that all $\mathbb{P}_\theta$ admit densities $p(y_1|\theta)$ and that there always is one true, unknown parameter value $\theta_0 \in \Theta$, that indeed correctly 
specifies the distribution of the data.  In the following, if not stated explicitly, we will omit the subscript and will always refer to the probability measure under the true parameter, that is $\mathbb{P} := \mathbb{P}_{\theta_0}$. 
One classical goal of frequentist point estimation in this situation is to estimate the true parameter.
That is, if we assume $y_1,y_2,...y_n \overset{i.i.d.}{\sim} p(y_1|\theta_0)$,
we seek to find a (deterministic) function 
\begin{equation}
\begin{aligned}
\hat{\theta}_n:  \mathbb{R}^m \times \mathbb{R}^m ... \times \mathbb{R}^m &\to \mathbb{R}^d \\
\end{aligned}
\end{equation}
which may depend on $n$, such that when we apply this function to $y$, $\hat{\theta}_n(y)$ is an estimator for the true parameter $\theta_0$.

To assess the quality of such estimators, various criteria have been studied. Classical criteria like unbiasedness and efficiency characterize the behaviour of an estimator in a finite setting \citep{casella2002statistical}.
These criteria will not be useful for our purposes, since, as we explain in \autoref{sec:bayescriteria}, point estimators in the Bayesian setting are typically motivated by a different set of quality criteria than frequentist estimators \citep{murphy2012machine,ruschendorf2014mathematische, van2000asymptotic}. 
Furthermore, in some situations, these finite sample criteria may result in "strange" optimal estimators or not yield comparability of two different estimators \citep{van2000asymptotic}.   
Instead, we want to concentrate on properties that are motivated by asymptotic notions.

\subsection{Asymptotic Properties of Point Estimators}\label{sec:freqcriteria}

Asymptotic statistics is concerned with limiting properties in statistics, typically this can be thought of as the "large sample" limit - the number of datapoints $n$ going to infinity. The interest in asymptotic properties in statistics arises for at least two reasons \citep{van2000asymptotic}:
First, many statistical estimators and tests that are difficult to analyze in a finite setting - for example, because their distribution does not take a manageable closed-form - behave like well-understood distributions with growing data size. Acclaimed limit results like the central limit theorem may be used to derive normality in many situations - giving statisticians the possibility to derive useful approximations in the limit. Secondly, as stated before, quality properties defined in the finite-sample setting may not yield an useful decision regarding an estimator. In this situation, asymptotic properties may still allow conclusions on the quality of an estimator, and they do so for a wider class of models.

To understand the following statements, recall the modes of stochastic convergence: Let $x^{(n)}$ be a sequence of random variables and $x$ be a fixed random variable. Then, $x^{(n)}$ converges to $x$ 
\begin{enumerate}
\item in distribution, if $\lim\limits_{n \to \infty} \mathbb{E}[ f(x^{(n)})] = \mathbb{E}[ f (x)]$, for all continuous, bounded $f$. Notation: $x^{(n)} \overset{d}{\longrightarrow} x$

\item in $\mathbb{P}$-Probability, if $\mathbb{P}\left[ \left|x^{(n)}-x\right| <\varepsilon\right] \to 0, \ \forall \ \varepsilon > 0$ . Notation: $x^{(n)} \overset{\mathbb{P}}{\longrightarrow} x$

\item almost surely (under $\mathbb{P}$), if $\mathbb{P}\left[ \underset{n \longrightarrow \infty}{\lim} x^{(n)} \neq x\right] = 0$ . Notation: $x^{(n)} \overset{\text{a.s.}}{\longrightarrow}x$

\end{enumerate} 
Below, the sequences of random variables in question will be of the form $\hat{\theta}_n(y_1,...,y_n)$ for a sequence of deterministic functions $\hat{\theta}_n$. For notational clarity, the dependence on the data is typically omitted. We thus write $\hat{\theta}^{(n)} =  \hat{\theta}_n(y_1,...,y_n)$, and we will further on simply refer to such a sequence of random variables as a sequence of point estimators.

With the previous definitions in mind, we are ready to state asymptotic quality criteria for point estimators.
Note that all upcoming definitions should be understood to hold regardless of the particular
value of the true parameter.
A sequence of estimators $\hat{\theta}^{(n)}$ is said to be asymptotically consistent if it converges in probability to the true parameter, that is
\begin{equation}
\hat{\theta}^{(n)} \overset{\mathbb{P}}{\longrightarrow} \theta_0.
\end{equation}
This criterion intuitively tells us, that as the sample size grows larger, our estimator will return the true parameter with arbitrarily high probability. In many situations, it can be proved that there exists a consistent estimator. Furthermore, under typical regularity conditions, the maximum-likelihood-estimator (MLE)
$$ \hat{\theta}^{(n)} = \underset{\theta \in \Theta}{\arg \max} \  p(y|\theta) $$
 is consistent. This property of the MLE is often used to prove the consistency of other estimators, which will also be relevant in the case of variational inference \citep{wang2019frequentist}. 
Consistency of an estimator establishes the true parameter as the asymptotic limit (in probability) of an estimator. However, it may be beneficial to also characterize the distributional shape these estimators will take in the limit. To obtain these distributions, the sequence of estimators needs to be rescaled and centered appropriately. In our setting, this appropriate rescaling will be given by $\sqrt{n}$, resembling classical results like the central limit theorem. We say that a sequence of estimators $\hat{\theta}^{(n)}$ is  asymptotically normal if for some $\Sigma \in \mathbb{R}^{d\times d} $ 
\begin {equation}
\begin{aligned}
&\sqrt{n}(\hat{\theta}^{(n)}-\theta_0) \overset{d}{\longrightarrow} G \sim \mathcal{N}(0,\Sigma),
\end{aligned}
\end{equation}
that is the sequence of estimators, after rescaling and centering them on the true parameter, converges in distribution to a normal distribution. This property of an estimator is, among other reasons, desirable, as it allows the construction of asymptotically valid confidence intervals based on the normal distribution, 
which is particularly convenient as this is among the best understood distributions. \citep{van2000asymptotic}.
In many situations, one can give a lower bound to the asymptotic covariance matrix: 
The Fisher information matrix $\mathcal{I}$ of a differentiable model $p(y|\theta)$ is given by 
 \begin{equation}
\mathcal{I}(\theta_0) := \mathbb{E}_{\theta_0}\left[\frac{\partial \ln p(y|\theta)}{\partial \theta}\left(\frac{\partial \ln p(y|\theta)}{\partial \theta}\right)^T(\theta_0) \right],\end{equation}
provided that the right hand side exists.
If the model is twice differentiable and we can exchange the order of integration and differentiation, that is
\begin{equation}\label{fisherconditions}
\begin{aligned}
 \frac{\partial}{\partial \theta} \int p(y|\theta) dy &= \int \frac{\partial}{\partial \theta} p(y|\theta)dy \\
  \frac{\partial^2}{\partial \theta \partial \theta^T}    \int  p(y|\theta) dy &= \frac{\partial}{\partial \theta}   \int  \frac{\partial}{\partial \theta} p(y|\theta) dy,
\end{aligned}
\end{equation}
then the Fisher information matrix can also be expressed as \citep{shao2003mathematical}
\begin{equation}
\mathcal{I}(\theta_0) = -\mathbb{E}_{\theta_0}\left[\frac{\partial^2 \ln p(y|\theta)}{\partial \theta \partial \theta^T}(\theta_0)\right].
\end{equation}
One reason for the crucial importance of this quantity in asymptotic statistics is that under regularity conditions, one has that the covariance matrix of the limiting distribution for a consistent, asymptotically normal estimator fulfills
\begin{equation} \Sigma \geq \mathcal{I}(\theta_0)^{-1},\end{equation}
where the inequality is to be understood in the sense of the difference being positive definite \citep{van2000asymptotic}.
This fact, together with results assuring that a normal distribution is in many ways the "best" limit distribution for a sequence of estimators \citep{van2000asymptotic}, gives rise to the notion of asymptotic efficiency:
A sequence of estimators  $(\hat{\theta}^{(n)})$ is called asymptotically efficient, if 
\begin{equation}\label{efficiency}
\begin{aligned}
 &\sqrt{n} \left(\hat{\theta}^{(n)} - \theta \right) \overset{d}{\longrightarrow}  G \sim \mathcal{N}(0, \mathcal{I}(\theta_0)^{-1}).
\end{aligned}
\end{equation}
Importantly, a MLE will usually be asymptotically efficient - that is, we can expect that there exists at least one estimator which is optimal in the sense of this definition. Thus it is desirable for estimators to have asymptotic variance close to this so called asymptotic Cramer-Rao bound.
Furthermore, we will see that the Fisher information matrix also plays a role in the asymptotic behaviour of Bayesian posterior distributions.

\subsection {Frequentist Treatment of Bayesian Quantities}\label{sec:bayescriteria}
As stated before, variational inference is situated in a Bayesian setting, such that the parameters $\theta$ are embedded in a probabilistic model. In the Bayesian framework, point estimators for such parameters are typically justified by minimization of posterior loss \citep{murphy2012machine}.
That is, one specifies a loss-function 
 \begin{equation}
\begin{aligned}
l : \Theta \times \Theta \to \mathbb{R};(\hat{\theta},\theta) \mapsto l(\hat{\theta},\theta),
\end{aligned}
\end{equation}
such that $l$ penalizes the deviation of an estimate $\hat{\theta}$ from $\theta$ in an adequate way for the problem. A Bayes-estimator $\hat{\theta}^*(y)$   with respect to the loss-function $l$ then solves
\begin{equation}
\hat{\theta}^*(y) =  \underset{\hat{\theta} \in \Theta}{\arg \min} \   \int l(\hat{\theta},\theta) p(\theta|y) d\theta .
\end{equation}
Hence, the estimate minimizes the loss that is expected under the posterior distribution, given the data $y$. If we choose quadratic loss for $l$, that is $l(\hat{\theta},\theta) = (\hat{\theta}-\theta)^2$, then the unique minimizer of the posterior loss is the posterior mean: ${\hat{\theta}^*(y) = \int \theta p(\theta|y) d\theta =: \overline{\mu}(y)}$.  This result motivates a natural choice for point estimation in variational inference. Assume we have derived an optimal variational density $q^* \in \mathcal{Q}$, such that it minimizes the Kullback-Leibler divergence to the posterior. 
The point estimate of choice would then correspondingly be the variational mean, that is $\hat{\theta}^*(y) = \int \theta q^*(\theta) d\theta =: \mu(y)$. 

If we now want to assess the quality of variational inference in terms of frequentist point estimator quality, it is necessary to assume a true parameter $\theta_0$ that 
governs the underlying distribution of the data.  Note that this is not contradicting the Bayesian point of view.
Assume that the data indeed comes from a true distribution $p(y|\theta_0)$.
The Bayesian procedure then reflects our uncertainty on the true parameter $\theta_0$, rather than an actual "true randomness" in the parameter. 
This notion makes the Bayesian point estimates amenable to frequentist analysis. For example, in the light of the preceding discussion, we could readily analyze the theoretical properties of the sequence of point estimators
\begin{equation}
\begin{split}
\overline{\mu}^{(n)} = \int \theta p(\theta|y_1,...,y_n) d\theta \\ y_i  \overset{\text{i.i.d.}}{\sim} p(y_1|\theta_0).
\end{split}
\end{equation}

We want to conclude this section by mentioning two important results regarding Bayesian asymptotics. When analyzing Bayesian inference from the frequentist perspective, one is not restricted to point estimators. 
As the main quantity of interest in Bayesian statistics would usually be the full posterior, rather than a point estimate, it is of interest to analyze the behaviour of this distribution in the large sample limit. 
The following theorems concern the asymptotics of this distribution \citep{van2000asymptotic}.
Formally, one analyses the posterior measure $\Pi(\cdot|y_1,...,y_n)$, which is random as it depends on the random variables $y_1,...,y_n$.
Doob's consistency theorem then assures that this posterior measure will collapse to a point mass at the true parameter, with growing $n$. It states that under any prior for almost every $\theta_0$,
the posterior $\Pi(\cdot|y_1,...,y_n)$ converges in distribution to $\delta_{\theta_0}$ under $\mathbb{P}$. 
The Bernstein-von-Mises theorem goes even further and gives an intuition on the asymptotic shape of the posterior(under appropriate rescaling). Roughly, it states that the posterior will converge to a normal distribution in the limit.
To make this statement more precise, a distance measure for this sort of convergence is needed. The total variation distance between the posterior measure and another measure $Q$ is defined as
\begin{equation}
\left\Vert \Pi(\cdot|y_1,...,y_n) - Q(\cdot)\right\Vert_{TV} = \underset{B}{\sup} \left|\Pi(B|y_1,...,y_n) - Q(B)\right|
\end{equation}
In the case that both are associated with continuous densities $p(\theta|y), q(\theta)$, this can equally be expressed as
\begin{equation}
\left\Vert \Pi(\cdot|y_1,...,y_n) - Q(\cdot) \right\Vert_{TV} = \frac{1}{2} \int \left| p(\theta|y_1,...,y_n) - q(\theta)\right| d\theta .
\end{equation}
Since we will only be dealing with such densities here, we will notationally identify the probability measures and the respective densities, which will allow a more concise presentation of results and is in line with current literature \citep{wang2019frequentist}.
Under regularity conditions, the Bernstein-von-Mises theorem now states that
\begin{equation}\label{bernstein_von_mises}
\left\Vert p(\theta|y_1,...,y_n) -\mathcal{N}(\theta; \hat{\theta}^{(n)}, \frac{1}{n} \mathcal{I}(\theta_0)^{-1})\right\Vert_{TV} \overset{\mathbb{P}}{\longrightarrow} 0 ,
\end{equation}
where $\hat{\theta}^{(n)}$ is a maximum-likelihood-estimator for the true parameter. 
Another formulation of the Bernstein-von-Mises theorem is
\begin{equation}
\left\Vert \widetilde{p}_h(h|y_1,...,y_n) - \mathcal{N}\left(h;\sqrt{n}\left(\hat{\theta}^{(n)}-\theta_0\right),  \mathcal{I}(\theta_0)^{-1}\right)\right\Vert_{TV} \overset{\mathbb{P}}{\longrightarrow} 0 ,
\end{equation}
where $\widetilde{p}_h$ corresponds to the posterior of the rescaled, centered parameter $h = \sqrt{n}(\theta-\theta_0)$ \citep{van2000asymptotic}. Note how this last formulation resembles equation \eqref{efficiency}.
The Bernstein-von-Mises theorem  may give an intuition on the asymptotics of variational inference: If the posterior in the limit behaves like a normal distribution, we would expect that under some conditions the best approximating variational distribtuion would also be normal.
We will elaborate further on this in the next section. Importantly, this limiting behaviour also transfers to point estimates that are derived from the posterior distribution, as long as the respective loss function is well-behaved. This includes as a special case the posterior mean. Indeed, we have that 
\begin{equation}
\sqrt{n} (\overline{\mu}^{(n)}-\theta_0) \overset{d} {\longrightarrow} \mathcal{N}(0,\mathcal{I}(\theta_0)^{-1}),
\end{equation}
which also implies that the posterior mean is consistent \citep{van2000asymptotic}.

\section{Frequentist Properties of Variational Inference}\label{sec:review}

Although the development of variational inference in the machine learning community has been taking place for quite some time \citep{hinton1993keeping,anderson1987mean}, research on its theoretical properties has only recently gained more attention.
Early papers could show some positive as well as negative results, but were often restricted to certain models only: \cite{wang2004convergence} showed that for exponential family models with missing values, the variational estimates converge to the true value as long as the starting point used for the algorithm is sufficiently close to the true value. Furthermore they proved the asymptotic normality of these estimates. Later, similar results for normal mixture models have been achieved \citep{wang2006convergence}. \cite{hall2011theory} establish consistency and give convergence rates for a Poisson mixture model. \cite{celisse2012consistency} could show consistency of variational estimators in graph stochastic block models. On the other hand, \cite{wang2004lack} derived that in state space models, variational estimates are not necessarily consistent. \cite{wang2005inadequacy} furthermore noticed that variational inference will often underestimate the variance, in the sense that the covariance of the variational distribution is too narrow compared to the corresponding asymptotic covariance of the posterior - that is, the inverse Fisher information.
More recently, \cite{westling2015establishing} proved consistency of variational inference in the setting of i.i.d. samples with local latent variables by means of identifying it as a special case of profiled M-estimation.

The most general work on the asymptotic properties of variational inference has, to our knowledge, been presented in  \cite{wang2019frequentist}. This work gives quite general conditions for the consistency of variational inference. The consistency results in this paper concern the variational point estimates as well as the asymptotics of the posterior distribution in the sense of total variation distance. They hence confirm the intuition that we have briefly discussed in the previous section: As the posterior distribution will asymptotically be approaching a normal distribution, the best approximating variational distributions will do so as well. Specifically, they will converge to the member of the variational family that best approximates the limit distribution for the true posterior, that is the Gaussian given in equation \eqref{bernstein_von_mises}. This furthermore gives a general proof of the underdispersion of variational densities obtained from mean-field-families: As long as the true posterior does not factorize as the variational densities do, the respective best approximating density will underestimate the true posterior variance. Since these results give a fairly general notion of the asymptotic properties of variational inference, we want to discuss them to some extent. The main conditions needed to establish these results are detailed in \autoref{sec:bleiresults}; notably they typically hold for i.i.d. observations in regular parametric models, we will restrict the following presentation of the results to this special case.
In the following, consider a sequence of optimal variational densities
\begin{equation}
q^{*(n)} = \underset{q \in \mathcal{Q}^d}{\arg \min}    \operatorname{D_{KL}}\left(q(\theta)||p(\theta|y_1,...,y_n)\right).
\end{equation}
The first result gives an equivalent of Doob's consistency theorem:
\begin{equation}
q^{*(n)} \overset{d}{\longrightarrow} \delta_{\theta_0}   \ \mathbb{P} - a.s.
\end{equation}
Hence, as is the case for the true posterior, the variational distributions will eventually collapse to point masses centered at the true parameter. 
The next result is an equivalent to the Bernstein-von-Mises theorem. For this, we need to define the closest member in the variational family to the limiting normal distributions:
\begin{equation}\label{optimalGaussian}
g^{(n)} = \underset{q \in \mathcal{Q}^d}{\arg \min}   \operatorname{D_{KL}}\left(q(\theta)|| \mathcal{N}(\theta; \hat{\theta}^{(n)}, \frac{1}{n} \mathcal{I}(\theta_0)^{-1})\right)
\end{equation}
We then get that the optimal variational density converges in total variation distance to $g^{(n)}$:
\begin{equation}\label{variationalbernstein}
\left\Vert q^{*(n)} (\theta) - g^{(n)}(\theta)\right\Vert_{TV} \overset{\mathbb{P}}{\longrightarrow}  0.
\end{equation}
Analogous to the Bernstein-von-Mises theorem, this result may also be stated as
\begin{equation}\label{rescaledg}
\left\Vert \widetilde{q}_h^{*(n)} (h) - \widetilde{g}_h^{(n)}(h)\right\Vert_{TV} \overset{\mathbb{P}}{\longrightarrow} 0,
\end{equation}
again with the reparameterization $h = \sqrt{n}(\theta-\theta_0)$.
Note that the approximating density $g^{(n)}$ will give an exact match if and only if the variational family contains the respective Gaussian. 
If we do not include all Gaussian distributions in our variational family, then we cannot expect the variational densities to retrieve the correct precision in the above sense. For example, when the variational family due to the mean-field assumption does assume independence between components that have nonnegative entries in the inverse Fisher information matrix $\mathcal{I}(\theta_0)^{-1}$ ,
one can see that $g^{(n)}$ underestimates the covariance in terms of differential entropy. To be precise, this means that 
\begin{equation}\label{underdisp}
\mathbb{H}\left[g^{(n)}\right] \leq \mathbb{H}\left[\mathcal{N}(\theta; \hat{\theta}^{(n)}, \frac{1}{n} \mathcal{I}(\theta_0)^{-1})\right].
\end{equation}
We show this in an example where we assume a fully-factorized Gaussian variational family in \autoref{sec:entropy}. Note that in this situation the mean will still be matched exactly by $g^{(n)}$.
As one could guess by now, for the point estimates obtained by the variational densities we correspondingly get consistency and asymptotic normality.
Precisely, the variational mean is consistent for the true parameter:
\begin{equation}
\mu^{(n)} \overset{\mathbb{P}}{\longrightarrow} \theta_0
\end{equation} 
Furthermore, we have that the rescaled variational mean converges to the rescaled MLE, which gives us the asymptotic normality (and efficiency) of the former:
\begin{equation} \begin{aligned}
&\sqrt{n}(\mu^{(n)}  - \theta_0) - \sqrt{n}(\hat{\theta}^{(n)} -\theta_0)   \overset{\mathbb{P}}{\longrightarrow} 0 \\
&\implies \sqrt{n}(\mu^{(n)}  - \theta_0) \overset{d}{\longrightarrow} \mathcal{N}(0,\mathcal{I}(\theta_0)^{-1})
\end{aligned}\end{equation} 
Thus we have seen, that under some regularity conditions, variational inference yields asymptotically consistent and efficient point estimators. This result justifies the widespread use of such procedures from a theoretical point of view at least for models that fulfill the assumptions.

\subsection{Applicability to Variational Laplace}

In the previous paragraph we have seen that under some general conditions, variational inference will be asymptotically consistent, as long as the variational family $\mathcal{Q}$ contains Gaussians. The results presented, however, rely on the crucial assumption that 
$$q^*(\theta) = \underset{q \in \mathcal{Q}}{\arg \min}    \operatorname{D_{KL}} \left(q(\theta)||p(\theta|y)\right),$$
in other words, the optimal variational density $q^*(\theta)$ that minimizes the Kullback-Leibler-divergence has correctly been identified.  Due to the approximations that are made to the objective, this assumption will in general not be justified in variational Laplace. 
That is, one will typically have that 
$$ \mathcal{N}(\theta;\mu^*,\Sigma^*) \neq \underset{\mu,\Sigma}{\arg \min}    \operatorname{D_{KL}} \left(\mathcal{N}(\theta;\mu,\Sigma)||p(\theta|y)\right),$$
where now $\mu^*,\Sigma^*$ refer to the optimal parameters that are obtained via the variational Laplace scheme \eqref{varlap}. To see this, it is enough to consider the special case of one parameter. As we have discussed before, in this case we have that
$$\mu^* = \underset{\theta}{\arg \max}  \ln p(y,\theta),$$
such that the variational expectation parameter corresponds to a MAP-estimate. Intuitively, one can contemplate that this objective is local in nature, whereas the minimization of Kullback-Leibler-divergence takes global properties of the posterior into account, such that these two objectives
should not coincide. In the \autoref{appendix:nonequivalence}, we show this non-equivalence in a simple case. Thus, to identify the asymptotic properties of variational Laplace, we will not be able to apply the results of \cite{wang2019frequentist} directly.
Instead, in the next section we first consider special cases of variational Laplace and then generalize our findings to some extent.

\section{Results}
In the following we proceed by studying asymptotic properties of variational Laplace in our example models. We will see that in these models, variational Laplace behaves well under some additional assumptions. Subsequently, we also state general conditions that are sufficient for variational Laplace to adhere to our asymptotic quality criteria. Throughout this section, we focus on presentation of results, derivations may be found in \autoref{sec:derivations}.
\subsection{Single Parameter Model}\label{sec:resultsingleparam}

As we have already seen, in the case of just one parameter variational Laplace reduces to MAP estimation. Thus, we can define the estimators in our first model as 
$$\mu^{(n)} = \underset{\theta \in \Theta} {\arg \max}     \ln p(y,\theta),$$
where we assume that these are correctly identified at every step by the algorithmic procedure employed. To establish consistency of such estimators, we can relate MAP to M-estimation \citep{van2000asymptotic}.
Many of the typical procedures in frequentist point-estimation are special cases of M-estimation. The general criterion of M-estimation is that the point estimates are (near) maximizers of a random function $M^{(n)}(\theta,y) $, that converges in some sense to a non-random function $M(\theta)$.
That is, one is dealing with estimators of the form $$ \mu^{(n)} = \underset{\theta}{\arg \max}\   M^{(n)}(\theta,y).$$
Furthermore, it will often be the case that $M^{(n)}(\theta,y)$ is the sample average of some function $m(\theta,y).$ In this case, the limit function is $$M(\theta) = \mathbb{E}_{\theta_0}[m(\theta,y)].$$
Maximum-likelihood-estimation is a classical, yet important example of M-estimation with $m(\theta,y) = \ln p(y|\theta)$.
It can be shown, that under some conditions on the random function $M^{(n)}$ and the nonrandom function $M$, the maximizers of the former will converge in probability to the maximizers of the latter.
For a MLE, this will typically be the case, such that one can expect the MAP estimate to be consistent too, given that the prior should not matter in the limit. Indeed, we find that the following conditions from \cite{van2000asymptotic}, which can be used to establish the convergence of the MLE are also sufficient to give the consistency of the MAP estimate, provided that the prior puts positive mass on $\theta_0$. 
\begin{proposition}\label{prop1}
Assume that $\ln p(\theta_0)$ is continuous, bounded from above and $ p(\theta_0) > 0$. Define $ M^{(n)}(\theta,y) = \frac{1}{n}\sum_{i=1}^{n} \ln \frac{p(y_i|\theta)}{p(y_i|\theta_0)}$, $M(\theta): = - \operatorname{D_{KL}}\left(p(y|\theta_0)||p(y|\theta)\right)$ and let $\mu^{(n)}$ be a maximizer of $M^{(n)}( \theta,y) + \frac{1}{n} \ln p(\theta)$. If we have that $$\sup_\theta \left| M^{(n)}(\theta,y) - M(\theta)\right| \overset{\mathbb{P}}{\longrightarrow} 0 $$ and $$\sup_{\theta \,:\, d(\theta,\theta_0) \,\geq\, \varepsilon} M(\theta) < M(\theta_0) , \ \forall \varepsilon $$ then $$\mu^{(n)} \overset{\mathbb{P}}{\longrightarrow} \theta_0.$$
\end{proposition} The conditions imposed here are usually too restrictive, and one can see that these will not hold in our example model if $f$ diverges in some direction. Thus, we will give another result that will give consistency for the case of our example model under mild conditions. Recall the model structure 
\begin{align*}
p(y_i|\theta) &= \mathcal{N}(y_i;f(\theta),\sigma^2)\\
p(\theta) &= \mathcal{N}(\theta;m_\theta,s_\theta^2).
\end{align*}
We then have the following result.
\begin{proposition}\label{prop2}
Assume that $f$ is a continuous function and that the model is identifiable, that is $\theta \neq \theta_0 \implies f(\theta_0) \neq f(\theta)$. Define $ M^{(n)}(\theta,y) = \sum_{i=1}^{n} \ln \frac{p(y_i|\theta)}{p(y_i|\theta_0)}$, $M(\theta): = - \operatorname{D_{KL}}\left(p(y|\theta_0)||p(y|\theta)\right)$ and let $\mu^{(n)}$ be a maximizer of $M^{(n)}( \theta,y) + \frac{1}{n} \ln p(\theta)$.
Then $$\mu^{(n)} \overset{\mathbb{P}}{\longrightarrow} \theta_0.$$
\end{proposition}
The proof of this proposition again is directly transferable from a proof of consistency of the MLE which can be found in \cite{van2000asymptotic}.
We further note that we could also relax the assumptions in the sense that the uniqueness of $f(\theta_0)$ could be dropped. Then one would have the convergence to the set of all points where the same value of $f$ is attained. 
As stated here, the results apply to models that are "slightly nonlinear", as had been required in \cite{friston2007variational}. This means for strictly monotone functions, consistency will hold. Of course this identifiability condition is an restriction, as more contrived examples
may not involve such well behaved functions. In such cases, however, working with the true posterior may also not be guaranteed to retrieve the correct parameter \citep{stuart2010inverse}.

\subsection{Linear Model}\label{sec:resultlinear}
In the case of two parameters, the objective of variational Laplace will in general differ from that of MAP. To illustrate this, we gave explicit update equations in the case of a linear transform with unknown variance parameter.
These equations describe an iterative scheme that will only depend on the initial values. It is thus natural to study the asymptotics of fixed points of these equations. Hence, if we define the mapping 
$$ F_n: (\mu_\theta,\sigma_\theta,\mu_\lambda,\sigma_\lambda) \mapsto (\mu_\theta^*,\sigma_\theta^*,\mu_\lambda^*,\sigma_\lambda^*) $$
by our update equations \eqref{updatelinear}, we want to consider points $ (\mu_\theta^{(n)},\sigma_\theta^{(n)},\mu_\lambda^{(n)},\sigma_\lambda^{(n)})$, such that  $$F_n(\mu_\theta^{(n)},\sigma_\theta^{(n)},\mu_\lambda^{(n)},\sigma_\lambda^{(n)}) =  (\mu_\theta^{(n)},\sigma_\theta^{(n)},\mu_\lambda^{(n)},\sigma_\lambda^{(n)}).$$ For such a sequence, we can then study the asymptotic properties. Our following result assures us that the point $(\theta_0,0,\lambda_0,0)$ is the probabilistic limit of this sequence.
\begin{proposition}\label{prop3}
Assume $ (\mu_\theta^{(n)},\sigma_\theta^{(n)},\mu_\lambda^{(n)},\sigma_\lambda^{(n)})$ is a fixed point of $F_n$. Then we have $$ (\mu_\theta^{(n)},\sigma_\theta^{(n)},\mu_\lambda^{(n)},\sigma_\lambda^{(n)}) \overset{\mathbb{P}} \longrightarrow (\theta_0,0,\lambda_0,0).$$
\end{proposition}
Hence, under the reasonable assumption that the variational Laplace scheme will have a fixed point in this model, this proposition in particular establishes the consistency of the variational estimates $(\mu_\theta^{(n)},\mu_\lambda^{(n)})$. We can further assess the quality of these  estimates in the light of their asymptotic distribution, as the next proposition shows.
\begin{proposition}\label{prop:asymptoticefficiency}
Under the assumption of the previous proposition, the variational mean parameters are asymptotically efficient. That is, we have
$$\sqrt{n} \left( \begin{pmatrix} \mu_\theta^{(n)} \\ \mu_\lambda^{(n)} \end{pmatrix} - \begin{pmatrix} \theta_0 \\ \lambda_0 \end{pmatrix} \right)  \overset{d}{\longrightarrow} G \sim \mathcal{N}\left(0,\mathcal{I}(\theta_0,\lambda_0)^{-1} \right).$$
\end{proposition}
In this sense, estimates generated by variational Laplace attain optimality in this model. To compare variational Laplace with the results presented in \fullref{sec:review}, it is of interest to also study the asymptotics of the variational distribution.
First, we note that \autoref{prop3}  implies the convergence of the variational densities to point masses in the distributional sense. That is, we have that 
\begin{equation}q_{\mu_\theta^{(n)},\sigma_\theta^{(n)}}(\theta) q_{\mu_\lambda^{(n)},\sigma_\lambda^{(n)}}(\lambda) \overset{d}{\longrightarrow} \delta_{\theta_0}(\theta) \delta_{\lambda_0}(\lambda),\end{equation}
where this follows from the fact that for a normal distribution the convergence of mean and variance imply convergence in distribution by Levy's continuity theorem \citep{van2000asymptotic}. Thus, in this sense, the variational densities behave like the true posterior in the limit. 
We also want to provide some insight on the rescaled variational distribution, in a similar spirit as \eqref{variationalbernstein} resembling the Bernstein-von-Mises theorem. To this end, we notice that with our assumptions we have 
\begin{equation}
n \begin{pmatrix}\sigma_\theta^{(n)}  & 0 \\ 0 & \sigma_\lambda^{(n)} \end{pmatrix} \overset{\mathbb{P}}{\longrightarrow} \begin{pmatrix}\frac{\exp(\lambda_0)}{a^2}  & 0 \\ 0 & 2 \end{pmatrix} =  \mathcal{I}(\theta_0,\lambda_0)^{-1}.      \\
\end{equation}
Combining this with \autoref{prop:asymptoticefficiency},  in analogy to the limiting rescaled posterior distributions in the Bernstein-von-Mises Theorem \eqref{bernstein_von_mises} and its variational inference counterpart, one can suggestively state the asymptotic form of the rescaled variational densities as
\begin{equation}
\mathcal{N}\left( G, \mathcal{I}\left(\theta_0,\lambda_0\right)^{-1}\right) ; G \sim \mathcal{N}\left(0,\mathcal{I}(\theta_0,\lambda_0)^{-1}\right).
\end{equation}
In particular it turns out that asymptotically, the rescaled versions of the variational distribution obtained by variational Laplace and the optimal distribution in terms of \eqref{optimalGaussian} are equivalent in total variation.
\begin{proposition}\label{variationconvergence}
Under the previous assumption, let  $$\widetilde{\mu}^{(n)} :=\begin{pmatrix} \sqrt{n} \left(\mu_\theta^{(n)}-\theta_0\right) \\ \sqrt{n}\left(\mu_\lambda^{(n)}-\lambda_0\right) \end{pmatrix},$$
$$\widetilde{\Sigma}^{(n)} := n \begin{pmatrix}\sigma_\theta^{(n)}  & 0 \\ 0 & \sigma_\lambda^{(n)} \end{pmatrix} ,$$
and  $\widetilde{q}^{*(n)}_h(h)$ the optimal variational density in terms of \eqref{optimalGaussian}, after the reparameterization $h = \sqrt{n}(\theta-\theta_0)$.
Then we have
$$
\left\Vert \widetilde{q}_{\widetilde{\mu}^{(n)},\widetilde{\Sigma}^{(n)}}(h) - \widetilde{q}^{*(n)}_h(h) \right \Vert_{TV} \overset{\mathbb{P}}{\longrightarrow} 0.
$$
\end{proposition}
Thus we see that for this model in the limit, in the sense of total variation, variational Laplace behaves as if the optimal density was correctly identified. Additionally, since the Fisher information matrix in this model is diagonal, both also eventually coincide with the limiting distribution for the posterior in the sense of \eqref{bernstein_von_mises}.

\subsection{General Asymptotic Properties of Variational Laplace}\label{sec:resultsgeneral}
We have seen that in our examples, variational Laplace enjoyed desirable asymptotic properties. Below, we provide two results that may affirm these properties of variational Laplace in a more general setting. 
In this section we will always assume an identifiable model $p(y|\theta)$, such that $ \theta \neq \theta_0 \implies p(y|\theta) \neq p(y|\theta_0)$. 
Furthermore, as before, we assume that we are working with a series of fixed points of the function $F_n$ conveying the iterative updates of the variational parameters.
That is, we have a sequence of variational parameters $(\mu_1^{(n)},\Sigma_1^{(n)},...,\mu_p^{(n)},\Sigma_p^{(n)})$ fulfilling  $$F_n(\mu_1^{(n)},\Sigma_1^{(n)},...,\mu_p^{(n)},\Sigma_p^{(n)}) = (\mu_1^{(n)},\Sigma_1^{(n)},...,\mu_p^{(n)},\Sigma_p^{(n)}).$$

Our first result will give sufficient conditions such that the true parameter is the only sensible limit of the variational Laplace scheme, in the sense that the sequence of variational means may not converge in probability to a parameter other than the true parameter.
To establish this result, we assume that the sequence of variational means converges in probability to some set of points
$$(\mu_1^{(n)},...,\mu_p^{(n)}) \overset{\mathbb{P}}{\longrightarrow} (\theta_1',...,\theta_p'),$$
where it is not determined that $\theta'= (\theta_1',...,\theta_p')$ corresponds to the true parameters $\theta_0$.
We furthermore define a modified version of the variational energy functions at the fixed points as
\begin{equation} I_j^{(n)}(\mu_j) = \sum_{i=1}^{n}\frac{\ln p(y_i|,\mu_j,\mu_{/j}^{(n)})}{\ln p(y|\theta_0)}  + \ln p(\mu_j,\mu_{/j}^{(n)})  + \frac{1}{2} \sum_{k \neq j}\operatorname{tr}\left[ \Sigma_k^{(n)} H_{L_y}^{\theta_k}(\mu_{j},\mu_{/j}^{(n)})\right]\end{equation}
where $\mu_j^{(n)}$ still maximizes $I_j^{(n)}$, as we just subtracted a constant.
\begin{proposition}\label{prop4}
Assume the Kullback-Leibler-divergence $\operatorname{D_{KL}}\left( p(y|\theta_0) || p(y|\theta) \right)$ is convex and differentiable in the parameter $\theta$. Furthermore, assume that the variational energies fulfill $$\frac{I_j^{(n)}(\mu_j)}{n} \overset{\mathbb{P}}{\longrightarrow} -\operatorname{D_{KL}}(p(y|\theta_0) || p(y| \mu_j, \theta'_{/j})$$
and in turn $$ \mu_j^{(n)} \overset{\mathbb{P}}{\longrightarrow} \underset{\mu_j}{\arg \max}  -\operatorname{D_{KL}}(p(y|\theta_0) || p(y| \mu_j, \theta'_{/j}).$$
Then $$ \theta' = \theta_0.$$
\end{proposition}
The proof of this proposition is rather short: It suffices to note that under these conditions we get that for all $j$, $$ \theta'_j = \underset{\theta_j}{\arg \max}   -\operatorname{D_{KL}}(p(y|\theta_0) || p(y| \theta_j, \theta'_{/j}),$$
which, by the concavity of $-\operatorname{D_{KL}}(p(y|\theta_0) || p(y| \theta_j, \theta'_{/j})$ makes it a local and hence global maximum \citep{rockafellar1970convex}, this in turn is unique due to the identifiability of the model and therefore has to be $\theta_0$, which concludes the proof. 

\noindent The intuition behind the conditions that we pose on the variational energies is motivated by the equivalence of MLE and MAP estimates that we discussed earlier. Here, one assumes that the additional terms involving the second derivatives of the model do not matter in the limit. Also, one has a slightly stronger notion of the convergence of the maximizers, as the likelihood is not evaluated at a deterministic point but depends on the sequence of the remaining estimators as well.
Whether the given assumptions will hold in typical models is not immediate and will require further analysis. In \fullref{sec:expfam} we discuss their validity in naturally parameterized exponential families and find that our conditions are indeed satisfied under mild assumptions on the prior.

The previous proposition required the convergence in probability of the variational means to some point as an assumption. Hence, the result does not exactly establish consistency, as such a convergence may not be given in the first place. In what follows, we want to provide a theorem that indeed yields consistency of variational Laplace. To establish this theorem, we have to make stronger assumptions on the underlying parameter space and the joint probability density, but can in turn drop the assumption of a given convergence of the parameters.

The main idea of the theorem is to characterize the fixed points of variational Laplace as zeros of the gradient of the variational energy functions, where this gradient deviates from the gradient of the likelihood only by a hopefully negligible remainder term. We will therefore assume that the log joint probability of our model is three times differentiable with respect to $\theta$ and in turn define a vector-valued remainder-term as 
\begin{equation}
\begin{aligned}
R^{(n)}(y,\mu_j,\mu_{/j})_j  = \frac{1}{n}\left( -\frac{\partial}{\partial \theta_j} \ln p(\theta_j,\theta_{/j}) \bigg \rvert_{ \left(\mu_j,\mu_{/j}\right)} +  \frac{1}{2} \sum_{k \neq j} \frac{\partial}{\partial \theta_j} \operatorname{tr}\left[ H_{L_y}^{\theta_k}\left(\mu_j,\mu_{/j}\right)^{-1}   H_{L_y}^{\theta_k}  \left(\theta_j,\theta_{/j}\right)    \right] \bigg \rvert_{\mu_j,\mu_{/j}} \right).
\end{aligned}
\end{equation}
Now the consistency and asymptotic normality of the model can be given under some global conditions on the model and the behaviour of this remainder term.
\begin{theorem}\label{varlaptheorem}
Assume the following conditions hold
\begin{enumerate}
\item $\mu^{(n)}$ is in the interior of $\Theta$
\item $ R^{(n)}(y,\mu^{(n)}) \overset{\mathbb{P}}{\longrightarrow} 0$
\item $\underset{\theta \in \Theta}{\sup}  \left\Vert \frac{1}{n} \frac{\partial}{\partial \theta} \ln  p(y|\theta) +  \frac{\partial}{\partial \theta}  \operatorname{D_{KL}}(p(y|\theta_0)|| p(y|\theta))\right\Vert \overset{\mathbb{P}}{\longrightarrow} 0$
\item $\underset{\theta : d(\theta,\theta_0) \geq \varepsilon}{\inf} \left\Vert  -\frac{\partial}{\partial \theta}  \operatorname{D_{KL}}(p(y|\theta_0)|| p(y|\theta))\right\Vert >0$ for all $\varepsilon >0$
\item $\frac{\partial}{\partial \theta} \int \ln p(y|\theta) p(y|\theta_0) dy = \int  \frac{\partial}{\partial \theta}\ln p(y|\theta) p(y|\theta_0) dy$
\end{enumerate}
then we have consistency of the variational mean,
$$ \mu^{(n)} \overset{\mathbb{P}}{\longrightarrow} \theta_0.$$
If in addition we have 
$$ \sqrt{n}  R^{(n)}(y,\mu^{(n)})\overset{\mathbb{P}}{\longrightarrow} 0,$$
the Fisher information matrix $\mathcal{I}(\theta_0)$ exists, is invertible and the exchange of derivative and integration as in \eqref{fisherconditions} is valid,
 then we have asymptotic efficiency
$$ \sqrt{n}(\mu^{(n)}-\theta_0) \overset{d}{\longrightarrow} G \sim \mathcal{N}(0,\mathcal{I}(\theta_0)^{-1}).$$
\end{theorem}
The proof of this theorem may be sketched as follows: The assumptions on the gradient of the likelihood and the Kullback-Leibler divergence establish the consistency of a MLE being a zero of this gradient by \autoref{Z-estimator}. 
As the remainder term is by assumption asymptotically negligible, this then also gives the consistency of the estimates generated by variational Laplace, as the aforementioned derivative at these estimates will deviate from zero exactly by the
remainder term. The additional assumption on the convergence rate of the remainder term is then enough to also give the asymptotic normality of the estimates with the limit distribution being the same as obtained by the MLE. 
Again, as in \autoref{prop1}, our assumptions may be too restrictive in practice, for example they will typically require compactness of the parameter space $\Theta$. Still, this theorem has some value in pointing out
that one main point in establishing the consistency of variational Laplace will be to control the remainder term, which may also be useful for more general proofs. In \fullref{sec:expfam} we again discuss the validity in natural parameter exponential families, where they are typically satisfied, provided the parameter space
is compact.

We want to conclude this section by some general remarks on the 
covariance parameters in variational Laplace. Recall that we have
\begin{equation}
(\Sigma_k^{(n)})^{-1} = -H_{L_y}^{\theta_k}(\mu^{(n)}) = -\sum \limits_{i=1}^{n}  \frac{\partial^2 \ln p(y_i|\theta)}{\partial \theta_k \partial \theta_k^T}(\mu^{(n)}) -   \frac{\partial^2 \ln p(\theta)}{\partial \theta_k \partial \theta_k^T}(\mu^{(n)})
\end{equation}
by definition. Thus, we get that 
\begin{equation}
(n\Sigma_k^{(n)})^{-1} = -\frac{1}{n}\sum\limits_{i=1}^{n}  \frac{\partial^2 \ln p(y_i|\theta)}{\partial \theta_k \partial \theta_k^T}(\mu^{(n)}) -  \frac{1}{n} \frac{\partial^2 \ln p(\theta)}{\partial \theta_k \partial \theta_k^T}(\mu^{(n)})
\end{equation}
Assuming the consistency of $\mu^{(n)}$, that is $\mu^{(n)} \overset{\mathbb{P}}{\longrightarrow} \theta_0$ the second term involving the prior will be negligible. The first term is what would be called the observed Fisher information matrix, if $\mu^{(n)}$ were the MLE.
If suitable conditions are given such that $\mu^{(n)}$ may be replaced by $\theta_0$ when passing to the limit, one would have
\begin{equation}\label{asymptoticvariance}
n\Sigma_k^{(n)}\overset{\mathbb{P}}{\longrightarrow} -\mathbb{E}\left[\frac{\partial^2 \ln p(y|\theta)}{\partial \theta_k \partial \theta_k^T}(\theta_0)\right]^{-1} 
\end{equation}
The expression on the right is the inverse of a quadratic subpart of the models Fisher information matrix. Hence, in this case one would have the same asymptotic underdispersion that we have seen previously in \eqref{underdisp}. As an example, in the case that all parameters are 
one-dimensional, one can see that this gives a matrix having the inverse diagonal elements of the models Fisher information matrix, such that in this case, our derivations in \autoref{sec:entropy} equally apply.
These considerations thus may suggest that at least in the case that the model is sufficiently well-behaved, general variational inference and variational Laplace show similar asymptotic behaviour in terms of the distribution.

\section{Simulations}\label{sec:simulations}

In the previous sections we have seen, that variational Laplace should have desirable theoretical properties in our example models. Here, we want to back up these notions by simulations. All simulation experiments were conducted using Python, version 3.7.6 ~\citep{pythonmanual}.
\subsection{Single Parameter Model}
\noindent\begin{minipage}[htb]{0.46\textwidth}
\noindent
As we have derived in \autoref{prop2}, if the transform function $f$ is well-behaved in the sense that different values of $\theta$ will yield different function values, one can expect that the parameter-estimates obtained by variational Laplace (which in this special case correspond to the 
maximum-a-posteriori estimates) will be asymptotically consistent. To illustrate this behaviour, we repeatedly drew $500$ samples with growing size $n$ from three models that fulfill these requirements (${f(\theta) = e^\theta, f(\theta) = \theta^3, f(\theta) = e^{2\theta} + \theta^3}$). We obtained the variational Laplace/MAP estimates by minimizing the negative log joint probability via a Newton conjugate gradient algorithm implemented in scipy \citep{scipy}. In all three models, we chose the same true parameter value ($\theta_0 = 1$) and the same prior and fixed parameters (${m_\theta = 10,s_\theta = 10,\sigma = 1}$). In all simulations, as \autoref{fig:nonlinearfig} illustrates, we observed that the distribution of the variational estimates $\mu_\theta$ increasingly concentrated on the true parameter value with growing $n$. These observations are in line with our theoretical derivations. 
\end{minipage}
\hfill
\begin{minipage}{0.46\textwidth}\raggedleft
  \includegraphics[width=\linewidth]{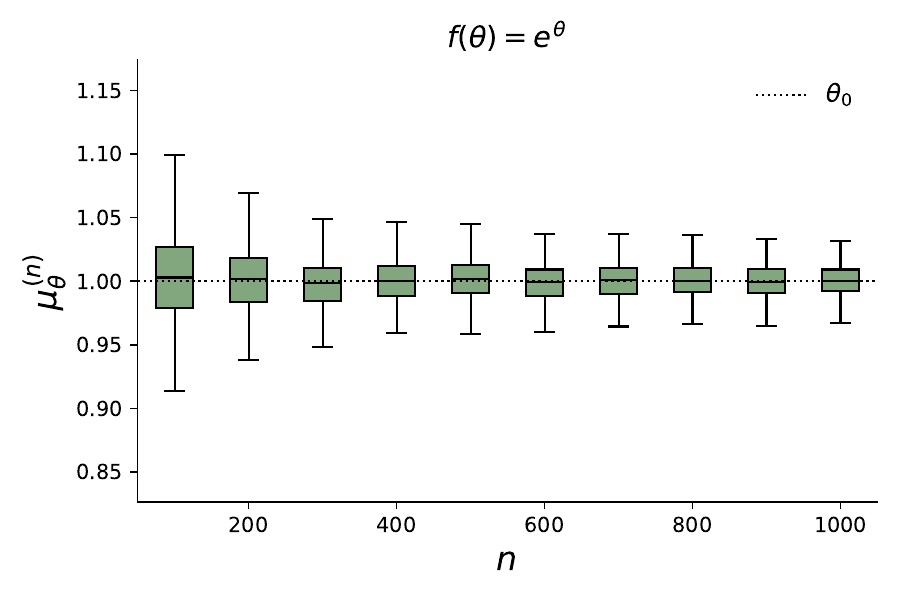} 
  \includegraphics[width=\linewidth]{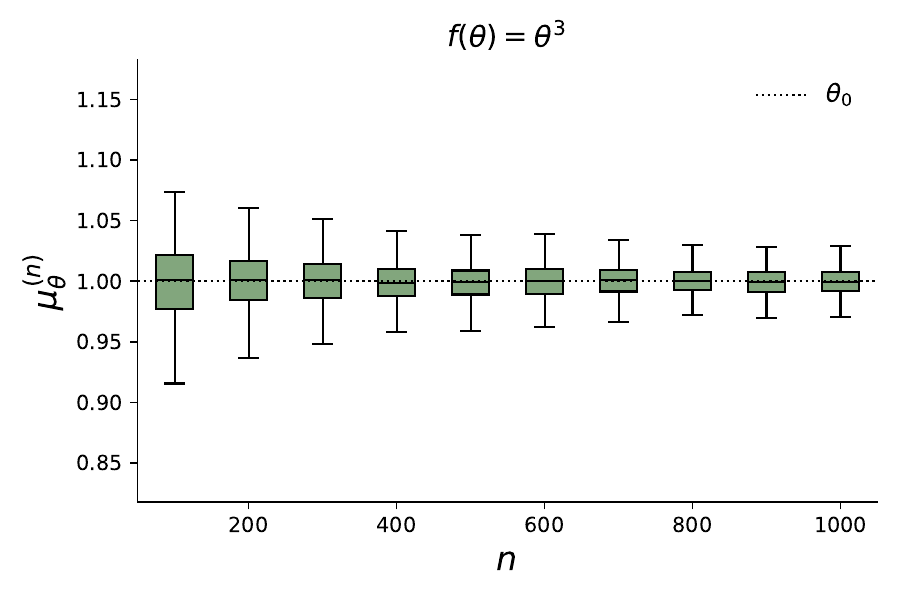}
  \includegraphics[width=\linewidth]{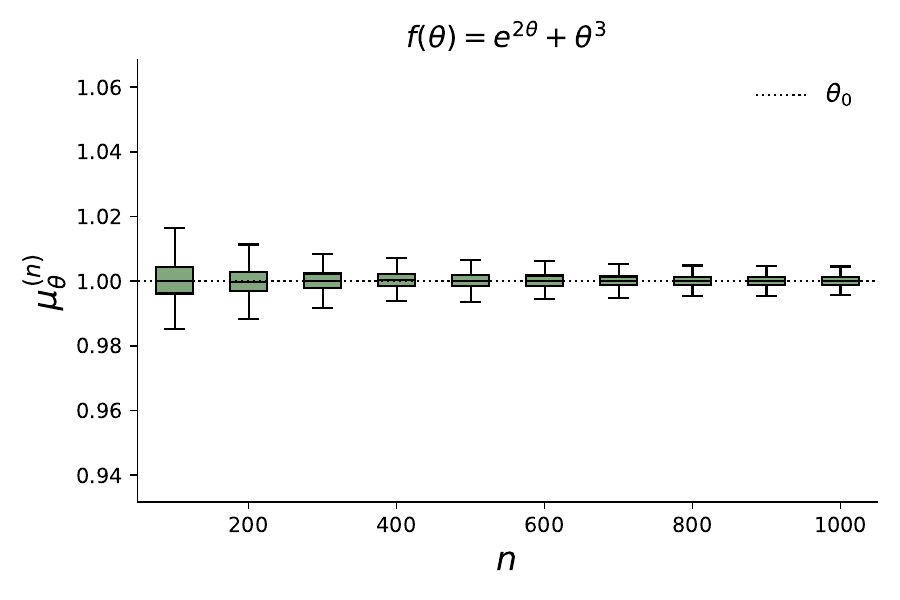}
\captionof{figure}{Boxplots of $500$ samples of variational expectation parameters $\mu_\theta^{(n)}$ with varying size $n$. True parameter is indicated by dotted line.}\label{fig:nonlinearfig}
\end{minipage}

\subsection{Linear Model}
To illustrate the convergence properties of variational Laplace in the linear model, we repeatedly drew $500$ samples with different size $n$ from a fixed model (${a=3, \theta_0 = 2, \lambda_0 = 2},\\ {m_\theta = m_\lambda = s_\theta=s_\lambda = 1}$).
We computed the variational mean and expecation parameters in an iterative manner according to the equations given in \eqref{updatelinear}
The result of these simulations demonstrate, that the theoretical derivations made in the previous chapter indeed hold in practice. We observed that for growing data size $n$, the distribution of the variational expectation parameters concentrates ever closer on the true parameters $(\theta_0,\lambda_0)$
as is shown in  \fullref{fig:probfig}. We furthermore observed that the variational variances tend to zero with $n$ growing, which also was in line with our predictions.
\begin{figure*}[htb]\label{linearplots}
\begin{subfigure}{.5\textwidth}
  \centering
  \includegraphics[width=\linewidth]{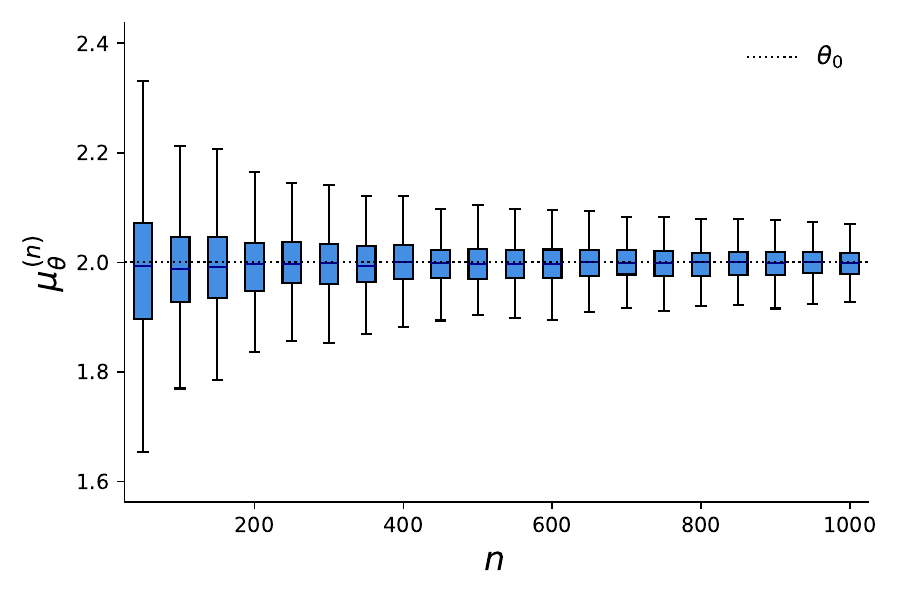}
  \label{fig:sfig1}
\end{subfigure}%
\hfill
\begin{subfigure}{.5\textwidth}
  \centering
  \includegraphics[width=\linewidth]{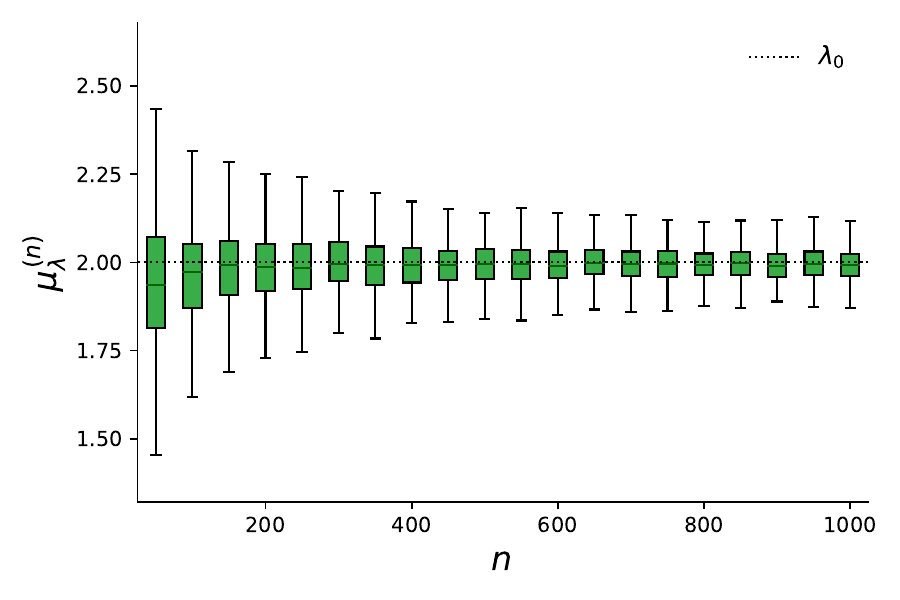}
  \label{fig:sfig2}
\end{subfigure}
\vskip\baselineskip
\begin{subfigure}{.5\textwidth}
  \centering
  \includegraphics[width=\linewidth]{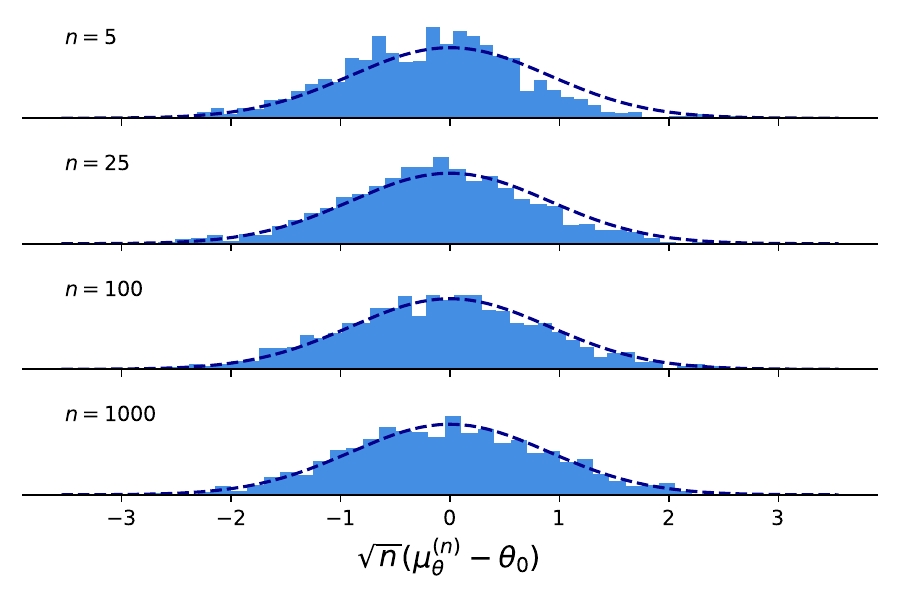}
  \label{fig:sfig1}
\end{subfigure}%
\hfill
\begin{subfigure}{.5\textwidth}
  \centering
  \includegraphics[width=\linewidth]{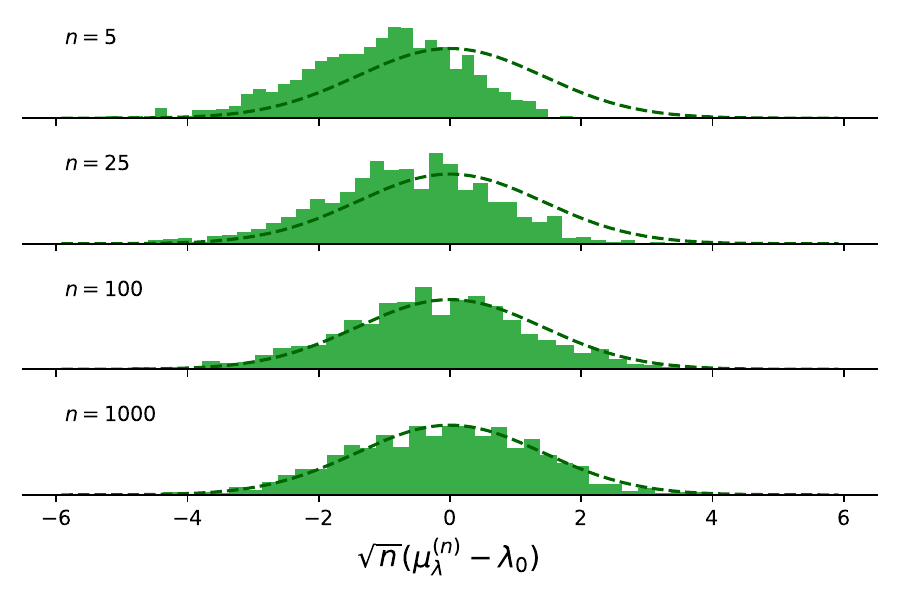}
  \label{fig:sfig2}
\end{subfigure}
\caption{\textbf{Top row}: Boxplots of $500$ samples of variational expectation parameters with varying size $n$ for $\theta$ (left) and $\lambda$ (right). True parameter is indicated by dotted line. \textbf{Bottom row}: Normalized histograms of rescaled variational expectation parameters for $\theta$ (left) and $\lambda$(right). Probability
density of the analytically derived limit distributions is plotted as solid line. }
\label{fig:probfig}
\end{figure*}
To illustrate the distributional limits,
we centered our variational means on the true parameters and rescaled them by the square root of the sample sizes. We computed normalized histograms ($40$ bins) of these values and plotted them against the marginal distributions of the analytically derived limit in \autoref{prop:asymptoticefficiency}, that is a Gaussian with mean zero and variance given by the respective diagonal element of the inverse Fisher information matrix. These plots, shown in the bottom row of \fullref{fig:probfig}, suggest that the rescaled parameters indeed approach this optimal Gaussian and are therefore asymptotically efficient estimators.
To give a visual intuition of the convergence of the variational densities in terms of total variation, we drew a small number of $10$ samples of growing size $n$ with the same parameters as before. We again computed rescaled, centered variational means and the rescaled variational variances.  In \fullref{fig:totalvarfig}, we plot the isocontour ellipses corresponding to one standard deviation in both directions for the resulting densities. We compare these isocontours with the respective ellipses of the limit density of the Bernstein-von-Mises theorem, that is a Gaussian with mean given by the rescaled, centered MLE and covariance matrix given by the inverse Fisher information matrix. We note that for larger sample size the resulting ellipses are nearly indistinguishable. 
\begin{figure}
\includegraphics[width=0.9\linewidth]{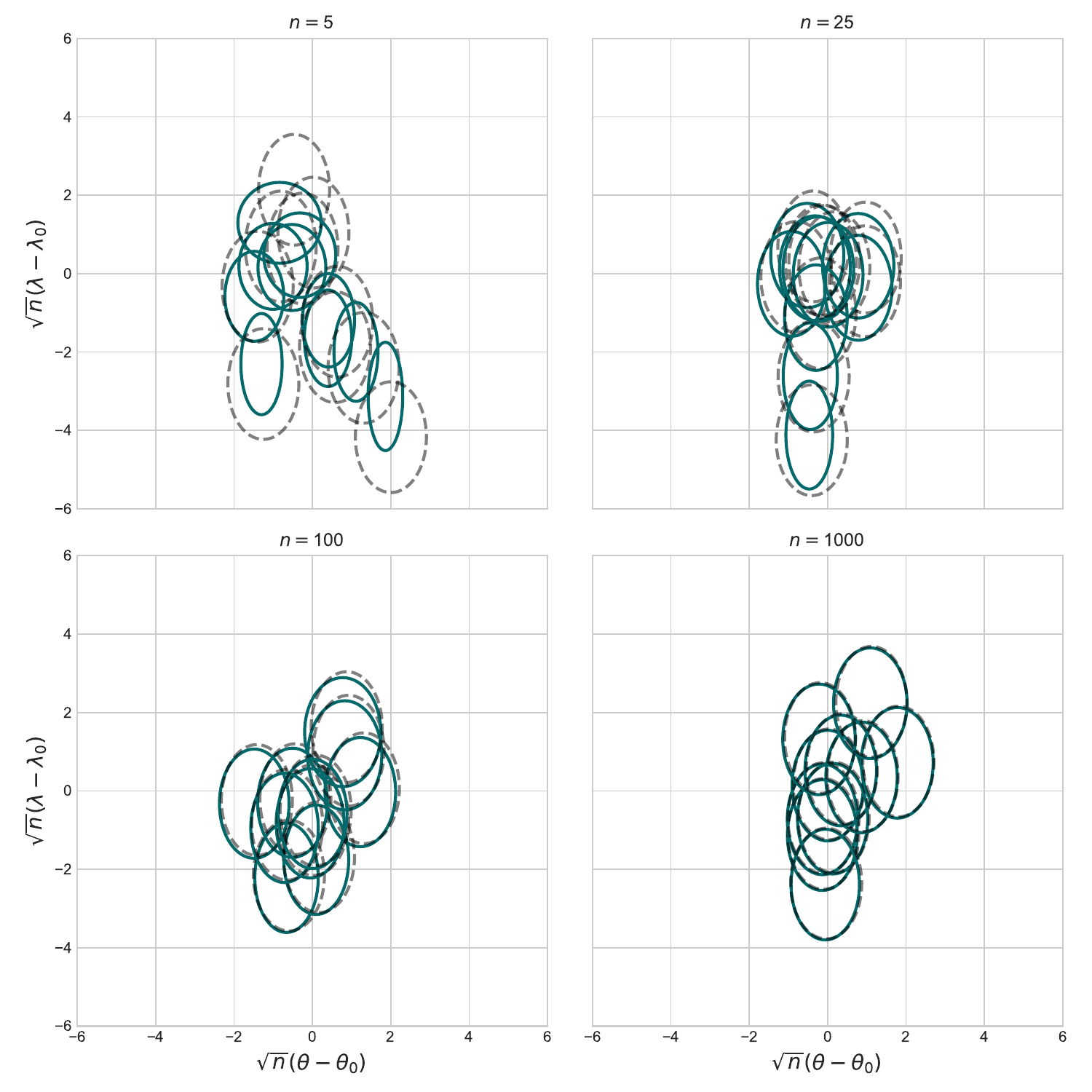}
\caption{Ellipses of half width and height given by the standard deviations corresponding to the respective entries in the diagonal covariance matrizes of the rescaled variational densities (solid blue lines) and a density with mean given by the rescaled MLE and covariance matrix given by the inverse Fisher information matrix (dashed lines).}
\label{fig:totalvarfig}
\end{figure}

\section{Discussion}\label{sec:discussion}
In this work, we have investigated asymptotic, frequentist statistical properties of variational Laplace, a scheme for approximate variational inference popular in the cognitive neuroscience community. We provided an introduction to all topics relevant for the understanding of this work: We briefly presented the main ideas behind variational inference in general and gave a distinction of free-form and fixed-form mean-field approaches. We detailed the iterative updates at the heart of variational Laplace and explicitly stated the update equations in particular example models. The frequentist asymptotic treatment of both point estimator and Bayesian posterior distributions was introduced. Subsequently, we exhibited the current state of research regarding these properties in variational inference. With this theoretical background elucidated, we proceeded to present our results: We pointed out that the expectation parameter in variational Laplace, in the case of one parameter, reduces to a maximum-a-posteriori estimate, such that in this simple case all theoretical properties of this procedure are inherited. We then proceeded to investigate a simple linear model in which this equivalence does not hold anymore, where we could confirm consistency and asymptotic normality of the variational expectation parameters, which also implied convergence in total variation distance of the variational distributions. The results of these applied examples were supported by simulation experiments that we conducted. Additionally, we established conditions under which variational Laplace will yield consistent and asymptotically efficient estimators.

As we have noted before, we are not aware of any work investigating these theoretical properties of variational Laplace. There are, however, a number of studies that have pursued to justify the use of variational Laplace by simulation studies. In terms of parameter recovery, schemes based on variational Laplace in the context of nonlinear dynamical systems identification have shown satisfying results \citep{friston2008variationalfilt,daunizeau2009variational,friston2008variationaldem}. Also, in simulation studies of DCM for EEG \citep{ostwald2016probabilistic} and GLM for FMRI \citep{starke2017}, using schemes similar in nature to the one described here, good parameter recovery was observed in most cases - with some caveats to the estimation of covariance component parameters. We furthermore note that many studies also considered the accuracy of variational Laplace in terms of model selection: Here, in simulations, variational Laplace and related methods in general were able to adequately retrieve the correct model generating the data from a number of alternatives \citep{friston2007variational,starke2017,friston2008multiple,friston2019variational,friston2008Bayesian,chappell2008variational}.
In the following, we want to discuss limitations of our approach, possible future research and implications for practice. 
\subsection{Model Assumptions}
Compared to models used in practice, all discussed models in this work are of considerable simplicity. While this facilitated the analytic treatment of the resulting equations in the variational Laplace scheme, these results may not directly transfer to more complicated models. In a similar manner, our general theorem makes assumptions that may not hold in typical applications, as for example a compact parameter space. Nevertheless, we suspect that the approaches taken here may be generalized to some extent to models used in the field. As an example, the linear model that we discussed here is a one-dimensional simplification of the popular example of a GLM where the covariance matrix is parameterized by a weighted sum of fixed, known matrizes $\Phi_l$, that is 
$$
Q(\lambda) = \sum_{l}  \exp(\lambda_l)\Phi_l,
$$
 compare also \citet{friston2007variational,lopez2014algorithmic,starke2017} and our discussion of this special case in \autoref{sec:linearlaplace}. Hence, one could speculate that our derivations made for the one dimensional case may also generalize to more contrived examples, provided that the model is still identifiable - which may be guaranteed by choosing $\Phi_l$ linearly independent and $X$ having full rank. However, in this more intricate setting it might not be possible to give an explicit analytical expression for the iterates anymore \citep{starke2017}. Further note that in this case, as already observed in \citet{friston2007variational} the Hessian with respect to the covariance component parameters will not be positive definite in general, such that in particular the log-concavity of the model is not given any more, which may impede derivations similar to ours. Also note that this might be at the root of general problems with covariance component estimation that seem to commonly arise in this situation \citep{starke2017}. 

We still believe that the most promising approach in this situation might be to consider the implicit definition of the fixed points as stationary points of the 'variational energies' as we did in \fullref{varlaptheorem} such that existing tools taylored to establish asymptotic properties of such zero estimators (Z-estimators) may be leveraged \citep{van2000asymptotic}.  We also have to note that our treatment would correspond to repeatedly drawing independent, identically distributed samples with the same fixed design matrix $X$. Although this consideration may still have some theoretical merit, a more natural approach is to consider just one sample $y$ with the row-size of the design matrix growing to $\infty$. Generally, consistency of (variational) estimators can be established in such a setting under additional assumptions on the matrix $X$ \citep{amemiya1974multivariate,van2000asymptotic,you2014variational}. Thus, it may be an interesting question to investigate whether our results transfer to this more natural setting.

At several points in the derivations in this work, we assumed that a fixed point of the variational Laplace update iteration exists and took this as a starting point for analytical treatment. The existence of such a fixed point is not guaranteed a priori \citep{granas2003elementary}. A more deliberate treatment of our examples would thus have included the verification of the existence of such a fixed point. It might be an interesting venture in itself to establish conditions that ensure the existence of a fixed point of the variational Laplace iterations, or a probablistic version of such a statement as for example has been obtained in the case of free-form mean-field variational inference in \citet{wang2012convergence}. In this regard, it could also be explored whether using the approximated versions of the Hessian, as in \eqref{varlaptheta}, affects the answer to this question.

\subsection{Variational Laplace in Practice}
In a restricted setting, we could show consistency and asymptotic normality of variational Laplace. Thus, these results may serve as a first step towards a theoretical justification of the use of this procedure in cognitive neuroscience. This is noteworthy, since the derivation of the update equations as we have seen is rather heuristically motivated. Still, as we saw, in the large sample limit, these heuristical arguments might suffice to recover true parameter values with this scheme. However, we want to emphasize that our results only concern the variational expectation parameters as point estimators. In particular, we have argued that if one considers the variational density as a proxy to the posterior distribution, the covariance parameter estimates of variational Laplace, just as in free-form mean-field variational inference, underestimate the true variance of the posterior asymptotically \citep{wang2019frequentist}.
This is remarkable, especially when one considers the direct comparison to the 'classical' Laplace approximation: As the latter is a special case of variational Laplace, under similar conditions  the posterior variance of the classical Laplace approximation will converge to the 'correct' asymptotic variance, namely the inverse Fisher information matrix - compare equation \eqref{asymptoticvariance}. Under the assumption that these convergences hold, one can thus see that asymptotically, variational Laplace will in this sense behave worse than the classical Laplace approximation. Hence, at least under the lense of frequentist asymptotics, our results do not provide a reason to prefer variational Laplace over the classical Laplace approximation, in particular since the former will typically involve the optimization of slightly more complicated formulae. Therefore, additional research might be needed to justify the use of variational Laplace over the related method.

In \citet{friston2007variational}, the introduction of variational Laplace over the classical Laplace approximation was motivated by cases of a ponounced non-Gaussianity of the posterior in a finite setting, "especially when the mode is not near the majority of the probability mass". Thus, it may be natural to study whether the set of fixed points of the variational Laplace scheme in such cases is closer to the lions share of the distribution, in a finite setting, and whether in turn the approximation to the posterior is superior.
If this yields positive results, one could possibly modify the method in the following way: Retaining the mean parameters obtained by variational Laplace but using the full Hessian of the model at these mean parameters, that is $$\Sigma^{(n)} = - H_{L{y}}\left(\mu^{(n)}\right)^{-1},$$ would leave the point estimates untouched but could deal with the problem of underestimating the variance. Note that a similar - albeit differently derived - modification of free-form mean-field variational approaches has been proposed to overcome the underestimation of variance \citep{giordano2015linear}.

Finally, we also want to mention that the results presented here did not concern the use of the $\operatorname{ELBO }$ for model comparison. This, however, is one frequent application of variational methods in cognitive neuroscience, where approximated versions of the $\operatorname{ELBO}$ derived by variational Laplace are used for model comparison. Hence, this particular application of variational Laplace should be analyzed in future work, both in a finite setting as well as in the asymptotic limit -  here, one also could draw from a growing body of literature investigating model selection in variational inference \citep{cherief2018consistency,you2014variational,mcgrory2007variational}.

\section{Conclusion}
In conclusion, in this work we studied asymptotic properties of variational Laplace in a simple setting. We are convinced that theoretically establishing asymptotic quality properties of methods widely used in practice is an endeavour that should receive ongoing attention by practicioners. 
In particular this is of importance when, as in this case, approximations and heuristical derivations are employed to design a procedure. Verifying whether these procedures adhere to standard quality criteria thus ultimately justifies their use and puts these schemes on a more stable theoretical foundation. We hope that our small contribution in this direction may serve as a valuable starting point for future investigation of these properties.

\appendix
\begin{appendices}

\counterwithin{lemma}{section}
\counterwithin{theorem}{section}

\section{Derivation of Main Results}\label{sec:derivations}

Below, we provide proofs of our main results. In \fullnameref{sec:aux} we provide additional details on the notation used here and also state important results from the literature which are used in our derivations.
\subsection{Single Parameter Model}
\subsubsection{Gradient-Based Update Scheme}
If we write out the full log-joint of the model , it reads:
\begin{equation}
\begin{aligned}
\ln p(\theta,y) = &-\frac{1}{2} (f(\theta) \mathbf {1}_n - y)^T s^{-2} I_n (f(\theta) \mathbf {1}_n - y) 
-\frac{1}{2}  \frac{(\theta - m_\theta)^2}{s_\theta^2} + C
\end{aligned}
\end{equation}
Now, calculating the derivatives with respect to $\theta$ of this yields for the first derivative
\begin{equation}
\begin{aligned}
&\frac{\partial}{\partial \theta} \ln p(\theta,y) = - (f'(\theta)\mathbf {1}_n)^T s^{-2} I_n (f(\theta) \mathbf{1}_n - y) -  \frac{\theta - m_\theta}{s_\theta^2} \\
&=  \frac{m_\theta-\theta}{s_\theta^2} - \frac{n}{s^2} (f'(\theta)(f(\theta) - \overline{y}) \\
\end{aligned}
\end{equation}
and for the second derivative
\begin{equation}
\begin{aligned}
& \frac{\partial}{\partial^2 \theta} \ln p(\theta,y) = -(f(\theta)'\mathbf{1}_n)^T s^{-2} I_n (f(\theta)' \mathbf{1}_n  - (f(\theta)'' \mathbf{1}_n)^T s^{-2} I_n (f(\theta) \mathbf{1}_n - y) - \frac{1}{s_\theta^2},
\end{aligned}
\end{equation}
which, after neglecting the second derivative of the function $f$ becomes 
\begin{equation}
\begin{aligned}
&\frac{\partial}{\partial^2 \theta} \ln p(\theta,y) \approx -(f(\theta)'\mathbf{1}_n)^T s^{-2} I_n (f(\theta)' \mathbf{1}_n  - \frac{1}{s_\theta^2}\\
& = -\frac{n}{s^2} f'(\theta)^2 - \frac{1}{s_\theta^2} = -\frac{n f'(\theta)^2s_\theta^2 - s^2}{s^2 s_\theta^2}.
\end{aligned}
\end{equation}

\subsubsection{Asymptotic Results}\label{sec:proofonedimensional}
\myparagraph{Proof of \fullref{prop1}}
To prove \fullref{prop1}, first observe that 
\begin{equation}
M^{(n)}(\theta,y) \overset{a.s.}{\longrightarrow} -\operatorname{D_{KL}}(p(y|\theta_0)||p(y|\theta)) = M(\theta)
\end{equation}
by the law of large numbers. 
Furthermore, define a maximum-likelihood-estimator 
\begin{equation}\label{defmle}
\hat{\theta}^{(n)}_{\text{MLE}} = \underset{\theta \in \Theta}{\arg \max} \  M^{(n)}(\theta,y).
\end{equation}
This estimator will be consistent by \fullref{M-estimator}.
Also observe that by definition we have
\begin{equation}
\begin{aligned}
M^{(n)}(\mu^{(n)},y) &= M^{(n)}(\hat{\theta}^{(n)}_{\text{MLE}},y) - ( M^{(n)}(\hat{\theta}^{(n)}_{\text{MLE}},y) - M^{(n)}(\mu^{(n)},y)) \\ &\geq M^{(n)}(\theta_0,y) - ( M^{(n)}(\hat{\theta}^{(n)}_{\text{MLE}},y) - M^{(n)}(\mu^{(n)},y))
\end{aligned}
\end{equation}
Thus, if we show that 
\begin{equation}\label{claim}
 R^{(n)} = M^{(n)}(\hat{\theta}^{(n)}_{\text{MLE}},y) - M^{(n)}(\mu^{(n)},y) = o_{\mathbb{P}}(1)
\end{equation}
then the consistency of $\hat{\theta}(y)$ follows from \fullref{M-estimator}.
Consider any $\varepsilon >0$. We will consider two cases separately. 
First, assume we have that
\begin{equation}\label{caseone}
\begin{aligned}
&M^{(n)}(\hat{\theta}^{(n)}_{\text{MLE}},y) - M^{(n)}(\mu^{(n)},y) > \varepsilon \\
&\implies M^{(n)}(\hat{\theta}^{(n)}_{\text{MLE}},y)   > M^{(n)}(\mu^{(n)},y)+ \varepsilon   \\
&\implies M^{(n)}(\hat{\theta}^{(n)}_{\text{MLE}},y)  + \frac{1}{n} \ln p(\hat{\theta}^{(n)}_{\text{MLE}})  > M^{(n)}(\mu^{(n)},y) +  \frac{1}{n} \ln p(\hat{\theta}^{(n)}_{\text{MLE}}) +\varepsilon    \\
\end{aligned}
\end{equation} 
By the consistency of the MLE, we can assume that  $\hat{\theta}^{(n)}_{\text{MLE}}$ is in some  compact set $K$ in a neighborhood of $\theta_0$, where we will have that 
\begin{equation}
 -\infty < \underset{\theta \in K}{\inf} \ln p(\theta) \leq \ln p(\hat{\theta}^{(n)}_{\text{MLE}}),
\end{equation}
by the continuity of the prior and its positivity at $\theta_0$.
Thus, in this compact set, for large enough $n$, we have that
\begin{equation}
\frac{1}{n} \ln p(\hat{\theta}^{(n)}_{\text{MLE}}) +\varepsilon > \frac{1}{n} \ln p(\mu^{(n)}) 
\end{equation}
since $\ln p(\theta)$ is bounded from above on all of $\Theta$. Plugging this inequality into \eqref{caseone} would yield 
\begin{equation}
M^{(n)}(\hat{\theta}^{(n)}_{\text{MLE}},y)  + \frac{1}{n} \ln p(\hat{\theta}^{(n)}_{\text{MLE}})  > M^{(n)}(\mu^{(n)},y) +  \frac{1}{n} \ln p(\mu^{(n)}),
\end{equation} 
which is prohibited by definition of $\mu^{(n)}$ as a maximizer. Thus, for large $n$, the first case will only occur if   $\hat{\theta}^{(n)}_{\text{MLE}} $ is not in $K$, which has arbitrarily low probability.
On the other hand, we see that 
\begin{equation}
\begin{aligned}
&M^{(n)}(\mu^{(n)},y)-M^{(n)}(\hat{\theta}^{(n)}_{\text{MLE}},y) > \varepsilon \\
&\implies M^{(n)}(\mu^{(n)},y) > M^{(n)}(\hat{\theta}^{(n)}_{\text{MLE}},y),
\end{aligned}
\end{equation}
which is impossible by the definition of the MLE in \eqref{defmle}. Thus, our claim \eqref{claim} is shown and the consistency of the MAP-estimate is established.

\myparagraph{Proof of \fullref{prop2}}
As we stated before, this proof closely follows \cite{van2000asymptotic}, where it is employed to proof consistency of the MLE. Here, we only introduce slight modifications to adjust to the MAP setting.
We define
\begin{equation}
m_\theta(y) = \ln \frac{p(y|\theta)}{p(y|\theta_0)}
\end{equation}
and  note that 
\begin{equation}
M^{(n)}(\theta,y) = \sum_{i=1}^{n} m_\theta(y) \overset{a.s.}{\longrightarrow} \ \mathbb{E}_{\theta_0}[m_{\theta}(y)]  =  -\operatorname{D_{KL}}(p(y|\theta_0)||p(y|\theta))
\end{equation}
by the law of large numbers, which has an unique maximum at $\theta_0$, by our assumption.
The idea of the proof is to make the parameter space compact. That is, we define $\overline{\mathbb{R}} = \mathbb{R} \cup \{-\infty,\infty \}$. A possible metric for the above parameter space is $d(x,y):= |\arctan(x) - \arctan(y)|$. Furthermore define
\begin{equation}
\begin{aligned}
m_{-\infty}(y) &= \underset{\theta \to - \infty} {\lim}  m_\theta(y) =: c_{-}\\
m_{\infty}(y) &= \underset{\theta \to -\infty} {\lim}  m_\theta(y) =: c_{+}. \\
\end{aligned}
\end{equation}
The values of these limits are determined by the limiting behaviour of $f$. Note that by our identifiability assumption and continuity we have $\underset{\theta \to \pm \infty}{\lim} f(\theta) \neq f(\theta_0)$.
Furthermore, we admit the values $c_{-}$ = $c_{+}$ = $-\infty$, which will be the case if $f(\theta) \to +-\infty$.
Note that  in this case  $\mathbb{E}_{\theta_0}[m_{\infty}(y)] =-\infty$, which is allowed due to the formal definition of the expectation as a Lebesgue integral and due to some bound we will see below.
These values will however, not bother the discussion, as in this case the maximum will be taken at some finite value.
We now define random variables $m_U:\overline{\mathbb{R}} \to \mathbb{R} \cup \{-\infty\}$ by
\begin{equation}
m_U(y):= \underset{\theta \in U}{\sup} \ m_\theta(y),
\end{equation}
where $U$ is any open Ball in $\overline{\mathbb{R}}$. We emphasized that $\infty$ will not be a value these functions take, as due to the part of the likelihood varying with $\theta$ is always nonpositive.
Even more,  we have
\begin{equation}
\frac{p(y|\theta)}{p(y|\theta_0)} \leq \exp(\frac{1}{2 \sigma^2} ( y-f(\theta_0)^2)),
\end{equation}
Such that 
\begin{equation}
\mathbb{E}_{\theta_0}[m_U(y)] \leq \int   \frac{1}{2 \sigma^2} ( y-f(\theta_0)^2)  p(y|\theta_0) dy =   \frac{1}{2}.
\end{equation}
Here again, this expectation is allowed to take the value $-\infty$.
Furthermore, we assume that $m_U$ indeed are measurable, although this seems to be an assumption that can be worked around \citep{geyer2012wald}.
Consider any monotone decreasing sequence of radii $r_n$ that goes to 0. Continuity implies that 
\begin{equation}
m_{B_{r_n}(\theta)} \longrightarrow m_\theta,
\end{equation}
where $B_{r_n}(\theta)$ is an open ball centered on $\theta$ with radius $r_n$.  We have that $ m_{B_{r_n}(\theta)} \geq m_{B_{r_m}(\theta)}$ for all $m \geq n$, as taking the supremum over a smaller ball will not increase. Furthermore, the previous note showed that one can apply the monotone convergence theorem,
as we have a decreasing sequence of functions with integral strictly smaller than $\infty$.
That is, we get that
\begin{equation}
\underset{n \to \infty}{\lim} \mathbb{E}_{\theta_0}[m_{B_{r_n}(\theta)}] = \mathbb{E}_{\theta_0}[m_\theta].
\end{equation}
This step is important because of the following: Since the Kullback-Leibler divergence has an unique minimum at $\theta_0$, for $\theta \neq \theta_0$ we will have that $\mathbb{E}_{\theta_0}[m_\theta] < \mathbb{E}_{\theta_0}[m_{\theta_0}]$. The last step now shows us that we can extend this the supremum over some open ball around $\theta$,  that is we can always fix a radius $r$ small enough such that $ \mathbb{E}_{\theta_0}[m_{B_r(\theta)}] <  \mathbb{E}_{\theta_0}[m_{\theta_0}]$.
We now fix $\varepsilon$ and consider the compact set: $B_\varepsilon(\theta_0)^c : = \{\theta \in \overline{\mathbb{R}} : d(\theta,\theta_0) \geq \varepsilon\}$. Now we will apply standard results from calculus: We have our compact set  $B_\varepsilon(\theta_0)^c$ covered by all the open balls centered on values of $\theta$ in that compact set. That is
\begin{equation}
B_\varepsilon(\theta_0)^c  \subset \bigcup\limits_{\theta \in B_\varepsilon(\theta_0)^c} U_\theta.
\end{equation}
Due to compactness, there exists a finite number of Balls $U_{\theta_1}...U_{\theta_l}$, such that
\begin{equation}
B_\varepsilon(\theta_0)^c  \subset \bigcup \limits_{k = 1}^{l} U_{\theta_k},
\end{equation}
and we can take those balls to have radius smaller than $r$. We thus have that 
\begin{equation}\label{inequality}
\begin{aligned}
&\underset{\theta \in B_\varepsilon(\theta_0)^c}{\sup}   \frac{1}{n}\sum_{i=1}^{n} m_{\theta}(y_i) +\frac{1}{n}\ln p(\theta)  \\
&\leq \underset{\theta \in B_\varepsilon(\theta_0)^c}{\sup}  \frac{1}{n}\sum_{i=1}^{n} m_{\theta}(y_i)  + \underset{\theta \in B_\varepsilon(\theta_0)^c}{\sup} \frac{1}{n}\ln p(\theta) \\
& \leq \underset{k=1,...l}{\max} \   \frac{1}{n}\sum_{i=1}^{n} m_{U_{\theta_k}}(y_i)  + \frac{C}{n} \\,
\end{aligned}
\end{equation}
where the second line follows by splitting the supremum, and the third line follows by our definition of $m_U$ and by dropping negative terms in the logarithm of the normal prior, leaving some constant $C$.
Now, we have that 
\begin{equation}\label{supconv}
\underset{k=1,...l}{\max} \  \frac{1}{n}\sum_{i=1}^{n} m_{U_{\theta_k}}(y_i)  + \frac{C}{n} \overset{a.s.}\longrightarrow \underset{k=1,...l}{\max} \  \mathbb{E}[m_{U_{\theta_k}}] < \mathbb{E}_{\theta_0}[m_{\theta_0}]
\end{equation}
by the strong law of large numbers. Note that  we can safely apply the strong law of large numbers here regardless of the extension to our compact space, since every of these Balls will still include values from $\mathbb{R}$.
But, on the other hand, we also have that 
\begin{equation}\label {ineq_one}
(\mu^{(n)} \in B_\varepsilon(\theta_0)^c) \implies \underset{\theta \in B_\varepsilon(\theta_0)^c}{\sup}   \frac{1}{n}\sum_{i=1}^{n} m_{\theta}(y_i) +\frac{1}{n}\ln p(\theta)  \geq  \frac{1}{n}\sum_{i=1}^{n} m_{\mu^{(n)}}(y_i) +\frac{1}{n}\ln p(\mu^{(n)}) 
\end{equation}
And, by definition,
\begin{equation}\label{ineq_two}
\frac{1}{n}\sum_{i=1}^{n} m_{\mu^{(n)}}(y_i) +\frac{1}{n}\ln p(\mu^{(n)}) \geq \frac{1}{n}\sum_{i=1}^{n} m_{\theta_0}(y_i) +\frac{1}{n}\ln p(\theta_0) \overset{\mathbb{P}}\longrightarrow \mathbb{E}_{\theta_0}[m_{\theta_0}].
\end{equation}
These two convergences show that $\mathbb{P}[(\mu^{(n)} \in B_\varepsilon(\theta_0)^c] \to 0$. Formally, this can be seen by noting that if $\mu^{(n)} \in B_\varepsilon(\theta_0)^c$, for a sufficiently small $\delta >0$, one either has
\begin{equation}
\frac{1}{n}\sum_{i=1}^{n} m_{\theta_0}(y_i) +\frac{1}{n}\ln p(\theta_0)  > \mathbb{E}_{\theta_0}[m_{\theta_0}] -\delta .
\end{equation}
But then, by \eqref{ineq_one} and \eqref{ineq_two}
\begin{equation}
\underset{\theta \in B_\varepsilon(\theta_0)^c}{\sup}   \frac{1}{n}\sum_{i=1}^{n} m_{\theta}(y_i) +\frac{1}{n}\ln p(\theta)  >  \mathbb{E}_{\theta_0}[m_{\theta_0}] -\delta,
\end{equation}
which, by \eqref{inequality} and \eqref{supconv} will have arbitrary small probability for small enough $\delta$.
On the other hand, the probability of  
\begin{equation}
\frac{1}{n}\sum_{i=1}^{n} m_{\theta_0}(y_i) +\frac{1}{n}\ln p(\theta_0)  \leq \mathbb{E}_{\theta_0}[m_{\theta_0}] -\delta 
\end{equation}
goes to $0$ by the law of large numbers. Thus, we have  established that 
$\mathbb{P}[(\mu^{(n)} \in B_\varepsilon(\theta_0)^c] \to 0$ which is exactly the definition of $\mu^{(n)} \overset{\mathbb{P}}{\longrightarrow} \theta_0$.

\subsection{Linear Model}

\subsubsection{Variational Laplace Update Equations}\label{sec:linearupdates}
In the following, we derive the variational Laplace update scheme for the linear model.
The log-joint-probability reads, up to constants,
\begin{equation}
\begin{aligned}
\ln p(y,\theta,\lambda) &= -\frac{1}{2} \sum_{i=1}^{n} \frac{(y_i - a \theta)^2}{ \exp(\lambda)}  -\frac{n}{2} \ln \exp(\lambda) \\
&- \frac{1}{2} \frac{(\theta-m_\theta)^2}{s_\theta^2} \\
&- \frac{1}{2} \frac{ (\lambda - m_\lambda)^2}{s_\lambda^2}
\end{aligned}
\end{equation}
The second derivatives are
\begin{equation}
\begin{aligned}
&H_L^\theta = \frac{\partial^2 \ln p(y,\theta,\lambda)}{\partial^2  \theta} = - (\frac{a^2 n}{\exp(\lambda)} + \frac{1}{s_\theta^2}) = - \frac{a^2ns_\theta^2 + \exp(\lambda)}{\exp(\lambda)s_\theta^2} \\
&H_L^\lambda = \frac{\partial^2 \ln p(y,\theta,\lambda)}{\partial^2 \lambda} = -\frac{1}{2}( \frac{\sum_{i=1}^{n}(y_i -a\theta)^2}{\exp(\lambda)} + \frac{2}{s_\lambda^2}) \\
&= -\frac{1}{2}( \frac{\sum_{i=1}^{n}(y_i -a\theta)^2 s_\lambda^2 + 2 \exp(\lambda)}{\exp(\lambda)s_\lambda^2}).
\end{aligned}
\end{equation}
Thus, having found an optimal $\mu_\theta^*$, the variance parameter is obtained by
\begin{equation}
\sigma_\theta^* = -H_{L_y}^\theta(\mu_\theta^*,\mu_\lambda)^{-1}=  \frac{\exp(\mu_\lambda)s_\theta^2}{a^2ns_\theta^2 + \exp(\mu_\lambda)} .
\end{equation}
Likewise, having found an optimal $\mu_\lambda^*$, the variance parameter is obtained by
\begin{equation}
\sigma_\lambda^* = -H_{L_y}^\lambda(\mu_\theta,\mu_\lambda^*)^{-1} = \frac{2\exp(\mu_\lambda^*)s_\lambda^2}{\sum_{i=1}^{n}(y_i -a\theta)^2 s_\lambda^2 + 2 \exp(\mu_\lambda^*)}
\end{equation}
Hence, the "variational energies" are
\begin{equation}\label{Itheta}
\begin{aligned}
I(\mu_\theta) &=L_y(\mu_\theta,\mu_\lambda^*)  + \frac{1}{2} H_{L_y}^\lambda(\mu_\theta,\mu_\lambda^*) \sigma_\lambda^* \\
&=  -\frac{1}{2} \sum_{i=1}^{n} \frac{(y_i - a \mu_\theta)^2}{ \exp(\mu_\lambda^*)} - \frac{1}{2} \frac{(\mu_\theta-m_\theta)^2}{s_\theta^2}  -\frac{1}{4} ( \frac{\sum_{i=1}^{n}(y_i -a \mu_\theta)^2}{\exp(\mu_\lambda^*)} + \frac{2}{s_\lambda^2}) \sigma_\lambda^* +C
\end{aligned}
\end{equation}
and
\begin{equation}\label{Ilambda}
\begin{aligned}
I(\mu_\lambda) &= L_y(\mu_\theta^*,\mu_\lambda)  + \frac{1}{2} H_{L_y}^\theta(\mu_\theta^*,\mu_\lambda^*)\sigma_\theta^*  \\
		 & =  -\frac{1}{2} \sum_{i=1}^{n} \frac{(y_i - a \mu_\theta^*)^2}{ \exp(\mu_\lambda)}  -\frac{n}{2} \mu_\lambda - \frac{1}{2} \frac{ (\mu_\lambda - m_\lambda)^2}{s_\lambda^2}  -\frac{1}{2}(\frac{a^2 n}{\exp(\mu_\lambda)} + \frac{1}{s_\theta^2}) \sigma_\theta^* +C
\end{aligned}
\end{equation}
The advantage of this simple model is that the maxima of the above formulae can be stated in closed-forms. First, we find the maximum of $I(\mu_\theta)$. To do so, we calculate the derivative:
\begin{equation}
\begin{aligned}
\frac{\partial I(\mu_\theta)}{\partial \mu_\theta} =  (1+ \frac{\sigma_\lambda^*}{2})  \sum_{i=1}^{n} \frac{a(y_i - a \mu_\theta)}{\exp(\mu_\lambda^*)}  + \frac{m_\theta - \mu_\theta}{s_\theta^2}
\end{aligned}
\end{equation}
We proceed by setting this derivative to $0$
\begin{equation}
\begin{aligned}
&\frac{\partial I(\mu_\theta)}{\partial \mu_\theta} =0 \\
&\iff   (1+ \frac{\sigma_\lambda^*}{2})  \frac{n a^2 \mu_\theta}{\exp(\mu_\lambda^*)} + \frac{\mu_\theta}{s_\theta^2} =   (1+ \frac{\sigma_\lambda^*}{2}) a \frac{\sum_{i=1}^{n} y_i}{\exp(\mu_\lambda^*)}  + \frac{m_\theta}{s_\theta^2} \\
& \iff   \mu_\theta =  \frac{  (1+ \frac{\sigma_\lambda^*}{2}) a \frac{\sum_{i=1}^{n} y_i}{\exp(\mu_\lambda^*)}  + \frac{m_\theta}{s_\theta^2}}{ (1+ \frac{\sigma_\lambda^*}{2})  \frac{n a^2}{\exp(\mu_\lambda^*)} + \frac{1}{s_\theta^2}} \\
& \iff \mu_\theta = \frac{   a (1+ \frac{\sigma_\lambda^*}{2})  \frac{\sum_{i=1}^{n} y_i}{n}  + \frac{\exp(\mu_\lambda^*)m_\theta}{n s_\theta^2}}{ a^2   (1+ \frac{\sigma_\lambda^*}{2}) +\frac{\exp(\mu_\lambda^*)}{n s_\theta^2}}
\end{aligned}
\end{equation}
Likewise, the maximum of $I(\mu_\lambda)$ can be found by setting the derivative to $0$. We have 
\begin{equation}
\begin{aligned}
\frac{\partial I(\mu_\lambda)}{\partial \mu_\lambda} = \frac{1}{2}  \sum_{i=1}^{n} \frac{(y_i - a \mu_\theta^*)^2}{ \exp(\mu_\lambda)}  + \frac{1}{2} \frac{a^2 n}{\exp(\mu_\lambda)} \sigma_\theta^*  - \frac{\mu_\lambda - m_\lambda}{s_\lambda^2} -\frac{n}{2}
\end{aligned}
\end{equation}
Thus
\begin{equation}
\begin{aligned}
&\frac{\partial I(\mu_\lambda)}{\partial \mu_\lambda} =0 \\
&\iff \sum_{i=1}^{n} \frac{(y_i - a \mu_\theta^*)^2}{ \exp(\mu_\lambda)} + \frac{a^2 n}{\exp(\mu_\lambda)} \sigma_\theta^* = 2\frac{\mu_\lambda - m_\lambda}{s_\lambda^2} + n \\
&\iff  \exp(\mu_\lambda) (\frac{2\mu_\lambda}{s_\lambda^2} - \frac{2 m_\lambda}{s_\lambda^2} + n) =  \sum_{i=1}^{n} (y_i - a \mu_\theta^*)^2 +a^2 n \sigma_\theta^*
\end{aligned}
\end{equation}
Using \fullref{lambertlemma}, this has an unique solution given by
\begin{equation}
\begin{aligned}
\mu_\lambda &= W\left( \frac{s_\lambda^2}{2} \left( \sum_{i=1}^{n} (y_i - a \mu_\theta^*)^2+a^2 n \sigma_\theta^*\right) \exp\left( n \frac{s_\lambda^2}{2} - m_\lambda \right) \right) -  ( n \frac{s_\lambda^2}{2} - m_\lambda) \\
&= W\left( n\frac{s_\lambda^2}{2} \left( \frac{\sum_{i=1}^{n} (y_i - a \mu_\theta^*)^2}{n}+a^2  \sigma_\theta^*\right) \exp\left( n \frac{s_\lambda^2}{2} - m_\lambda \right) \right) -  ( n \frac{s_\lambda^2}{2} - m_\lambda),
\end{aligned} 
\end{equation}
where $W$ is Lambert's $W$-function that inverts the function $x \mapsto x \ \exp(x)$.
Thus, we can now state the full update equations for one iteration as follows: Given $(\mu_\theta,\sigma_\theta,\mu_\lambda,\sigma_\lambda)$ from a previous iteration, these read
\begin{equation}
\begin{aligned}
&\mu_\theta^* =  \frac{   a (1+ \frac{\sigma_\lambda}{2})  \frac{\sum_{i=1}^{n} y_i}{n}  + \frac{\exp(\mu_\lambda)m_\theta}{n s_\theta^2}}{ a^2   (1+ \frac{\sigma_\lambda}{2}) +\frac{\exp(\mu_\lambda)}{n s_\theta^2}}\\
&\sigma_\theta^* =\frac{\exp(\mu_\lambda)s_\theta^2}{a^2ns_\theta^2 + \exp(\mu_\lambda)} \\
& \mu_\lambda^* = W\left( n\frac{s_\lambda^2}{2} \left( \frac{\sum_{i=1}^{n} (y_i - a \mu_\theta^*)^2}{n}+a^2  \sigma_\theta^*\right) \exp\left( n \frac{s_\lambda^2}{2} - m_\lambda \right) \right) -  ( n \frac{s_\lambda^2}{2} - m_\lambda)\\
& \sigma_\lambda^* =  \frac{2 \exp(\mu_\lambda^*)s_\lambda^2}{\sum_{i=1}^{n}(y_i -a\theta)^2 s_\lambda^2 + 2 \exp(\mu_\lambda^*)}
\end{aligned}
\end{equation}
\subsubsection{Asymptotic results}
\myparagraph{Preliminary lemma}
To proof convergence in probability in the linear model, the following lemma is used repeatedly.
\begin{lemma}\label{fraclemma}
Suppose we have sequences of random variables $x^{(n)},y^{(n)},z^{(n)}$ such that  $x^{(n)} = o_\mathbb{P}(1)$,  $y^{(n)} \geq 0, z^{(n)} \geq 0$ and positive constants $c,d \in \mathbb{R}_{+}$. Then we have
$$\frac{x^{(n)}}{c(d + z^{(n)}) + y^{(n)}} = o_\mathbb{P}(1).$$
\end{lemma}
To see this, we note that 
\begin{equation}
\begin{aligned}
\left|\frac{x^{(n)}}{c(d + z^{(n)}) + y^{(n)}}\right|   = \frac{\left|x^{(n)}\right|}{\left|c(d + z^{(n)}) + y^{(n)}\right|}  = \frac{\left|x^{(n)}\right|}{c(d + z^{(n)}) + y^{(n)}} \leq \frac{\left|x^{(n)}\right|}{cd} \overset{\mathbb{P}}{\longrightarrow} 0,
\end{aligned}
\end{equation}
using the nonnegativity of the sequences $y^{(n)},z^{(n)}$ and the continuous mapping theorem.

\myparagraph{Consistency - proof of \fullref{prop3}}
Recall that we  are studying fixed points, that is our variational estimates are values  $ \mu_\theta^{(n)},\sigma_\theta^{(n)},\mu_\lambda^{(n)},\sigma_\lambda^{(n)}$, such that $F(\mu_\theta^{(n)},\sigma_\theta^{(n)},\mu_\lambda^{(n)},\sigma_\lambda^{(n)}) = (\mu_\theta^{(n)},\sigma_\theta^{(n)},\mu_\lambda^{(n)},\sigma_\lambda^{(n)})$. 
Plugging this into the update equations yields
\begin{equation}
\begin{aligned}
&\mu_\theta^{(n)} =  \frac{   a (1+ \frac{\sigma_\lambda^{(n)}}{2})  \frac{\sum_{i=1}^{n} y_i}{n}  + \frac{\exp(\mu_\lambda^{(n)})m_\theta}{n s_\theta^2}}{ a^2   (1+ \frac{\sigma_\lambda^{(n)}}{2}) +\frac{\exp(\mu_\lambda^{(n)})}{n s_\theta^2}}\\
&\sigma_\theta^{(n)} =  \frac{\exp(\mu_\lambda^{(n)})s_\theta^2}{a^2ns_\theta^2 + \exp(\mu_\lambda^{(n)})} \\
& \mu_\lambda^{(n)} = W\left( n\frac{s_\lambda^2}{2} \left( \frac{\sum_{i=1}^{n} (y_i - a \mu_\theta^{(n)})^2}{n}+a^2  \sigma_\theta^{(n)}\right) \exp\left( n \frac{s_\lambda^2}{2} - m_\lambda \right) \right) -  ( n \frac{s_\lambda^2}{2} - m_\lambda)\\
& \sigma_\lambda^{(n)} = 2 \frac{\exp(\mu_\lambda^{(n)})s_\lambda^2}{\sum_{i=1}^{n}(y_i -a\mu_\theta^{(n)})^2 s_\lambda^2 + 2 \exp(\mu_\lambda^{(n)})}.
\end{aligned}
\end{equation}
\mbox{}\\
Let us for now assume that  $\mu_\lambda^{(n)} = \mathcal{O}_{\mathbb{P}}(1)$, which will facilitate the following derivations. We will show that this assumption is indeed true at the end of the proof. By this property, we have
\begin{equation}\label{muthetaderivation}
\begin{aligned}
\mu_\theta^{(n)} &=  \frac{   a (1+ \frac{\sigma_\lambda^{(n)}}{2})  \frac{\sum_{i=1}^{n} y_i}{n}  + \frac{\exp(\mu_\lambda^{(n)})m_\theta}{n s_\theta^2}}{ a^2   (1+ \frac{\sigma_\lambda^{(n)}}{2}) +\frac{\exp(\mu_\lambda^{(n)})}{n s_\theta^2}} \\
			&= \frac{   a (1+ \frac{\sigma_\lambda^{(n)}}{2})  \frac{\sum_{i=1}^{n} y_i}{n} }{ a^2   (1+ \frac{\sigma_\lambda^{(n)}}{2}) +\frac{\exp(\mu_\lambda^{(n)})}{n s_\theta^2}}     +       
    \frac{ \frac{\exp(\mu_\lambda^{(n)})m_\theta}{n s_\theta^2}}{a^2   (1+ \frac{\sigma_\lambda^{(n)}}{2}) +\frac{\exp(\mu_\lambda^{(n)})}{n s_\theta^2}} \\
& = \frac{     \frac{\sum_i y_i}{n} }{ a (1   +\frac{\exp(\mu_\lambda^{(n)})}{n s_\theta^2  a^2(1+ \frac{\sigma_\lambda^{(n)}}{2})})} + \frac{ \frac{\exp(\mu_\lambda^{(n)})m_\theta}{n s_\theta^2}}{a^2   (1+ \frac{\sigma_\lambda^{(n)}}{2}) +\frac{\exp(\mu_\lambda^{(n)})}{n s_\theta^2}}  \overset{\mathbb{P}}{\longrightarrow}  \theta_0 .
\end{aligned}
\end{equation}
This can be seen as follows: Since $\exp(\mu_\lambda^{(n)}) = O_{\mathbb{P}}(1) \implies \frac{\exp(\mu_\lambda^{(n)})}{n} = o_{\mathbb{P}}(1)$ and $\frac{\sigma_\lambda^{(n)}}{2} \geq 0$, the second term converges in probability to $0$ by  \fullref{fraclemma}. For the first term, we use that $\frac{\exp(\mu_\lambda^{(n)})}{n s_\theta^2  a^2(1+ \frac{\sigma_\lambda^{(n)}}{2})} \overset{\mathbb{P}}{\longrightarrow} 0$ by  \fullref{fraclemma},
$\frac{\sum_i y_i}{n} \overset{\mathbb{P}}{\longrightarrow}  a\theta_0$ by the law of large numbers, and then apply  Slutsky's Lemma.
Furthermore, again by Slutsky's Lemma we have 
\begin{equation}
\sigma_\theta^{(n)} =  \frac{\frac{\exp(\mu_\lambda^{(n)})s_\theta^2}{n}}{a^2s_\theta^2 + \frac{\exp(\mu_\lambda^{(n)})}{n}}  \overset{\mathbb{P}}{\longrightarrow}  0
\end{equation}
For $\mu_\lambda^{(n)}$ we now can use that for $x \to \infty$, $W(x) = \ln x - \ln (\ln x)  + o(1)$ \citep{corless1996lambertw}.
By our previous limit considerations, we can rewrite $ \frac{\sum_{i=1}^{n} (y_i - a \mu_\theta^{(n)})^2}{n}+a^2  \sigma_\theta^{(n)} = \exp(\lambda_0) + o_{\mathbb{P}}(1)$.
This is because 
\begin{equation}
\begin{aligned}
&\frac{\sum_{i=1}^{n} (y_i - a \mu_\theta^{(n)})^2}{n} \\&= \frac{\sum_{i=1}^{n}y_i^2}{n} - 2a\mu_\theta^{(n)}\frac{\sum_{i=1}^{n}y_i}{n} + (a \mu_\theta^{(n)})^2 \\& \overset{\mathbb{P}}{\longrightarrow} a^2\theta_0^2 +\exp(\lambda_0)  - 2a^2\theta_0^2 + a^2\theta_0^2\\ &= \exp(\lambda_0)  
\end{aligned}
\end{equation}
by the law of large numbers and Slutsky's Lemma.

Hence, we see that the argument of Lambert's $W$ function in our expression for $\mu_\lambda^{(n)}$ will be arbitrarily big with arbitrarily high probability, such that the limit expression above is justified. We thus write 
\begin{equation}\label{mulambdaderivation}
\begin{aligned}
 \mu_\lambda^{(n)} &=    W\left( n\frac{s_\lambda^2}{2} \left(   \exp(\lambda_0) +  o_{\mathbb{P}}(1)  \right) \exp\left( n \frac{s_\lambda^2}{2} - m_\lambda \right) \right) -  ( n \frac{s_\lambda^2}{2} - m_\lambda)  \\
& = \ln  \left( n\frac{s_\lambda^2}{2} \left(   \exp(\lambda_0) +  o_{\mathbb{P}}(1)  \right) \exp\left( n \frac{s_\lambda^2}{2}  - m_\lambda \right) \right)  \\
 &- \ln \ln \left( n\frac{s_\lambda^2}{2} \left(   \exp(\lambda_0) +  o_{\mathbb{P}}(1)  \right) \exp\left( n \frac{s_\lambda^2}{2} - m_\lambda \right) \right) - ( n \frac{s_\lambda^2}{2} - m_\lambda) + o_\mathbb{P}(1) \\
& =     \ln   \frac{n\frac{s_\lambda^2}{2}    \exp(\lambda_0) +  o_{\mathbb{P}}(1)}{\ln \left( n\frac{s_\lambda^2}{2} \left(   \exp(\lambda_0) +  o_{\mathbb{P}}(1)  \right) \exp\left( n \frac{s_\lambda^2}{2} - m_\lambda \right) \right)} + o_\mathbb{P}(1) \\
& =  \ln   \frac{n\frac{s_\lambda^2}{2}    \exp(\lambda_0) +  o_{\mathbb{P}}(1)}{ n\frac{s_\lambda^2}{2} -   m_\lambda     +      \ln   n\frac{s_\lambda^2}{2}    \left(\exp(\lambda_0) +  o_{\mathbb{P}}(1)\right)   } + o_\mathbb{P}(1) \\
& =  \ln   \frac{  \exp(\lambda_0) +  o_{\mathbb{P}}(1)}{ 1 -  \frac {m_\lambda}{n\frac{s_\lambda^2}{2}}     +      \frac{\ln   n\frac{s_\lambda^2}{2} \left(   \exp(\lambda_0) +  o_{\mathbb{P}}(1)  \right) }{n\frac{s_\lambda^2}{2}}} + o_\mathbb{P}(1)
 \overset{\mathbb{P}}{\longrightarrow}  \ln \frac{\exp(\lambda_0)}{1} = \lambda_0,
\end{aligned} 
\end{equation}
where we used that $$ \frac {m_\lambda}{n\frac{s_\lambda^2}{2}}  +   \frac{\ln   n\frac{s_\lambda^2}{2} \left(   \exp(\lambda_0) +  o_{\mathbb{P}}(1)  \right) }{n\frac{s_\lambda^2}{2}}  \overset{\mathbb{P}}{\longrightarrow}0$$ since
$\underset{n \to \infty}{\lim} \frac{\ln n}{n} = 0$ and then applied Slutsky's Lemma to the fraction inside the logarithm. The continuity of the logarithm then allows the conclusion by the continuous mapping theorem.
Finally, we have that 
\begin{equation}
\begin{aligned}
\sigma_\lambda^{(n)} &= 2 \frac{\exp(\mu_\lambda^{(n)})s_\lambda^2}{\sum_{i=1}^{n}(y_i -a\mu_\theta^{(n)})^2 s_\lambda^2 + 2 \exp(\mu_\lambda^{(n)})} \\
&= 2 \frac{\frac{\exp(\mu_\lambda^{(n)})s_\lambda^2}{n}}{\frac{\sum_{i=1}^{n}(y_i -a\mu_\theta^{(n)})^2 s_\lambda^2}{n} +  \frac{2 \exp(\mu_\lambda^{(n)}}{n}}    \overset{\mathbb{P}}{\longrightarrow} 0
\end{aligned}
\end{equation}
again by Slutsky's Lemma.

\bigbreak
Let us now conclude the proof by showing the claim that $\mu_\lambda^{(n)} = \mathcal{O}_{\mathbb{P}}(1)$. 
We first observe that (using the same decomposition as in \eqref{muthetaderivation}) 
\begin{equation}
\begin{aligned}
\left\vert\mu_\theta^{(n)}\right\vert &= \left\vert \frac{     \frac{\sum_{i=1}^{n} y_i}{n} }{ a (1   +\frac{\exp(\mu_\lambda^{(n)})}{n s_\theta^2  a^2(1+ \frac{\sigma_\lambda^{(n)}}{2})})} + \frac{ \frac{\exp(\mu_\lambda^{(n)})m_\theta}{n s_\theta^2}}{a^2   (1+ \frac{\sigma_\lambda^{(n)}}{2}) +\frac{\exp(\mu_\lambda^{(n)})}{n s_\theta^2}}\right\vert\\
&\leq \left \vert  \frac{     \frac{\sum_{i=1}^{n} y_i}{n} }{ a (1   +\frac{\exp(\mu_\lambda^{(n)})}{n s_\theta^2  a^2(1+ \frac{\sigma_\lambda^{(n)}}{2})})} \right \vert + \left \vert  \frac{ \frac{\exp(\mu_\lambda^{(n)})m_\theta}{n s_\theta^2}}{a^2   (1+ \frac{\sigma_\lambda^{(n)}}{2}) +\frac{\exp(\mu_\lambda^{(n)})}{n s_\theta^2}}\right\vert\\
&\leq \left \vert \frac{ \frac{\sum_{i=1}^{n} y_i}{n}}{a} \right  \vert + \left \vert m_\theta \right \vert  \overset{\mathbb{P}}{\longrightarrow} \left\vert\theta_0\right\vert + \left\vert m_\theta\right\vert,
\end{aligned}
\end{equation}
which means that $\mu_\theta^{(n)}$ is bounded in probability, $\mu_\theta^{(n)}= \mathcal{O}_{\mathbb{P}}(1)$ . Note that we are using only the positivity of $\sigma_\lambda^{(n)}$  and  $\exp(\mu_\lambda^{(n)})$ without constraining $\mu_\lambda^{(n)}$.
We also observe that 
 \begin{equation}
\begin{aligned}
0 < \sigma_\theta^{(n)} =  \frac{\frac{\exp(\mu_\lambda^{(n)})s_\theta^2}{n}}{a^2s_\theta^2 + \frac{\exp(\mu_\lambda^{(n)})}{n}} \leq s_\theta^2\\
\end{aligned}
\end{equation}
Using the law of large numbers we have that
\begin{equation}
\begin{aligned}
&\frac{\sum_{i=1}^{n} (y_i - a \mu_\theta^{(n)})^2}{n} \\
&=\frac{\sum_{i=1}^{n} y_i^2}{n} - 2 a \mu_\theta^{(n)} \frac{\sum_{i=1}^{n} y_i}{n} + (a \mu_\theta^{(n)})^2\\
&= \exp(\lambda_0) + (a(\theta_0-\mu_\theta^{(n)}))^2  + o_{\mathbb{P}}(1).
\end{aligned}
\end{equation}
By our previous consideration, the term $ (a(\theta_0-\mu_\theta^{(n)}))^2 = \mathcal{O}_{\mathbb{P}}(1)$, and in particular this term is nonnegative.
These observations allow us to rewrite  
\begin{equation}
 \frac{\sum_{i=1}^{n} (y_i - a \mu_\theta^{(n)})^2}{n}+a^2  \sigma_\theta^{(n)} = \exp(\lambda_0) + r^{(n)} +  o_{\mathbb{P}}(1)
\end{equation}
with  $r^{(n)} = \mathcal{O}_\mathbb{P}(1)$ and $r^{(n)} >0$. In particular, this means that for any $\varepsilon >0$ we can find constants $0<m_\varepsilon \leq M_\varepsilon$ and $N_\varepsilon$ 
such that
\begin{equation}\label{bounds}
\exp(\lambda_0) + r^{(n)} +  o_{\mathbb{P}}(1) \ \in  \ \left[m_\varepsilon,M_\varepsilon \right]
\end{equation}
with probability greater that $1-\varepsilon$. 
In particular, we see that the argument of $W$  in the defintion of $\mu_\lambda^{(n)}$ can be written as 
\begin{equation}
n\frac{s_\lambda^2}{2} \left(\exp(\lambda_0) + r^{(n)} +  o_{\mathbb{P}}(1)\right) \exp\left( n \frac{s_\lambda^2}{2} - m_\lambda \right) 
\end{equation}
which will be arbitrarily big with arbitrarily high probability, such that the same limit expression as in \eqref{mulambdaderivation} can be employed to yield 
\begin{equation}\label{limit}
\mu_\lambda^{(n)} = \ln   \frac{  \exp(\lambda_0) +  r^{(n)}+  o_{\mathbb{P}}(1)}{ 1 -  \frac {m_\lambda}{n\frac{s_\lambda^2}{2}}     +      \frac{\ln   n\frac{s_\lambda^2}{2} \left(   \exp(\lambda_0) + r^{(n)}+  o_{\mathbb{P}}(1)  \right) }{n\frac{s_\lambda^2}{2}}} + o_\mathbb{P}(1).
\end{equation}
Here, we note that 
\begin{equation}\label{expr}
\frac{\ln   n\frac{s_\lambda^2}{2} \left(   \exp(\lambda_0) + r^{(n)}+  o_{\mathbb{P}}(1)  \right) }{n\frac{s_\lambda^2}{2}} \overset{\mathbb{P}}{\longrightarrow} 0.
\end{equation}
This is because using the monotonicity of the logarithm, one can upper and lower bound the expression by plugging in the positive bounds in \eqref{bounds}, where both will yield convergence to $0$. 
Thus, one has that in equation \eqref{limit},
\begin{equation}
1 -  \frac {m_\lambda}{n\frac{s_\lambda^2}{2}}     +      \frac{\ln   n\frac{s_\lambda^2}{2} \left(   \exp(\lambda_0) + r^{(n)}+  o_{\mathbb{P}}(1)  \right) }{n\frac{s_\lambda^2}{2}} \overset{\mathbb{P}}{\longrightarrow} 1
\end{equation}
Using \eqref{bounds} again for the enumerator, we thus have that 
\begin{equation}
\frac{  \exp(\lambda_0) +  r^{(n)}+  o_{\mathbb{P}}(1)}{ 1 -  \frac {m_\lambda}{n\frac{s_\lambda^2}{2}}     +      \frac{\ln   n\frac{s_\lambda^2}{2} \left(   \exp(\lambda_0) + r^{(n)}+  o_{\mathbb{P}}(1)  \right) }{n\frac{s_\lambda^2}{2}}}
\end{equation}
is bounded in probability and in particular positive with arbitrarily high probability. 
These considerations show that the whole expression in equation \eqref{limit} is bounded in probability, that is
\begin{equation}
\mu_\lambda^{(n)} = \mathcal{O}_{\mathbb{P}}(1)
\end{equation}
as claimed, which concludes the proof.

Our proof of consistency was rather explicit by using the analytic expressions for the update equations. As an alternative, one could also employ a more implicit proof by using \autoref{Z-estimator} after establishing boundedness in probability of the variational means.

\myparagraph{Distributional limits - proof of \autoref{prop:asymptoticefficiency}}
Let us now consider the distributional limits.  Here, first observe that we have by assumption 
\begin{align}
0 = \frac{\partial}{\partial \theta}\left( \ln p(y,\theta,\lambda) - \frac{1}{2} \frac{H_{L_y}^\lambda(\theta,\lambda)}{H_{L_y}^\lambda(\mu_\theta^{(n)},\mu_\lambda^{(n)})} \right) \bigg \rvert_{(\mu_\theta^{(n)},\mu_\lambda^{(n)}} \\
0 = \frac{\partial}{\partial \lambda}\left( \ln p(y,\theta,\lambda) - \frac{1}{2} \frac{H_{L_y}^\theta(\theta,\lambda)}{H_{L_y}^\theta(\mu_\theta^{(n)},\mu_\lambda^{(n)})} \right) \bigg \rvert_{(\mu_\theta^{(n)},\mu_\lambda^{(n)}} \\
\end{align}
since $\mu_\lambda^{(n)},\mu_\theta^{(n)}$ are maximizers of \eqref{Ilambda},\eqref{Itheta} and by plugging in the identities for the covariance parameters $\sigma_\lambda^{(n)} = H_{L_y}^\theta(\mu_\theta^{(n)},\mu_\lambda^{(n)})$,  $\sigma_\theta^{(n)} = H_{L_y}^\lambda(\mu_\theta^{(n)},\mu_\lambda^{(n)})$
 Thus, by multiplying both sides by $\frac{1}{n}$ and summarizing the above in a vector valued equation we have 
\begin{equation}
0 = \Psi^{(n)}\left(\mu_\theta^{(n)},\mu_\lambda^{(n)}\right)  + R^{(n)}\left( \mu_\theta^{(n)},\mu_\lambda^{(n)}\right)
\end{equation}
where 
\begin{equation}
\Psi^{(n)}\left(\mu_\theta,\mu_\lambda \right) = \frac{1}{n} \begin{pmatrix} \sum \limits_{i} \frac{\partial}{\partial \lambda}\ln p(y_i|\theta,\lambda) \bigg \rvert_{(\mu_\theta,\mu_\lambda)} \\ \sum \limits_{i} \frac{\partial}{\partial \theta}\ln p(y_i|\theta,\lambda) \bigg \rvert_{(\mu_\theta,\mu_\lambda)}  \end{pmatrix}
\end{equation}
and 
\begin{equation}
 R^{(n)}\left( \mu_\theta,\mu_\lambda \right) = \frac{1}{n}
 \begin{pmatrix}\frac{\partial}{\partial \theta}\left(  \ln p(\theta) + \frac{1}{2} \frac{H_{L_y}^\lambda(\theta,\lambda)}{H_{L_y}^\lambda(\mu_\theta^{(n)},\mu_\lambda^{(n)})} \right) \bigg \rvert_{(\mu_\theta,\mu_\lambda)} \\ \frac{\partial}{\partial \lambda}\left(  \ln p(\lambda) + \frac{1}{2} \frac{H_{L_y}^\theta(\theta,\lambda)}{H_{L_y}^\theta(\mu_\theta^{(n)},\mu_\lambda^{(n)})} \right) \bigg \rvert_{(\mu_\theta,\mu_\lambda)} \end{pmatrix}.
\end{equation}
Now, the first part, $\Psi^{(n)}\left(\mu_\theta^{(n)},\mu_\lambda^{(n)}\right) $ is the sample mean of the derivative of the log likelihood, i.e. the score, in a regular model, evaluated at the variational expectation parameters.
One thus would expect that the zeros of the above equations, if sufficiently close to zeros of the mean score, behave asymptotically equivalent. 
Theorem 5.21 in \citet{van2000asymptotic} now can be used to make 'sufficiently close' precise: Since the log-likelihood is twice continuously differentiable, we can get a local Lipschitz-condition and hence for this theorem to hold it will be enough to show that $\sqrt{n} R^{(n)}\left(\mu_\theta^{(n)},\mu_\lambda^{(n)}\right) = o_\mathbb{P}(1)$, compare also \eqref{varlapnearzero}.
We first observe that 
\begin{equation}
\frac{\sqrt{n} }{n}\frac{\partial}{\partial \lambda} \ln p(\lambda)\rvert_{(\mu_\theta^{(n)},\mu_\lambda^{(n)}} = \frac{\mu_\lambda^{(n)} - m_\lambda}{s_\lambda \sqrt{n}} = o_\mathbb{P}(1)
\end{equation}
and 
\begin{equation}
\frac{\sqrt{n} }{n}\frac{\partial}{\partial \theta} \ln p(\theta)\rvert_{(\mu_\theta^{(n)},\mu_\lambda^{(n)}} = \frac{\mu_\theta^{(n)} - m_\theta}{s_\theta \sqrt{n}} = o_\mathbb{P}(1)
\end{equation}
by Slutsky's Lemma. For the terms depending on the second derivatives we get
\begin{equation}
\begin{aligned}
&\frac{\sqrt{n} }{n}\frac{\partial}{\partial \theta}\left(\frac{H_{L_y}^\lambda(\theta,\lambda)}{H_{L_y}^\theta(\mu_\theta^{(n)},\mu_\lambda^{(n)})} \right) \bigg \rvert_{(\mu_\theta^{(n)},\mu_\lambda^{(n)}} \\
&= \frac{1}{\sqrt{n}}  \frac{( a \sum_{i=1}^{n}(y_i -a\mu_\theta^{(n)}) s_\lambda^2 }{\sum_{i=1}^{n}(y_i -a\mu_\theta^{(n)})^2 s_\lambda^2 + 2 \exp(\mu_\lambda^{(n)})}\\ 
&= \frac{1}{\sqrt{n}}  \frac{( \frac{a \sum_{i=1}^{n}(y_i -a\mu_\theta^{(n)})}{n} s_\lambda^2 }{\frac{\sum_{i=1}^{n}(y_i -a\mu_\theta^{(n)})^2}{n} s_\lambda^2 + 2 \frac{\exp(\mu_\lambda^{(n)})}{n} } \\
&= o_\mathbb{P}(1),
\end{aligned}
\end{equation}
since $\frac{a \sum_{i=1}^{n}(y_i -a\mu_\theta^{(n)})}{n} \overset{\mathbb{P}}{\longrightarrow} 0$, $\frac{\sum_{i=1}^{n}(y_i -a\mu_\theta^{(n)})^2}{n}\overset{\mathbb{P}}{\longrightarrow} \exp(\lambda_0)$ and  $2 \frac{\exp(\mu_\lambda^{(n)})}{n}   \overset{\mathbb{P}}{\longrightarrow} 0$ and then applying Slutsky's Lemma.
We also have
\begin{equation}
\begin{aligned}
&\frac{\sqrt{n} }{n}\frac{\partial}{\partial \lambda}\left(\frac{H_{L_y}^\theta(\theta,\lambda)}{H_{L_y}^\theta(\mu_\theta^{(n)},\mu_\lambda^{(n)})} \right) \bigg \rvert_{(\mu_\theta^{(n)},\mu_\lambda^{(n)}} \\
&= \frac{\sqrt{n}}{n} \frac{a^2 n}{\exp(\mu_\lambda^{(n)})} \frac{\exp(\mu_\lambda^{(n)})s_ \theta^2}{a^2ns_\theta^2 + \exp(\mu_\lambda^{(n)})}\\
&=\frac{1}{\sqrt{n} + \frac{\exp(\mu_\lambda^{(n)})}{a^2 \sqrt{n} s_\theta^2}} = o_\mathbb{P}(1).
\end{aligned}
\end{equation}
Thus, we can make use of the aforementioned theorem: we get that the limit distribution will be the same as for a true zero of the score -  that is, a maximum likelihood estimator. This estimator we know to be asymptotically efficient. We thus have
\begin{equation}
\sqrt{n} \left( \begin{pmatrix} \mu_\theta^{(n)} \\ \mu_\lambda^{(n)} \end{pmatrix} - \begin{pmatrix} \theta_0 \\ \lambda_0 \end{pmatrix} \right)  \overset{d}{\longrightarrow} G \sim \mathcal{N}\left(0,\mathcal{I}(\theta_0,\lambda_0)^{-1} \right)
\end{equation}
as claimed.

The limits of the rescaled variational variances are easy to see: we get that
\begin{equation}
n\sigma_\theta^{(n)} =  \frac{\exp(\mu_\lambda^{(n)})s_\theta^2}{a^2s_\theta^2 + \frac{\exp(\mu_\lambda^{(n)})}{n}} \overset{\mathbb{P}}{\longrightarrow} \frac{\exp(\lambda_0)}{a^2}
\end{equation}
and
\begin{equation}\label{sigmaconv}
 n\sigma_\lambda^{(n)} = 2 \frac{\exp(\mu_\lambda^{(n)})s_\lambda^2}{ \frac{\sum_{i=1}^{n}(y_i -a\mu_\theta^{(n)})^2}{n} s_\lambda^2 + 2 \frac{\exp(\mu_\lambda^{(n)})}{n}} \overset{\mathbb{P}}{\longrightarrow} 2 \frac{\exp{\lambda_0}}{\exp{\lambda_0}} = 2
\end{equation}
by using Slutsky's Lemma in both cases.
The Fisher information matrix in this case can be easily expressed as 
\begin{equation}
\begin{aligned}
\mathcal{I}(\theta_0,\lambda_0)  &= \mathbb{E}_{p(y|\theta_0,\lambda_0)}\left[ -\begin{pmatrix} \frac{\partial^2}{\partial^2 \theta} {\ln p(y|\theta,\lambda)} \bigg \rvert_{\theta,\lambda = \theta_0,\lambda_0} & \frac{\partial^2}{\partial \theta \partial \lambda} {\ln p(y|\theta,\lambda)} \bigg \rvert_{\theta,\lambda = \theta_0,\lambda_0} \\ \frac{\partial^2}{\partial \theta \partial \lambda} {\ln p(y|\theta,\lambda)} \bigg \rvert_{\theta,\lambda = \theta_0,\lambda_0} & \frac{\partial^2}{\partial^2 \lambda} {\ln p(y|\theta,\lambda)} \bigg \rvert_{\theta,\lambda = \theta_0,\lambda_0} \end{pmatrix} \right] \\
&= \mathbb{E}_{p(y|\theta_0,\lambda_0)}\left[ -\begin{pmatrix} -\frac{a^2}{\exp(\lambda_0)} &   \frac{2(y-a\theta_0)a}{\exp(\lambda_0)} \\   \frac{2(y-a\theta_0)a}{\exp(\lambda_0)}  & - \frac{1}{2}\frac{(y-a\theta_0)^2}{\exp(\lambda_0)} \end{pmatrix} \right]\\
&= \begin{pmatrix} \frac{a^2}{\exp(\lambda_0)} &   0 \\   0  &  \frac{1}{2} \end{pmatrix}.\\
\end{aligned}
\end{equation}
\myparagraph{Convergence in total variation - proof of \fullref{variationconvergence}}
Given the asymptotic efficiency of the variational means and the convergence in \eqref{sigmaconv} it is straightforward to see that the density obtained by variational Laplace and the optimal density in terms of  \eqref{KL_opt} converge in total variation. First note that for the optimal Gaussian in terms of \eqref{optimalGaussian} we have 
\begin{equation}\begin{aligned} g^{(n)}\begin{pmatrix}\theta \\ \lambda \end{pmatrix} &= \underset{q \in \mathcal{Q}^d}{\arg \min} \  \operatorname{D_{KL}}\left(q\begin{pmatrix}\theta \\ \lambda \end{pmatrix} \bigg\Vert \mathcal{N}\left(\begin{pmatrix}\theta \\ \lambda \end{pmatrix} ;\begin{pmatrix}\hat{\theta}^{(n)} \\ \hat{\lambda}^{(n)} \end{pmatrix} ,\frac{1}{n} \mathcal{I}\begin{pmatrix}\theta_0 \\ \lambda_0 \end{pmatrix} ^{-1}\right)\right)\\ &= \mathcal{N}\left(\begin{pmatrix}\theta \\ \lambda \end{pmatrix} ;\begin{pmatrix}\hat{\theta}^{(n)} \\ \hat{\lambda}^{(n)} \end{pmatrix} ,\frac{1}{n} \mathcal{I}\begin{pmatrix}\theta_0 \\ \lambda_0 \end{pmatrix} ^{-1}\right), \end{aligned}\end{equation} 
since this Gaussian is in the variational family, such that also 
\begin{equation}
\widetilde{g}_h^{(n)} =\mathcal{N}\left(h; \begin{pmatrix} \sqrt{n}\left(\hat{\theta}^{(n)}- \theta_0\right) \\ \sqrt{n}\left(\hat{\lambda}^{(n)}-\lambda_0\right) \end{pmatrix}, \mathcal{I}\begin{pmatrix}\theta_0 \\ \lambda_0 \end{pmatrix} ^{-1}\right).\end{equation}
By the triangle inequality one then has
\begin{equation}
\left \Vert \widetilde{q}_{\widetilde{\mu}^{(n)},\widetilde{\Sigma}^{(n)}}(h) - \widetilde{q}^{*(n)}_h(h)  \right \Vert_{TV} \leq \left \Vert \widetilde{q}_{\widetilde{\mu}^{(n)},\widetilde{\Sigma}^{(n)}}(h)  -\widetilde{g}_h^{(n)} \right \Vert_{TV} + \left \Vert \widetilde{q}^{*(n)}_h(h) - \widetilde{g}_h^{(n)}  \right \Vert_{TV} \overset{\mathbb{P}}{\longrightarrow} 0,
\end{equation}
where this follows since the first term in the sum converges to $0$ in probability by \eqref{variationalbernstein} and the second converges to $0$ in probability by \autoref{totalvariationconvergencelemma} below.
The same line of reasoning gives that both converge in total variation to the posterior by using the Bernstein-von-Mises theorem.

\begin{lemma}\label{totalvariationconvergencelemma}
Assume the maps $ \theta \mapsto  \sqrt{p(y|\theta)}$ are continuously differentiable for all $y$. Furthermore, assume $\theta \mapsto \mathcal{I}(\theta)$ exists; is continuous in $\theta$ and $\mathcal{I}(\theta_0)$ is invertible.    
Let  $\mu^{(n)}$,$\hat{\theta}^{(n)}$ be two asymptotically efficient estimators. Let $\widetilde{\Sigma}$ be a symmetric, positive definite matrix and $\widetilde{\Sigma}^{(n)}$ be a sequence of symmetric, positive definite, random matrizes, such that $\widetilde{\Sigma}^{(n)} \overset{\mathbb{P}}{\longrightarrow} \Sigma$.
Then $$ \left \Vert \mathcal{N}\left(\cdot;\sqrt{n}\left(\mu^{(n)}-\theta_0\right),\widetilde{\Sigma}^{(n)}\right) - \mathcal{N}\left(\cdot;\sqrt{n}\left(\hat{\theta}^{(n)}-\theta_0\right),\widetilde{\Sigma}\right) \right \Vert_{TV} \overset{\mathbb{P}}{\longrightarrow} 0 $$.
\end{lemma}
\begin{proof} To prove this Lemma, we first note that our assumptions establish 'differentiability in quadratic mean' by Lemma 7.6 in \cite{van2000asymptotic}. This property in turn, using Lemma 8.14 in \cite{van2000asymptotic}, ensures  that two rescaled, centered, asymptotically efficient estimator sequences necessarily converge in probability. That is, we have
\begin{equation}\label{eff_equivalence} \sqrt{n}(\mu^{(n)} - \theta_0) - \sqrt{n}(\hat{\theta}^{(n)}-\theta_0) \overset{\mathbb{P}}{\longrightarrow} 0.\end{equation}
Furthermore, we will use Pinsker's Inequality \citep{massart2007concentration}:  For any two probability measures $P,Q$ with densities $p,q$ we have that 
\begin{equation} \left \Vert P - Q \right \Vert_{TV} \leq \sqrt{\frac{1}{2} \operatorname{D_{KL}}(p||q)}. \end{equation}
Now one can leverage the fact that the Kullback-Leibler Divergence between two Gaussians has a closed form expression \citep{murphy2012machine}. We have
\begin{equation}
\begin{aligned}
&2 \ \operatorname{D_{KL}} \left(  \mathcal{N}\left(\cdot;\sqrt{n}\left(\mu^{(n)}-\theta_0\right),\widetilde{\Sigma}^{(n)}\right) ||  \mathcal{N}\left(\cdot;\sqrt{n}\left(\hat{\theta}^{(n)}-\theta_0\right) ,\widetilde{\Sigma}\right) \right) = 
\\  & \left(\operatorname{tr} \left( \widetilde{\Sigma}^{-1}\widetilde{\Sigma}^{(n)} \right) - p \right) \\&
  + \left(\sqrt{n} (\mu^{(n)}-\theta_0) - \sqrt{n}(\hat{\theta}^{(n)}-\theta_0) \right)^T  \widetilde{\Sigma}^{-1} \left(\sqrt{n} (\mu^{(n)}-\theta_0) - \sqrt{n}(\hat{\theta}^{(n)}-\theta_0) \right)  \\&+ \ln \frac{\det \widetilde{\Sigma}}{\det \widetilde{\Sigma}^{(n)}}.
\end{aligned}
\end{equation}
Using \eqref{eff_equivalence}, the convergence of $\widetilde{\Sigma}^{(n)}$  together with the continuos mapping theorem/Slustky's lemma, we get that all three terms of the above sum converge to $0$ in probability, which concludes the proof.
\end{proof}

\subsection{General Asymptotics of Variational Laplace}

\subsubsection{Proof of  \fullref{varlaptheorem}}
To prove \fullref{varlaptheorem}, recall that we consider fixed points $ \mu^{(n)} =  \begin{pmatrix} \mu_1^{(n)},...,\mu_k^{(n)} \end{pmatrix}$, that satisfy
\begin{equation}
\mu_j^{(n)} = \underset{\theta_j}{\arg \max} \  \ { \ln p(y,\theta_j,\mu_{/j}^{(n)})} +  \frac{1}{2} \sum_{i\neq j} \operatorname{tr} \left[ H_{L_y}^{\theta_i}\left(\mu_j^{(n)},\mu_{/j}^{(n)}\right)^{-1}   H_{L_y}^{\theta_i}  \left(\theta_j,\mu_{/j}^{(n)}\right)    \right]
\end{equation}
By assumption we have that $\mu_j^{(n)}$ is a maximum on an open set. Thus, necessarily, the derivatives at this point vanish. We thus have that
\begin{equation}
0 = \frac{\partial }{\partial \theta_j} \ln p(y,\theta_j,\theta_{/j}) \bigg \rvert_{ \left(\mu_j^{(n)},\mu_{/j}^{(n)}\right)} +   \frac{1}{2} \sum_{i\neq j} \frac{\partial}{\partial \theta_j} \operatorname{tr} \left[ H_{L_y}^{\theta_i}\left(\mu_j^{(n)},\mu_{/j}^{(n)}\right)^{-1}   H_{L_y}^{\theta_i}  \left(\theta_j,\theta_{/j}\right)    \right] \bigg \rvert_{ \left(\mu_j^{(n)},\mu_{/j}^{(n)}\right)}
\end{equation}
We note that the multiplication with a scalar will not affect this value. We thus also have that
\begin{equation}
\begin{aligned}
0 = &\frac{1}{n} \frac{\partial }{\partial \theta_j} \ln p(y|\theta_j,\theta_{/j}) \bigg \rvert_{ \left(\mu_j^{(n)},\mu_{/j}^{(n)}\right)} \\
 &+\frac{1}{n}\left( \frac{\partial}{\partial \theta_j} \ln p(\theta_j,\theta_{/j}) \bigg \rvert_{ \left(\mu_j^{(n)},\mu_{/j}^{(n)}\right)} -  \frac{1}{2} \sum_{i\neq j} \frac{\partial}{\partial \theta_j} \operatorname{tr} \left[ H_{L_y}^{\theta_i}\left(\mu_j^{(n)},\mu_{/j}^{(n)}\right)^{-1}   H_{L_y}^{\theta_i}  \left(\theta_j,\mu_{/j}^{(n)}\right)    \right]\bigg \rvert_{ \left(\mu_j^{(n)},\mu_{/j}^{(n)}\right)}  \right)
\end{aligned}
\end{equation}
We see that this equation has the following structure: The first term is the derivative of the likelihood ratio. In an identifiable model, under regularity conditions, this will asymptotically be equal to the derivative of the Kullback-Leibler divergence,which  an unique zero at $\theta_0$. The second term is an remainder term depending on higher order derivatives and the prior.
Thus, if this second term is asymptotically negligible in the appropriate sense, we will get consistency. Let us formalize these notions. We will define vector-valued random functions $\Psi^{(n)}$,  where the components are given by 
\begin{equation}
\Psi^{(n)}(y,\mu_j,\mu_{/j})_j =  \frac{1}{n}  \frac{\partial }{\partial \theta_j} \ln p(y|\theta_j,\theta_{/j}) \bigg \rvert_{ \left(\mu_j,\mu_{/j}\right)}
\end{equation} 
using a more compact notation, we can write this as
\begin{equation}
\Psi^{(n)}(y,\mu) =\frac{1}{n} \frac{\partial \ln p(y|\theta)}{\partial \theta}  (\mu)
\end{equation}
where  $\mu = (\mu_1,...,\mu_k)$.
Let us furthermore define the limit-function $\Psi$ as 
\begin{equation}
\Psi(\mu) = - \frac{\partial}{\partial \theta}  \operatorname{D_{KL}}\left( p(y|\theta_0)|| p(y|\theta)\right) (\mu).
\end{equation}
By the law of large numbers and our assumptions we have 
\begin{equation}
\begin{aligned}
\Psi^{(n)}(y,\mu) \overset{\mathbb{P}}{\longrightarrow} &\int \frac{\partial \ln p(y|\theta)}{\partial \theta}(\mu) p(y|\theta_0) dy \\&= \frac{\partial \int  \ln p(y|\theta)  p(y|\theta_0) dy }{\partial \theta}(\mu)\\
 &= - \frac{\partial \operatorname{D_{KL}}(p(y|\theta_0)||p(y|\theta)) }{\partial \theta} (\mu) = \Psi(\mu),
\end{aligned}
\end{equation}
where the first line follows by exchanging the order of derivative and integration and the last line follows since addition of a constant does not change the derivative.
Also by assumption, we will have that $\Psi$ has an unique, well separated zero at $\theta_0$. 
This corresponds to the first half of our equation above and will establish the consistency of any "near zero" of $\Psi^{(n)}$. To show that $\mu^{(n)}$ is such a "near zero", we recall that we defined a vector-valued random function $R^{(n)}$, where the components are given by
\begin{equation}
R^{(n)}(y,\mu_j,\mu_{/j})_j  = \frac{1}{n}\left( -\frac{\partial}{\partial \theta_j} \ln p(\theta_j,\theta_{/j}) \bigg \rvert_{ \left(\mu_j^{(n)},\mu_{/j}^{(n)}\right)} +  \frac{1}{2} \sum_{i\neq j} \frac{\partial}{\partial \theta_j} \operatorname{tr} \left[ H_{L_y}^{\theta_i}\left(\mu_j^{(n)},\mu_{/j}^{(n)}\right)^{-1}   H_{L_y}^{\theta_i}  \left(\theta_j,\mu_{/j}^{(n)}\right)    \right] \right).
\end{equation}
Now  by assumption we have
\begin{equation}\label{varlapnearzero}
0 = \Psi(y,\mu^{(n)}) - R^{(n)}(y,\mu^{(n)}) \\
\implies  \Psi(y,\mu^{(n)}) = R^{(n)}(y,\mu^{(n)})  =  o_{\mathbb{P}}(1).
\end{equation}
Hence, the consistency follows by using  \fullref{Z-estimator}.
The asymptotic efficiency then follows by using theorem 5.21 in \citet{van2000asymptotic}: First note that we get a local Lipschitz bound since we assumed the log-likelihood to be sufficiently smooth. Then, under our assumptions, this theorem yields that 
\begin{equation}
\sqrt{n}(\mu^{(n)} - \theta_0)   \overset{d}{\longrightarrow} G \sim \mathcal{N}\left(0, V_{\theta_0}^{-1} \mathbb{E}_{\theta_0} \left[ \frac{\partial \ln p(y|\theta)}{\partial \theta}\left(\frac{\partial \ln p(y|\theta)}{\partial \theta}  \right)^T \bigg\rvert_{\theta_0}\right] V_{\theta_0}^{-1} \right)
\end{equation}
where 
\begin{equation}
V_{\theta_0} = \frac{\partial}{\partial \theta} \left( \int \frac{\partial}{\partial \theta^T} p(y|\theta) dy \right) \bigg\rvert_{\theta_0}
\end{equation}
Now by the assumed interchange of derivation and integration, as in \eqref{fisherconditions}, we get that $V_{\theta_0} =  - \mathcal{I}(\theta_0)$. The term in the middle is by definition equal to $\mathcal{I}(\theta_0)$. Hence,
\begin{equation}
V_{\theta_0}^{-1} \mathbb{E}_{\theta_0} \left[ \frac{\partial \ln p(y|\theta)}{\partial \theta}\left(\frac{\partial \ln p(y|\theta)}{\partial \theta}  \right)^T \bigg\rvert_{\theta_0}\right] V_{\theta_0}^{-1} = \mathcal{I}(\theta_0)^{-1},
\end{equation}
which concludes the proof.

\subsubsection{Variational Laplace in Natural Exponential Family Models}\label{sec:expfam}
Here we want to discuss the applicability of our general results on variational Laplace to natural exponential family models . We will first state some properties of such models. Exponential families are characterized by the density taking the form
\begin{equation}
\begin{aligned}
&p(y|\theta) = g(\theta) \exp \left(\theta^T T(y)\right)\\
&g(\theta)^{-1} = \int \exp \left(\theta^T T(y)\right) dy.
\end{aligned}
\end{equation}
\citep{van2000asymptotic,casella2002statistical,ruschendorf2014mathematische}.
This expression makes sense as long as $g$ is well defined such that the integral has a finite value. The set $\Theta$ of all values $\theta$ that fulfill these requirement is called natural parameter space of the model. In the interior of this space,
one has that
\begin{equation}
\frac{\partial}{\partial \theta} \ln g(\theta) = - \mathbb{E}_{\theta}[T]
\end{equation}
\begin{equation}
\frac{\partial^2}{\partial^2\theta} \ln g(\theta) = -\operatorname{Cov}_{\theta}[T].
\end{equation}
Also, we have that
\begin{equation}
\operatorname{D_{KL}}\left(p(y|\theta_0)||p(y|\theta)\right) = \ln \frac{g(\theta_0)}{g(\theta)} + \theta_0^T \mathbb{E}_{\theta_0}[T] - \theta^T \mathbb{E}_{\theta_0}[T]
\end{equation}
such that we also get
\begin{equation}
\frac{\partial^2}{\partial \theta \partial \theta^T} \operatorname{D_{KL}}\left(p(y|\theta_0)||p(y|\theta)\right) = \operatorname{Cov}_{\theta}[T],
\end{equation}
which, together with the fact that covariance matrizes are positive definite, ensures the convexity of the Kullback-Leibler-divergence  in $\theta$.
In the following, we will denote  by  $\operatorname{Cov}^k_{\theta}[T]$ the block-diagonal part of the covariance matrix that corresponds to the $k$-th parameter set. Also we will also denote $\pi(\theta) = \ln p(\theta)$ for the logarithm of the prior
\myparagraph{\fullref{prop4} in exponential families}
Recalling that due to the parameters being fixed points we have that
\begin{equation}
\Sigma_k^{(n)} =  -H_{L_y}^{\theta_k} (\mu^{(n)})^{-1}.
\end{equation}
Thus, we get
\begin{equation}
\begin{aligned}
&\operatorname{tr}\left[ \Sigma_k^{(n)}   H_{L_y}^{\theta_k}(\mu_j,\mu_{/j}^{(n)})\right] \\
& =  \operatorname{tr} \left[ - H_{L_y}^{\theta_k} (\mu_j^{(n)},\mu_{/j}^{(n)}) ^{-1} H_{L_y}^{\theta_k}(\mu_j,\mu_{/j}^{(n)})  \right] \\
& = \operatorname{tr} \left[ -    \left(n\operatorname{Cov}^k_{\mu_j^{(n)},\mu_{/j}^{(n)}}[T] +H_{\pi}^{\theta_k}  (\mu_j^{(n)},\mu_{/j}^{(n)}) \right)^{-1}   \left( n\operatorname{Cov}^k_{\mu_j,\mu_{/j}^{(n)}}[T] +  H_{\pi}^{\theta_k} (\mu_j,\mu_{/j}^{(n)})\right) \right] \\
& = \operatorname{tr} \left[ -    \left(\operatorname{Cov}^k_{\mu_j^{(n)},\mu_{/j}^{(n)}}[T] + \frac{1}{n} H_{\pi}^{\theta_k}  (\mu_j^{(n)},\mu_{/j}^{(n)}) \right)^{-1}   \left( \operatorname{Cov}^k_{\mu_j,\mu_{/j}^{(n)}}[T] + \frac{1}{n}  H_{\pi}^{\theta_k} (\mu_j,\mu_{/j}^{(n)})\right) \right] \\
& \overset{\mathbb{P}}{\longrightarrow}  \operatorname{tr} \left[ -    \left(\operatorname{Cov}^k_{\theta'_j,\theta'_{/j}}[T] \right)^{-1}   \left( \operatorname{Cov}^k_{\mu_j,\theta'_{/j}}[T]   \right) \right]
\end{aligned}
\end{equation}
by the continuous mapping theorem and Slutsky's lemma, where the main point here is that the second derivatives of the likelihoods are not data-dependent anymore. At this point we  have to make the additional assumption that the prior is such that the above term is well-defined.
Then, all the additional trace terms will be in $\mathcal{O}_{\mathbb{P}}(n)$ and will not influence our limit, as we required.
Thus, we can study the remaining part of $I_j^{(n)}(\mu_j)$. We have that
\begin{equation}
\begin{aligned}
&\frac{1}{n} \sum_{i=1}^{n} \frac{\ln p(y_i|\mu_j,\mu_{/j}^{(n)})}{\ln p(y|\theta_0)}\\
&=  \ln \frac{ g(\mu_j,\mu_{/j}^{(n)})}{g(\theta_0)}   +  \left( \mu_j,\mu_{/j}^{(n)} - \theta_0 \right)^T \frac{\sum_{i=1}^{n} T(y_i)}{n} \\
& \overset{\mathbb{P}}{\longrightarrow} \ln \frac{g(\mu_j, \theta'_{/j})}{g(\theta_0)} + (\mu_j, \theta'_{/j})^T \mathbb{E}_{\theta_0}[T] - \theta_0^T \mathbb{E}_{\theta_0}[T] = - \operatorname{D_{KL}}\left( p(y|\theta_0)|| p(y|| \mu_j, \theta'_{/j}) \right)
\end{aligned}
\end{equation}
again using Slutsky's lemma and the continuous mapping theorem. Finally, we should have that 
\begin{equation}
\frac{\ln p( \mu_j,\mu_{/j}^{(n)})}{n} \overset{\mathbb{P}}{\longrightarrow} 0,
\end{equation}
which will be true as long as $p(\mu_j,\theta'_{/j}) >0$ - but else, a maximum should not attained at these points anyway.
Taking all our considerations together we see that indeed
\begin{equation}
\frac{I_j^{(n)}(\mu_j)}{n} \overset{\mathbb{P}}{\longrightarrow} -\operatorname{D_{KL}}(p(y|\theta_0) || p(y| \mu_j, \theta'_{/j}).
\end{equation}
Thus, as maximum-likelihood estimators are consistent in natural parameter exponential family models \citep{van2000asymptotic}, one could conjecture that also in our case the convergence of the maximizers will take place.
Indeed, since we already assumed the convergence of the maximizers $\mu_j^{(n)} \overset{\mathbb{P}}{\longrightarrow} \theta_j'$, we can find a compact set $K_\varepsilon$ such that $\mu_j^{(n)} \in K_\varepsilon$ with probability $1-\varepsilon$, for arbitrarily small $\varepsilon$. 
Let us denote the additional Hessian-matrix dependent terms as
\begin{equation}
\begin{aligned}
S_k(\mu_j) := \operatorname{tr} \left[ -\left(\operatorname{Cov}^k_{\mu_j^{(n)},\mu_{/j}^{(n)}}[T] + \frac{1}{n} H_{\pi}^{\theta_k}  (\mu_j^{(n)},\mu_{/j}^{(n)}) \right)^{-1}   \left( \operatorname{Cov}^k_{\mu_j,\mu_{/j}^{(n)}}[T] + \frac{1}{n}  H_{\pi}^{\theta_k} (\mu_j,\mu_{/j}^{(n)})\right) \right]   
\end{aligned}
\end{equation}
Now, for our compact sets we have
\begin{equation}
\begin{aligned}
&\underset{\mu_j \in K_\varepsilon}{\sup} \left| \frac{I_j^{(n)}(\mu_j)}{n} - (-\operatorname{D_{KL}}(p(y|\theta_0) || p(y| \mu_j, \theta'_{/j}))\right|  \\
&= \underset{\mu_j \in K_\varepsilon}{\sup} \left| \ln \frac{ g(\mu_j,\mu_{/j}^{(n)})}{g(\mu_j, \theta'_{/j})}  +   \left( \mu_j,\mu_{/j}^{(n)} - \theta_0 \right)^T \frac{\sum_{i=1}^{n} T(y_i)}{n}  -\left(\mu_j, \theta'_{/j} - \theta_0\right)^T \mathbb{E}_{\theta_0}[T]   
 + \frac{1}{n} \sum_{k \neq j} S_k(\mu_j)\right| \\
& \leq \underset{\mu_j \in K_\varepsilon}{\sup} \left| \ln \frac{ g(\mu_j,\mu_{/j}^{(n)})}{g(\mu_j, \theta'_{/j})}   \right| +  \underset{\mu_j \in K_\varepsilon}{\sup} \left| \left( \mu_j,\mu_{/j}^{(n)}\right)^T  \frac{\sum_{i=1}^{n} T(y_i)}{n}  -\left(\mu_j, \theta'_{/j}\right)^T\mathbb{E}_{\theta_0}[T] \right| \\
&+ \frac{1}{n} \sum_{k \neq j}  \underset{\mu_j \in K_\varepsilon}{\sup} \left| S_k(\mu_j)\right| \\
& \leq \underset{\mu_j \in K_\varepsilon}{\sup} \left| \ln \frac{ g(\mu_j,\mu_{/j}^{(n)})}{g(\mu_j, \theta'_{/j})}   \right| +  \underset{\mu_j \in K_\varepsilon}{\sup} \left| \mu_j \left(\frac{\sum_{i=1}^{n} T_j(y_i)}{n} - \mathbb{E}_{\theta_0}[T_j]\right)\right| \\
&+   \left| \sum_{k \neq j} \mu_k^{(n)} \frac{\sum_{i=1}^{n} T_k(y_i)}{n} - \theta'_k \mathbb{E}_{\theta_0}[T_k]\right|     + \frac{1}{n} \sum_{k \neq j}  \underset{\mu_j \in K_\varepsilon}{\sup} \left| S_k(\mu_j)\right| \\
&\overset{\mathbb{P}}{\longrightarrow} 0
\end{aligned}
\end{equation}
which follows from the continuous mapping theorem, since all of these terms will be continuous functions in $\mu_j$, and taking the supremum of a continuous function over a subset of parameters on a compact set is itself a continuous function in the remaining parameters \citep{hu2013handbook}
Thus, we can apply  \fullref{M-estimator}, with $I_j^{(n)}$ taking the place of $M^{(n)}$. The second condition needed for the proof, that is the well-separated-ness of the limit function will be also fulfilled by the convexity of the Kullback-Leibler-divergence. Applying this proof on the compact set $K_\varepsilon$ 
\begin{equation}
\mathbb{P} \left[ \left(\mu_j^{(n)} \in K_\varepsilon \right) \land \left( \left|\mu_j^{(n)} - \mu_j^*\right| > \delta\right) \right] \longrightarrow 0,
\end{equation}
where
\begin{equation}
\mu_j^* = \underset{ \mu_j }{\arg \max}  (-\operatorname{D_{KL}}(p(y|\theta_0) || p(y| \mu_j, \theta'_{/j})).
\end{equation}
And thus, since $\varepsilon$ was arbitrary,
\begin{equation}
\mathbb{P} \left[\left|\mu_j^{(n)} - \mu_j^*\right| > \delta\right] \longrightarrow 0.
\end{equation}
This shows that all our requirements are indeed satisfied.
\myparagraph{ \fullref{varlaptheorem} in exponential families}
The conditions needed to establish  \fullref{varlaptheorem} are similar to those that were needed to establish the previous proposition. However, since the assumptions on the probabilistic convergence of the likelihood functions are stronger in this case, we will need stronger assumptions on the parameter space.
Indeed we have that
\begin{equation}
\begin{aligned}
 &\frac{1}{n}\frac{\partial}{\partial \theta} \ln {p(y|\theta)} +  \frac{\partial}{\partial \theta}  \operatorname{D_{KL}}(p(y|\theta_0)|| p(y|\theta))  \\
&= \frac{\sum_{i=1}^{n} T(y_i)}{n} - \mathbb{E}_{\theta}[T].
\end{aligned}
\end{equation}
The supremum of this expression over the whole $\Theta$ may be unbounded, such that we might need the restriction of $\Theta$ being compact. Under this additional assumption, and assuming the variational parameters are in the interior of this compact set, our
assumptions are fulfilled: The well-separated-ness condition
\begin{equation}
\underset{\theta : d(\theta,\theta_0) \geq \varepsilon}{\inf} ||  \frac{\partial}{\partial \theta}  \operatorname{D_{KL}}(p(y|\theta_0)|| p(y|\theta))|| >0, \  \forall \varepsilon >0
\end{equation}
is satisfied by the identifiability and the convexity of the model. The exchange of derivative and integration is permitted in exponential families \citep{van2000asymptotic,shao2003mathematical}. Thus, it remains to check the asymptotics of the remainder term. Again, as long as the logarithm of the prior 
is twice continuously differentiable in some neighborhood of $\theta_0$, the influence of the prior will not matter in the limit. As in the previous paragraph, we can evaluate
\begin{equation}
\begin{aligned}
 &\frac{\partial}{\partial \theta_j} \operatorname{tr} \left[ H_{L_y}^{\theta_k}\left(\mu_j^{(n)},\mu_{/j}^{(n)}\right)^{-1}   H_{L_y}^{\theta_k}  \left(\theta_j,\theta_{/j}\right)    \right] \bigg \rvert_{\mu_j^{(n)},\mu_{/j}^{(n)}}  \\
&=  \frac{\partial}{\partial \theta_j} \operatorname{tr} \left[  \left(\operatorname{Cov}^k_{\mu_j^{(n)},\mu_{/j}^{(n)}}+ \frac{1}{n} H_{\pi}^{\theta_k}(\mu_j^{(n)},\mu_{/j}^{(n)}\right) ^{-1}   \left(\operatorname{Cov}^k_{\theta_j,\theta_{/j}}+ \frac{1}{n} H_{\pi}^{\theta_k}(\theta_j,\theta_{/j}\right)  \right] \bigg \rvert_{\mu_j^{(n)},\mu_{/j}^{(n)}} 
\end{aligned}
\end{equation}
which we observe is again not data dependent anymore. Thus, as long as this expression is a continuous function in $\theta$, this term will be bounded, as we are working on a compact set. For the continuity of this expression, it will for example be sufficient to assume that the prior is three times continuously differentiable and log-concave on $\Theta$, which is for example fulfilled by Gaussian priors. By this boundedness we thus see that 
\begin{equation}
\begin{aligned}
&R^{(n)}(y,\mu^{(n)}) = o_\mathbb{P}(1) \\
&\sqrt{n}R^{(n)}(y,\mu^{(n)}) = o_\mathbb{P}(1).  \\
\end{aligned}
\end{equation}
Hence, under the given conditions, we have that variational Laplace is consistent and asymptotically efficient in natural parameter exponential families on a compact parameter space.

\section{Supplementary Results}

\subsection{Derivation of the Variational Laplace Update Scheme}

\subsubsection{General Case}\label{sec:varlapderivation}
Here we want to show how the approximated version of the $\operatorname{ELBO}$ in variational Laplace is derived by a Taylor approximation. 
To write out the respective equations explicitly, and for notational ease, in this section we assume we want to update the $\operatorname{ELBO}$ with respect to the density on the
parameter $\theta_1$, but the derivations are exactly the same for any other of the parameters.
Hence, we assume we have fixed $\mu_2^*,...,\mu_p^*$  and $\Sigma_2^*,...,\Sigma_p^*$  from our previous iterations step and only want to optimize with respect to the first variational expectation and covariance parameter.
That is, we define
\begin{equation}
\begin{aligned}
&E(\mu_1,\Sigma_1) := \\
&\int \ln p(y,\theta_1,\theta_2,...,\theta_p) q_{\mu_1,\Sigma_1}(\theta_1) q_{\mu_2^*,\Sigma_2^*}....q_{\mu_p^*,\Sigma_p^*}(\theta_p) d\theta_1 ...d \theta_p \\
&+ \mathbb{H}[q_{\mu_1,\Sigma_1}(\theta_1)] + \sum_{j=2}^{p}\mathbb{H}[ q_{\mu_j^*,\Sigma_j^*}(\theta_j)]
\end{aligned}
\end{equation}
and want to optimize this  function with respect to both parameters. The second part of this equation are the differential entropies of the variational densities which evaluate to
\begin{equation}
\frac{1}{2} \ln \det (\Sigma_1) +   \frac{1}{2} \sum_{j=2}^{p} \ln \det(\Sigma_j^*) + C.
\end{equation}
We now can employ a Taylor approximation to approximate the first integral. to this end, recall the definition
\begin{equation}
L_y(\theta) = L_y(\theta_1,...,\theta_p) = \ln p(y,\theta_1,...,\theta_p),
\end{equation}
We proceed by applying a Taylor approximation centered at the variational expectation parameters $(\mu_1,\mu_2^*,...,\mu_p^*) = (\mu_1,\mu_{/1}^*)$ to obtain
\begin{equation}
\begin{aligned}
&\int {L_y(\theta_1,...,\theta_p) q_{\mu_1,\Sigma_1}(\theta_1) ...q_{\mu_p^*,\Sigma_p^*}(\theta_p) d\theta_1 ... d \theta_p} \approx \\
& \int{L_y(\mu_1,\mu_{/1}^*) q_{\mu_1,\Sigma_1}(\theta_1) ... q_{\mu_p^*,\Sigma_p^*}(\theta_p) d\theta_1 ... d \theta_p} + \\
&  \int  {J_{L_y}(\mu_1,\mu_{/1}^*)^T \left( \left(\begin{array}{c} \theta_1   \\ ...\\ \theta_p \end{array} \right) - \left(\begin{array}{c} \mu_1  \\ ...\\ \mu_p^* \end{array} \right) \right)  q_{\mu_1,\Sigma_1}(\theta_1) ... q_{\mu_p^*,\Sigma_p^*}(\theta_p) d\theta_1 ... d \theta_p }+\\
& \int{\frac{1}{2}    \left( \left(\begin{array}{c} \theta_1   \\ ...\\ \theta_p \end{array} \right) - \left(\begin{array}{c} \mu_1   \\ ...\\ \mu_p^* \end{array} \right) \right)^T     H_{L_y}(\mu_1,\mu{/1}^*)  \left( \left(\begin{array}{c} \theta_1    \\ ...\\ \theta_p \end{array} \right) - \left(\begin{array}{c} \mu_1 \\ ...\\ \mu_p^* \end{array} \right) \right)  q_{\mu_1,\Sigma_1}(\theta_1) ... q_{\mu_p^*,\Sigma_p^*}(\theta_p) d\theta_1 ... d \theta_p}  .\\
\end{aligned}
\end{equation}
We obviously have 
\begin{equation}
\int{L_y(\mu_1,\mu_{/1}^*) q_{\mu_1,\Sigma_1}(\theta_1) ... q_{\mu_p^*,\Sigma_p^*}(\theta_p) d\theta_1 ... d \theta_p}   = L_y(\mu_1,\mu_{/1}^*) .
\end{equation}
Furthermore, by our choice of the expansion point of the Taylor approximation, we get that
\begin{equation}
\begin{aligned}
& \int   J_L(\mu_1,\mu_{/1}^*)^T \left( \left(\begin{array}{c} \theta_1    \\ ...\\ \theta_p \end{array} \right) - \left(\begin{array}{c} \mu_1    \\ ...\\ \mu_p^* \end{array} \right) \right)  q_{\mu_1,\Sigma_1}(\theta_1) ...q_{\mu_p^*,\Sigma_p^*}(\theta_p) d\theta_1 ...d \theta_p \\
&=  J_L(\mu_1,\mu_{/1}^*)^T  \int  \left( \left(\begin{array}{c} \theta_1   \\ ...\\ \theta_p \end{array} \right) - \left(\begin{array}{c} \mu_1 \\ ...\\ \mu_p^* \end{array} \right) \right)  q_{\mu_1,\Sigma_1}(\theta_1) ...q_{\mu_p^*,\Sigma_p^*}(\theta_p) d\theta_1 ...d \theta_p . \\
&= 0.
\end{aligned}
\end{equation}
To evaluate the third integral, we first note that due to the mean-field factorization, the covariance matrix of $(\theta_1,\theta_2,...,\theta_p)$ takes a block-diagonal form as follows 
\begin{equation}
\Sigma:=\left(\begin{array}{c|c|c|c}\Sigma_1  & 0  &0 & \hdots  \\ \hline 0 & \Sigma_2^* & 0 & \hdots \\ \hline \vdots & 0 & \ddots & \ddots  \end{array}\right).
\end{equation}
We have that 
\begin{equation}
\begin{aligned}
& \int{\frac{1}{2}    \left( \left(\begin{array}{c} \theta_1   \\ ...\\ \theta_p \end{array} \right) - \left(\begin{array}{c} \mu_1^*    \\ ...\\ \mu_p^* \end{array} \right) \right)^T     H_{L_y}(\mu_1,\mu_{/1}^*)  \left( \left(\begin{array}{c} \theta_1  \\ ...\\ \theta_p \end{array} \right) - \left(\begin{array}{c} \mu_1^*     \\ ...\\ \mu_p^* \end{array} \right) \right)  q_{\mu_1,\Sigma_1}(\theta_1) ...q_{\mu_p^*,\Sigma_p^*}(\theta_p) d\theta_1 ...d \theta_p}  \\\\
&= \frac{1}{2} \left(\left(\begin{array}{c} \mu_1   \\ ...\\ \mu_p^* \end{array} \right) - \left(\begin{array}{c} \mu_1   \\ ...\\ \mu_p^* \end{array} \right) \right)^TH_{\theta_1,\mu_{/1}^*}  \left(\left(\begin{array}{c} \mu_1   \\ ...\\ \mu_p^* \end{array} \right) - \left(\begin{array}{c} \mu_1    \\ ...\\ \mu_p^* \end{array} \right) \right) + \frac{1}{2} \operatorname{tr}(H_{L_y}(\mu_1,\mu_{/1}^*) \Sigma)\\\\
& = 0 + \frac{1}{2} \operatorname{tr}( H_{L_y}(\mu_1,\mu_{/1}^*) \Sigma),
\end{aligned}
\end{equation}
which again follows from the choice of expansion point and using \eqref{expectquad}.
The trace term, due to the block diagonal form of $\Sigma$, can be written as
\begin{equation}
\begin{aligned}
&\frac{1}{2} \sum_{k=1}^{p} \operatorname{tr}\left(\Sigma_k H_{L_y}^{\theta_k}(\mu_1,\mu_{/1}^* )\right).\\
\end{aligned}
\end{equation}
Thus, altogether the function that we want to optimize now reads
\begin{equation}
E(\mu_1,\Sigma_1) \approx L_y(\mu_1,\mu_{/1}^*) +  \frac{1}{2} \sum_{k=1}^{p} \operatorname{tr}(\Sigma_k H_{L_y}^{\theta_k}(\mu_1,\mu_{/1}^* )) + \frac{1}{2} \sum_{k=1}^{p} \ln \det(\Sigma_k) =: \widetilde{E}(\mu_1,\Sigma_1).
\end{equation}
For a fixed $\mu_1$, using \eqref{derivtrace} and \eqref{derivlog}, the matrix derivative with respect to $\Sigma_1$ of this equation yields
\begin{equation}
\frac{\partial}{\partial \Sigma_1}\widetilde{E}(\mu_1,\Sigma_1)  = \frac{1}{2}(H_{L_y}^{\theta_1}(\mu_1,\mu_{/1}^* ))  +  \Sigma_1^{-1}),
\end{equation}
which is $0$ iff 
\begin{equation}
\Sigma_1  = - \frac{1}{2}H_{L_y}^{\theta_1}(\mu_1,\mu_{/1}^* ),
\end{equation}
provided that this yields a positive definite symmetric matrix as discussed in the main text.
We could now, in principle, plug this optimal value into the above equation and optimize with respect to $\mu_1$. This would leave us, up to constants, with optimizing the following expression:
\begin{equation}\
\widetilde{E}(\mu_1) = L_y(\mu_1,\mu_{/1}^*) +  \frac{1}{2} \sum_{k=2}^{p} \operatorname{tr}(\Sigma_k H_{L_y}^{\theta_k}(\mu_1,\mu_{/1}^* )) + \frac{1}{2}  \ln \det  (- \frac{1}{2}H_{L_y}^{\theta_1}(\mu_1,\mu_{/1}^* )).
\end{equation} 
Alternatively, one could also identify critical points by optimizing  with respect to the variational expectation parameter first. This would entail optimizing
\begin{equation}
\widetilde{E}(\mu_1) =L_y(\mu_1,\mu_{/1}^*)   +  \frac{1}{2} \sum_{k=1}^{p} \operatorname{tr}(\Sigma_k H_{L_y}^{\theta_k}(\mu_1,\mu_{/1}^* )) 
\end{equation}
first, and then setting the covariance parameter according to the preceding analytical derivations. However, in the literature, a different optimization scheme is advocated \citep{friston2007variational,daunizeau2009variational,buckley2017free}.
In this scheme, the "variational energy" 
\begin{equation}
I(\mu_1) = L_y(\mu_1,\mu_{/1}^*)   +  \frac{1}{2} \sum_{k=2}^{p} \operatorname{tr}(\Sigma_k H_{L_y}^{\theta_k}(\mu_1,\mu_{/1}^* )) 
\end{equation}
is optimized with respect to $\mu_1$.
A general motivation seems to be the assumption that the curvature of the model is of minor importance\citep{parr2020markov,buckley2017free,friston2007variational}. In this case, the additional term $ \ln \det  (- \frac{1}{2}H_{L_y}^{\theta_1}(\mu_1,\mu_{/1}^* ))$ will not matter for optimization. 
However, this will not hold true in general, such that this motivation may not be justified.  
 Thus, we want to give another consideration that may be behind this update scheme:
Recall the "fundamental lemma of variational Inference", which determines the optimal solution for a coordinate-wise update of the mean-field variational densities,
\begin{equation}
q(\theta_1) \propto \exp\left(\int L_y(\theta_1,\theta_2,...,\theta_p) q_2(\theta_2)q_3(\theta_3)...q_n(\theta_p) d \theta_2 d\theta_3...d\theta_p\right),
\end{equation}
where $q_2(\theta_2),...,q_n(\theta_p)$ are assumed to have been previously optimized.
It is also convenient to express this as 
\begin{equation}\label{fundamental}
 \ln q(\theta_1) = \int L_y(\theta_1,\theta_2,...,\theta_p) q_2(\theta_2)q_3(\theta_3)...q_n(\theta_p) d \theta_2 d\theta_3...d\theta_p + C.
\end{equation}
Now, for $j \in  \{ 2,...,n\}$, let us denote
\begin{equation}
\begin{aligned}
& \mu_j = \int \theta_j q_j(\theta_j) d\theta_j \\
& \Sigma_j = \int (\theta_j-\mu_j)(\theta_j -\mu_j)^T q_j(\theta_j) d\theta_j,
\end{aligned}
\end{equation}
note that we did not yet specify the form of $q_j$.
One can now, again, pursue a Taylor approximation to evaluate the integral involved in the above update step for the variational density $q_1$. Thus, here one approximates
\begin{equation}
\begin{aligned}
\ln q_1(\theta_1) = &\int L_y(\theta_1,\theta_2,...,\theta_p)  q(\theta_2)...q(\theta_p) d \theta_1 ... d\theta_p \approx\\
& \int L_y(\theta_1,\mu_{/1})q(\theta_2)...q(\theta_p) d \theta_1 ... d\theta_p \ +  \\
&  \int J_{L_y}(\theta_1,\mu_{/1})^T \left( \left(\begin{array}{c} \theta_1 \\ \theta_2   \\ ...\\ \theta_p \end{array} \right) - \left(\begin{array}{c} \theta_1 \\ \mu_2   \\ ...\\ \mu_p \end{array} \right) \right) q(\theta_2)...q(\theta_p) d \theta_1 ... d\theta_p  \ + \\
&  \int \frac{1}{2}    \left( \left(\begin{array}{c} \theta_1 \\ \theta_2   \\ ...\\ \theta_p \end{array} \right) - \left(\begin{array}{c} \theta_1 \\ \mu_2   \\ ...\\ \mu_p \end{array} \right) \right)^T     H_{L_y}(\theta_1,\mu_{/1} )  \left( \left(\begin{array}{c} \theta_1 \\ \theta_2   \\ ...\\ \theta_p \end{array} \right) - \left(\begin{array}{c} \theta_1 \\ \mu_2   \\ ...\\ \mu_p \end{array} \right) \right)   q(\theta_2)...q(\theta_p) d \theta_1 ... d\theta_p. \\
\end{aligned}
\end{equation}
Note that the here, in contrast to the previously described scheme, the expansion point of the Taylor approximation is $ (\theta_1,\mu_{/1}) = (\theta_1,\mu_2,...,\mu_p)$, to emphasize that the result of this approximation is a density on $\theta_1$.
With the same derivations as employed before, this approximation then leads to 
\begin{equation}
\ln q_1(\theta_1) =  L_y(\theta_1,\mu_{/1}) + \frac{1}{2} \sum_{j=2}^{p} \operatorname{tr}(\Sigma_j H_{L_y}^{\theta_k}(\theta_1,\mu_{/1} )) + C
\end{equation}
We notice that up to this point we did not restrict the form of $q_2,...,q_p$ to be Gaussian. Furthermore, we see that a mode of this approximated density $q_1$ will exactly be a maximum of the previously defined "variational energy" $I(\mu_1)$.
Now, with the assumption of $q_1$ being Gaussian, we get that the mode will be equal to the mean, and hence the mode could be passed as a new value for $\mu_1$ to proceed to updating the other densities. However, this Gaussian assumption will make make the expression of the optimal density as in  equation \eqref{fundamental} invalid, such that the previous derivation via the Taylor approximation becomes meaningless. Hence, these ideas should be understood as a heuristic, rather than a rigorous derivation.

\subsubsection{Nonlinear Transform Model}\label{sec:varlapderivation_model}
In this section we will derive the variational Laplace update equations in the nonlinear transform model , but first we want to give a brief reasoning why the free-form approach may not be feasible in this model.
to this end, we posit a mean-field factorized variational density of the form $q(\theta,\lambda) = q(\theta) q(\lambda)$. Without any further assumptions, by the fundamental Lemma of variational inference, the coordinate-wise update for the variational density $q(\theta)$ would read
\begin{equation}
\begin{aligned}
\ln q(\theta) &=  \int  \ln p(y|\theta,\lambda) + \ln p(\theta) + \ln p(\lambda)  q(\lambda) d\lambda +C \\
		& =  \int -\frac{1}{2} (y - f(\theta)^T \Sigma(\lambda)^{-1} (y-f(\theta)    + \ln p(\theta) +C \\
		&= -\frac{1}{2} (y-f(\theta))^T \mathbb{E}_{q(\lambda)}[ Q(\lambda)^{-1}] (y-f(\theta)) +  \ln p(\theta) +C
\end{aligned}
\end{equation}
where we have avoided all terms not depending on $\theta$, as these will be normalized out. We thus see, that this approach will not yield any significant advantage over the computation of the true posterior: As long as  $\mathbb{E}_{q(\lambda)}[ Q(\lambda)^{-1}]$ will yield a valid covariance matrix, this variational density will have the same functional form as the conditional posterior $p(\theta| \lambda,y)$. The nonlinearities in the likelihood will typically make the involved calculations impractical, so that the free-form approach will not yield a easy-to-solve update scheme.
Thus, it may be justified to apply the approximation employed by the variational Laplace scheme. 

To make this section self-contained, we will explicitly state the steps of the derivation again.
Therefore, we assume Gaussian densities $q_\theta = \mathcal{N}(\mu_\theta,\Sigma_\theta) $ and $q_\lambda = \mathcal{N}(\mu_\lambda,\Sigma_\lambda)$.  
We want to update our variational densities at some iteration step (iteration superscripts will be left out for clarity), where from the previous iteration we have found the optimal parameters 
$ ( \mu_\theta^*, \mu_\lambda^*, \Sigma_\theta^*,\Sigma_\lambda^*) $. We want to now optimize the $\operatorname{ELBO}$ coordinate-wise, and start by optimization with respect to $(\mu_\theta,\Sigma_\theta)$. Thus, we want to optimize the function
\begin{equation}
\begin{aligned}
&E(\mu_\theta,\Sigma_\theta) = \operatorname{ELBO}(q_{\mu_\theta,\Sigma_\theta}(\theta),q_{\mu_\lambda^*,\Sigma_\lambda^*}(\lambda))\\
&= \int L_y(\theta,\lambda) q_{\mu_\theta,\Sigma_\theta}(\theta),q_{\mu_\lambda^*,\Sigma_\lambda^*}(\lambda) d\theta d\lambda  - \int q_{\mu_\theta,\Sigma_\theta}(\theta) d\theta - \int q_{\mu_\lambda^*,\Sigma_\lambda^*}(\lambda)d\lambda \\
&= \int L_y(\theta,\lambda) q_{\mu_\theta,\Sigma_\theta}(\theta),q_{\mu_\lambda^*,\Sigma_\lambda^*}(\lambda) d\theta d\lambda  + \mathbb{H}[q_{\mu_\theta,\Sigma_\theta}(\theta)] + C.
\end{aligned}
\end{equation}
We can directly evaluate 
\begin{equation}
 \mathbb{H}[q_{\mu_\theta,\Sigma_\theta}(\theta)] = \frac{1}{2} \ln \det (\Sigma_\theta)   + C,
\end{equation}
where we will ignore constants that do not depend on our variational parameters.
For the remaining integral term, we proceed by employing a Taylor approximation of the function $ L_y(\theta,\lambda)  =\ln p(y,\theta,\lambda)$ at the expansion point $(\mu_\theta,\mu_\lambda^*)$. 
We can then approximate the integral 
\begin{equation}
\begin{aligned}
& \int L_y(\theta,\lambda) q_{\mu_\theta,\Sigma_\theta}(\theta),q_{\mu_\lambda^*,\Sigma_\lambda^*}(\lambda) d\theta d\lambda  \\
& \approx L_y(\mu_\theta,\mu_\lambda^*) +  J_L(\mu_\theta,\mu_\lambda^*)    \int    \left(\begin{array}{c} \theta \\ \lambda \end{array}\right)  -  \left(\begin{array}{c} \mu_\theta \\ \mu_\lambda^* \end{array}\right) q_{\mu_\theta,\Sigma_\theta}(\theta),q_{\mu_\lambda^*,\Sigma_\lambda^*}(\lambda) d\theta d\lambda \\
& + \frac{1}{2}   \int      \left(\begin{array}{c} \theta \\ \lambda \end{array}\right)  -  \left(\begin{array}{c} \mu_\theta \\ \mu_\lambda^* \end{array}\right)^T  H_{L_y}(\mu_\theta,\mu_\lambda^*)  \left(\begin{array}{c} \theta \\ \lambda \end{array}\right)  -  \left(\begin{array}{c} \mu_\theta \\ \mu_\lambda^* \end{array}\right)   q_{\mu_\theta,\Sigma_\theta}(\theta),q_{\mu_\lambda^*,\Sigma_\lambda^*}(\lambda) d\theta d\lambda  \\
& =  L_y(\mu_\theta,\mu_\lambda^*)   + \frac{1}{2}   \operatorname{tr} [H_{L_y}(\mu_\theta,\mu_\lambda^*) \Sigma],
\end{aligned}
\end{equation}
Where $ \Sigma =  \left(\begin{array} {c|c} \Sigma_\theta & 0   \\ \hline 0 &  \Sigma_\lambda^* \end{array}\right) $. We see that due to this block-diagonal form of $\Sigma$, we can split the trace term on the right hand side up, to yield
\begin{equation}
 \frac{1}{2}   \operatorname{tr} [H_{L_y}(\mu_\theta,\mu_\lambda^*)\Sigma] =  \frac{1}{2}   \operatorname{tr} [H_{L_y}^{\theta}(\mu_\theta,\mu_\lambda^*) \Sigma_\theta] + \frac{1}{2}   \operatorname{tr} [H_{L_y}^{\lambda}(\mu_\theta,\mu_\lambda^*) \Sigma_\lambda^*],
\end{equation}
Thus, the function that we want to optimize now reads
\begin{equation}\label{objective}
E(\mu_\theta,\Sigma_\theta) =   L_y(\mu_\theta,\mu_\lambda^*) + \frac{1}{2}   \operatorname{tr} [H_{L_y}^{\theta}(\mu_\theta,\mu_\lambda^*) \Sigma_\theta] + \frac{1}{2}   \operatorname{tr} [H_{L_y}^{\lambda}(\mu_\theta,\mu_\lambda^*)  \Sigma_\lambda^*] + \frac{1}{2} \ln \det (\Sigma_\theta) .
\end{equation}
If we assume that we can calculate the matrix derivative of this expression with respect to $\Sigma_\theta$, using \eqref{derivtrace} and \eqref{derivlog} it evaluates to 
\begin{equation}
\frac{\partial}{\partial \Sigma_\theta}  E(\mu_\theta,\Sigma_\theta) = \frac{1}{2} (  H_{L_y}^{\theta}(\mu_\theta,\mu_\lambda^*)   + \Sigma_\theta^{-1}),
\end{equation}
which is zero iff
\begin{equation}
\Sigma_\theta = - H_{L_y}^{\theta}(\mu_\theta,\mu_\lambda^*)^{-1}.
\end{equation}
At this point it is important to note that this derivation is made without constraining the form of $\Sigma_\theta$. Thus, this will only yield an optimum for our problem if $-H_{L_y}^{\theta}(\mu_\theta,\mu_\lambda^*)$ is indeed a valid covariance matrix. 
Symmetry will be given, as long as the model is at least two times continuously differentiable. Whether $-H_{L_y}^{\theta}(\mu_\theta,\mu_\lambda^*)$ will also be positive definite depends on our model. If the model is log-concave, then this assumption will be fulfilled. 
This may however not be given for all models of the above form. We will later see how this problem is typically circumvented in the literature by using an approximation to the Hessian that indeed will result in a positive-definite matrix.
In the following, we will proceed to use the expression $H_{L_y}^{\theta}(\mu_\theta,\mu_\lambda^*)$, referring to either the exact term or the respective approximation, such that the equations in question remain meaningful.
If we now  plug in this expression for $\Sigma_\theta$ in \eqref{objective}, we would be  left to optimize
\begin{equation}
E(\mu_\theta)  = L_y(\mu_\theta,\mu_\lambda^*) +  \frac{1}{2}   \operatorname{tr} [H_{L_y}^{\lambda}(\mu_\theta,\mu_\lambda^*)  \Sigma_\lambda^*] + \frac{1}{2} \ln \det (- H_{L_y}^{\theta}(\mu_\theta,\mu_\lambda^*)^{-1}).
\end{equation}
As stated before, the last term involving the determinant is usually dropped, so that what is left to optimize becomes
\begin{equation}
I(\mu_\theta) =L_y(\mu_\theta,\mu_\lambda^*) +  \frac{1}{2}   \operatorname{tr} [H_{L_y}^{\lambda}(\mu_\theta,\mu_\lambda^*)  \Sigma_\lambda^* ].
\end{equation}
Writing out the log-joint probability in the first term, $I(\mu_\theta)$ reads 
\begin{equation}
\begin{aligned}
-\frac{1}{2} (y-f(\mu_\theta))^T Q(\mu_\lambda^*)^{-1} (y-f(\mu_\theta)) -\frac{1}{2} (\mu_\theta-m_\theta)^T S_\theta^{-1}  (\mu_\theta-m_\theta) \\
-\frac{1}{2} (\mu_\lambda^*-m_\lambda)^T S_\lambda^*)  (\mu_\lambda^*-m_\lambda) +  \frac{1}{2}   \operatorname{tr} [H_{L,\mu_\theta}(\mu_\lambda^*) \Sigma_\lambda^*],
\end{aligned}
\end{equation}
up to constants. We note that we can drop any terms not involving $\mu_\theta$ for the optimization. In particular, for the Hessian in the last term we only have to compute the second derivatives of the log-likelihood term, due to the linearity of the trace and derivative operators.
Thus, we have that
\begin{equation}
\begin{aligned}
I(\mu_\theta) =
		    &-\frac{1}{2} \left(y-f\left(\mu_\theta\right)\right)^T Q\left(\mu_\lambda^*\right)^{-1} \left(y-f\left(\mu_\theta\right)\right) \\
		    & -\frac{1}{2} \left(\mu_\theta - m_\theta\right)^T S_\theta^{-1}\left(\mu_\theta - m_\theta\right) \\
		    & + \frac{1}{2} \operatorname{tr} \left(\Sigma_\lambda^*   \frac{\partial^2}{\partial\lambda \partial \lambda^T} \left(-\frac{1}{2} \left(y-f\left(\mu_\theta\right)\right)^T Q(\lambda)^{-1}\left(y-f\left(\mu_\theta\right)\right)\right) \bigg\rvert_{\lambda = \mu_\lambda^*}\right)+C
\end{aligned}
\end{equation}
The second derivative that one needs to compute here can be written out explicitly, as can be seen for example in \cite{friston2007variational}. After an optimal value $\mu_\theta^*$ is found, the covariance parameter is then set according to the previously stated analytic expression.
However, it is advocated \citep{friston2007variational,daunizeau2009variational} to approximate the Hessian involved in this step. We have that
\begin{equation}
\frac{\partial}{\partial \theta} (\ln p(y,\theta,\mu_\lambda^*)) = -\frac{1}{2} J_f(\theta)^T(Q(\mu_\lambda^*)^{-1} (y-f(\theta) -\frac{1}{2} S_\theta^{-1} (\theta-m_\theta),
\end{equation}
where $J_f$ is the Jacobian of the transform $f$.
Thus, the Hessian is 
\begin{equation}
H_{L_y}^{\theta}(\mu_\theta,\mu_\lambda^*)= -\frac{1}{2} ( J_f(\mu_\theta)^T Q(\mu_\lambda^*){-1}) J_f(\mu_\theta) + S_\theta^{-1})  + \delta^2_f,
\end{equation}
where we denoted by $\delta^2_f$ some term depending on the second derivatives of $f$ \citep{petersen}. This term is then proposed to drop, which may be motivated by the transform being only "slightly nonlinear"\citep{friston2007variational} or in order to "render the covariance matrix positive definite"\citep{daunizeau2009variational}. The latter follows from the fact that since by definition $Q(\mu_\lambda^*)^{-1}$ is a positive definite matrix, $A^T Q(\mu_\lambda^*)^{-1} A$ is  positive semidefinite for any matrix $A$. Since $S_\theta^{-1}$ is positive definite, this renders the sum positive definite.
With the same considerations as before, one can derive the update scheme for $\mu_\lambda,\Sigma_\lambda$. Here, one will be left to optimize
\begin{equation}
E(\mu_\theta,\Sigma_\theta) =   L_y(\mu_\theta^*,\mu_\lambda) + \frac{1}{2}   \operatorname{tr} [H_{L_y}^\theta(\mu_\lambda,\mu_\theta^*) \Sigma_\theta^*] + \frac{1}{2}   \operatorname{tr} [H_{L_y}^\lambda(\mu_\theta^*,\mu_\lambda) \Sigma_\lambda] + \frac{1}{2} \ln \det (\Sigma_\lambda) .
\end{equation}
If we assume that we can calculate the matrix derivative of this expression with respect to $\Sigma_\lambda$, it evaluates to 
\begin{equation}
\frac{\partial}{\partial \Sigma_\lambda}  E(\mu_\lambda,\Sigma_\lambda) = \frac{1}{2} ( H_{L_y}^\lambda(\mu_\theta^*,\mu_\lambda)  + \Sigma_\lambda^{-1}),
\end{equation}
which is zero iff
\begin{equation}
\Sigma_\lambda = - H_{L_y}^\lambda(\mu_\theta^*,\mu_\lambda)^{-1},
\end{equation}
again, provided that this yields a positive definite matrix.  Note that, in contrast to the previous case, in \citet{friston2007variational} it is discussed to choose the 
parameterization of the covariance matrix beforehand in a way that ensures the positive-definiteness of $- H_{L_y}^\lambda(\mu_\theta^*,\mu_\lambda)$ instead of using an approximation.
Again, one proceeds by optimizing
\begin{equation}
I(\mu_\lambda) =L_y(\mu_\theta^*,\mu_\lambda) +  \frac{1}{2}   \operatorname{tr} [H_{L_y}^\theta(\mu_\lambda;\mu_\theta^*) \Sigma_\theta^*]
\end{equation}
and then setting the covariance parameter as above. Again, for the second Hessian with respect to $\theta$ involved in this optimization, it is proposed to drop higher derivatives of $f$. 
Thus, one arrives at 
\begin{equation}
\begin{aligned}
I(\mu_\lambda)  = &-\frac{1}{2} \left(y-f\left(\mu_\theta^*\right)\right)^T Q\left(\mu_\lambda\right)^{-1} \left(y-f\left(\mu_\theta^*\right)\right) \\
&-\frac{1}{2} \ln \det Q(\mu_\lambda)\\
		    & -\frac{1}{2} \left(\mu_\lambda - m_\lambda\right)^T S_\lambda^{-1}\left(\mu_\lambda - m_\lambda\right) \\
		    & - \frac{1}{2} \operatorname{tr} \left(\Sigma_\theta^*   \left(J_f\left(\mu_\theta^*\right)\right)^T Q\left(\mu_\lambda\right)^{-1} J_f\left(\mu_\theta^*\right)  + S_\theta^{-1}\right) +C.
\end{aligned}
\end{equation}

\subsection{Asymptotics of Variational Inference}

\subsubsection{Assumptions in \citet{wang2019frequentist}}\label{sec:bleiresults}

Here we state the assumptions needed for the results presented by \cite{wang2019frequentist}.
The first assumption is an assumption on the prior. It assures that the prior is sufficiently smooth and puts some mass on a neighborhood of the true parameter $\theta_0$. Hence, it is required that
\begin{enumerate}
\item $p(\theta)$ is continuous on $\Theta$
\item  $p(\theta_0) > 0$
\item $\exists M_p > 0 ,\ s.t. \ \left\Vert \frac{\partial \ln p(\theta)}{\partial \theta \partial \theta ^T}\right\Vert \leq M_p e^{\left\Vert\theta\right\Vert^2} $
\end{enumerate}
The second assumption is a consistent testability assumption: It is required that for every $\varepsilon > 0$, there exists some sequence of tests $\phi_n$, such that 
\begin{equation} \begin{aligned}
&\int \phi_n(y) p(y|\theta_0) dy \longrightarrow 0 \\
&\underset {\theta : \left\Vert\theta-\theta_0\right\Vert \geq \varepsilon}{\sup} \int (1-\phi_n(y))  p(y|\theta) dy \longrightarrow 0 
\end{aligned}\end{equation} 
This assumption intuitively states that one can test the null hypothesis of $\theta = \theta_0$ against the alternative $\left\Vert \theta-\theta_0\right\Vert>0$ with uniform adequacy over all possible $\theta$.
For example, this assumption is satisfied if there exists a uniform (over $\theta)$) consistent estimator $ \hat{\theta}^{(n)}$ , with
\begin{equation}
\phi_n := \mathbb{I}[\left\Vert \hat{\theta}^{(n)}-\theta_0\right\Vert \geq \frac{\varepsilon}{2}].
\end{equation}
Then
\begin{equation} \begin{aligned}
&\int \phi_n(y) p(y|\theta_0) dy  = \mathbb{P}[\hat{\theta}^{(n)} - \theta_0 \geq \frac{\varepsilon}{2}] \longrightarrow 0 \\
&\underset {\theta : \left\Vert \theta-\theta_0\right\Vert \geq \varepsilon}{\sup} \int (1-\phi_n(y)) p(y|\theta) dy\\
&= \underset {\theta : \left\Vert \theta-\theta_0\right\Vert \geq \varepsilon}{\sup} \mathbb {P}_{\theta}[\left \Vert \hat{\theta}^{(n)} - \theta_0\right \Vert < \frac{\varepsilon}{2}] \\
&\leq \underset {\theta : \left\Vert \theta-\theta_0 \right\Vert \geq \varepsilon}{\sup} \mathbb {P}_{\theta}\left[\left\Vert \hat{\theta}^{(n)} - \theta\right\Vert + \left\Vert \theta-\theta_0\right\Vert < \frac{\varepsilon}{2}\right] \\
&= \underset {\theta : \left\Vert \theta-\theta_0\right\Vert \geq \varepsilon}{\sup} \mathbb {P}_{\theta}\left[ \left\Vert \theta-\theta_0\right\Vert - \frac{\varepsilon}{2}<\left\Vert \hat{\theta}^{(n)} - \theta \right\Vert \right] \\
&\leq \underset {\theta : \left\Vert \theta-\theta_0\right\Vert \geq \varepsilon}{\sup} \mathbb {P}_{\theta}\left[ \frac{\varepsilon}{2}<\left\Vert \hat{\theta}^{(n)} - \theta\right\Vert \right] \longrightarrow 0
\end{aligned}\end{equation} 
Using the triangle inequality and the fact that $\left\Vert \theta-\theta_0\right\Vert \geq \frac{\varepsilon}{2}$.
In \cite{van2000asymptotic} sufficient conditions on a model for the existence of such estimators are given.
The third assumption is a local asymptotic normality (LAN) condition: It is required that there exist some random vectors, bounded in probability $\Delta_{n,\theta_0}(Y)$ and a nonsingular matrix $V_{\theta_0}$ such that for any compact $K$ and any $\delta_n \longrightarrow 0$.
\begin {equation}
\underset{h \in K}{\sup} \left| \ln{p(Y|\theta + \delta_n h)} - \ln(p(Y|\theta) - h^T V_{\theta_0} \Delta_{n,\theta_0}(Y) + \frac{1}{2} h^T V_{\theta_0} h \right| \overset{\mathbb{P}}{\longrightarrow} 0 
\end{equation}
LAN is a general property of models that lies at the heart of many theoretic considerations in asymptotic statistics \citep{van2000asymptotic}. Roughly it states that the likelihood ratios of the true model and models based on small deviations in the neighborhood of the "local" true parameter behave as in a normal experiment based on a single observation. The random vectors $\Delta_n$ and the matrix $V_{\theta_0}$ are used to derive limit distributions of point estimators in such models.
Importantly, the LAN condition will be satisfied under typical regularity conditions that are assumed to prove the asymptotic normality of the maximum-likelihood-estimator in i.i.d. models.
For example, differentiability in quadratic mean of the model together with a local Lipschitz condition on the likelihood will be sufficient. 
Provided the asymptotic consistency of the MLE is given we have that:
\begin{equation}
\ln \frac{p(Y|\theta_0 + \frac{h}{\sqrt{n}})}{p(Y|\theta_0)} \overset{\mathbb{P}}{\longrightarrow} \frac{1}{\sqrt{n}} h \frac{\partial \ln p(Y|\theta)}{\partial \theta}(\theta_0) - \frac{1}{2} h^T \mathcal{I}(\theta_0) h
\end{equation}
Such that we get: 
\begin{equation}
\begin{aligned}
\Delta_{n,\theta_0}(Y) &=\frac{1}{\sqrt{n}}\mathcal{I}(\theta_0)^{-1} \frac{\partial \ln p(Y|\theta)}{\partial \theta}(\theta_0) \\
V_{\theta_0} &= \mathcal{I}(\theta_0)
\end{aligned}
\end{equation}
As stated in \citet{van2000asymptotic}, theorem 7.12, the convergence in probability also holds when the supremum over a compact set is taken.

\subsubsection {Underdispersion of Variational Approximations}\label{sec:entropy}
Here we show inequality \eqref{underdisp} in the special case of a diagonal covariance matrix, closely following the exposition of the inequality in \citet{wang2019frequentist}. The inequality actually does not depend on the particular form of the right hand side;
it will always hold true when approximating one Gaussian by a Gaussian with diagonal covariance matrix. Thus we will show it in this general setting.
Assume a fully-factorized Gaussian family, that is 
\begin{equation}
q(\theta) = \mathcal{N}(\theta,\mu_1, \Sigma_1);\Sigma^1 = \operatorname{diag}(\sigma_1,...,\sigma^1_p) 
\end{equation}
We proceed by evaluating the Kullback-Leibler divergence between densities of this form and a fixed density $\mathcal{N}(\theta,\mu_2,\Sigma_2)$,
\begin{equation} \begin{aligned}
&\operatorname{D_{KL}}(\mathcal{N}(\mu_1, \operatorname{diag}(\sigma_1,...,\sigma_p))||\mathcal{N}(\mu_2,\Sigma_2)) \\
&= -\frac{1}{2}(\ln\frac{\det(\operatorname{diag}(\sigma_1,...,\sigma_p))}{\det(\Sigma_{2})}+p) +\frac{1}{2}(\mu_1-\mu_2)^T \Sigma_{2}^{-1}(\mu_1-\mu_2) \\
&+ \frac{1}{2} \operatorname{tr}(\Sigma_{2}^{-1}\  \operatorname{diag}(\sigma_1,...,\sigma_p)),
\end{aligned}\end{equation} 
where we used the well-known closed-form expression for the Kullback-Leibler divergence between two Gaussians \citep{murphy2012machine}.
By the positive-definiteness of $\Sigma_{2}^{-1}$ we directly see that the minimum for the expectation parameter is $\mu_1 = \mu_2$.  Hence dropping terms that do not depend on $\Sigma_1$, and using the diagonal form we can get
\begin{equation} \begin{aligned}
& -\frac{1}{2}\ln\det(\operatorname{diag}(\sigma_1,...,\sigma_p)) + \frac{1}{2} \operatorname{tr}(\Sigma_{2}^{-1}\ \operatorname{diag}(\sigma_1,...,\sigma_p)) \\
& = \frac{1}{2} \sum_{j=1}^{p}  \Sigma_{2,jj}^{-1} \sigma_j -\ln \sigma_j,
\end{aligned}\end{equation} 
where $\Sigma_{2,jj}^{-1}$ is the $j$-th diagonal entry of $\Sigma_2^{-1}$. Calculating the derivatives with respect to the diagonal entries, we find the maxima as
\begin{equation} \begin{aligned}
&\frac{\partial}{\partial \sigma_j} \frac{1}{2} \sum_{j=1}^{p}  \Sigma_{2,jj}^{-1} \sigma_j -\ln \sigma_j = - \frac{1}{\sigma_j} + \Sigma_{2,jj}^{-1} \\
& - \frac{1}{\sigma_j} + \Sigma_{2,jj}^{-1} = 0 \iff  \sigma_j = \frac{1}{\Sigma_{2,jj}^{-1}}.
\end{aligned}\end{equation} 
This means that the 'closest' diagonal matrix in terms of Kullback-Leibler divergence has the diagonal elements given by the inverses of the diagonal elements of the fixed matrix.
Denoting the optimal diagonal covariance matrix as $\Sigma_1^* = \operatorname{diag}(\frac{1}{\Sigma_{2,11}^{-1}},...,\frac{1}{\Sigma_{2,pp}^{-1}})$, one gets that the best approximating density is given by
\begin{equation}
q^*(\theta) = \mathcal{N}(\theta;\mu_2,\Sigma_1^*).
\end{equation}
With this optimal member identified we can see that the choice of a diagonal density as an approximation will hence always result in less or equal entropy than the optimal density. This is because as a corollary of Hadamard's inequality one has that the determinant of a positive semi-definite matrix is bounded from above by the product of its diagonal entries \citep{rozanski2017more}, which gives us
\begin{equation} \begin{aligned}\label{proofunderdisp}
&\det({\Sigma_{2}^{-1}}) \leq \prod_j^p \Sigma_{2,jj}^{-1}  = \det(\Sigma_1^*)^{-1} \\
&\implies \det(\Sigma_1^*) \leq \det(\Sigma_{2}).
\end{aligned}\end{equation} 
Thus, using \eqref{Gaussianentropy} for the entropies we get 
\begin{equation}
\mathbb{H}[q^*(\theta) ] = \frac{1}{2} (\ln (2\pi) + p + \ln \det(\Sigma_1^*) \leq   \frac{1}{2} (\ln (2\pi) + p + \ln \det(\Sigma_2) = \mathbb{H}[\mathcal{N}(\theta,\mu_2,\Sigma_2)].
\end{equation}
This in particular shows inequality \eqref{underdisp}.

\subsubsection{Non-Equality of Variational Inference and MAP}\label{appendix:nonequivalence}

In this subsection we want to give an example of a situation where the MAP-estimate does not correspond the the maximizer of the $\operatorname{ELBO}$. To do so, we will consider a special case of our nonlinear transform model with a single data point,
\begin{equation}
\begin{aligned}
p(y|\theta) = \mathcal{N}(y; \exp(\theta),1) \\
p(\theta) = \mathcal{N}(\theta;m_\theta,s_\theta^2).
\end{aligned}
\end{equation}
The advantage of this model is, that the $\operatorname{ELBO}$ function $E(\mu_\theta,\sigma_\theta^2)$ can be evaluated directly. We will use the fact that for a normally distributed random variable 
$$\theta \sim \mathcal{N}(\theta;\mu,\sigma^2)$$
we can analytically evaluate the expectation of an exponential of this variable \citep{casella2002statistical}, that is
$$ \mathbb{E}[\exp(\theta)] = \exp(\mu + \frac{\sigma^2}{2}).$$
Using this, we get that the $\operatorname{ELBO}$ up to constants can be expressed as
\begin{equation}
E(\mu_\theta,\sigma_\theta^2)  = y \exp(\mu_\theta) \exp(\frac{1}{2}\sigma_\theta^2)  - \frac{1}{2}  \exp(2\mu_\theta) \exp(2\sigma_\theta^2) - \frac{1}{2} \frac{(\mu_\theta - m_\theta)^2}{s_\theta^2} - \frac{1}{2} \frac{\sigma_\theta^2}{s_\theta^2}  + \frac{1}{2} \ln \sigma_\theta^2.
\end{equation}
A maximizer of this function will have that the derivatives with respect to both arguments evaluate to zero. We calculate the derivatives as 
\begin{equation}
\begin{aligned}
&\frac{\partial E}{\partial \mu_\theta} (\mu_\theta,\sigma_\theta^2) =  y \exp(\mu_\theta)\exp(\frac{1}{2}\sigma_\theta^2)   - \exp(2\mu_\theta) \exp(2\sigma_\theta^2) - \frac{\mu_\theta}{s_\theta^2} + \frac{m_\theta}{s_\theta^2}\\
&\frac{\partial E}{\partial \sigma_\theta^2} (\mu_\theta,\sigma_\theta^2)  = \frac{y}{2} \exp(\mu_\theta)\exp(\frac{1}{2}\sigma_\theta^2)   - \exp(2\mu_\theta) \exp(2\sigma_\theta^2) - \frac{1}{2s_\theta^2} + \frac{1}{2\sigma_\theta^2}.
\end{aligned}
\end{equation}
Let us now on the other hand consider the MAP-estimate. to this end, we will also consider the derivative
\begin{equation}
\begin{aligned}
&\frac{\partial \ln p(y,\mu_\theta)}{\partial \mu_\theta} =  y \exp(\mu_\theta) - \exp(\mu_\theta)^2  - \frac{\mu_ \theta-m_\theta}{s_\theta^2}.
\end{aligned}
\end{equation}
A MAP estimate will fulfill that this derivative is equal to $0$. Now assume that we have some point $\mu_\theta,\sigma_\theta$, such that both   $\frac{\partial E}{\partial \mu_\theta} (\mu_\theta,\sigma_\theta^2) =  0 $ and $\frac{\partial \ln p(y,\mu_\theta)}{\partial \mu_\theta} = 0.$
This means in particular  that $\frac{\partial E}{\partial \mu_\theta} (\mu_\theta,\sigma_\theta^2)  = \frac{\partial \ln p(y,\mu_\theta)}{\partial \mu_\theta}$  and hence
\begin{equation}
\begin{aligned}
 & (y - y \exp(\frac{\sigma_\theta^2}{2}))   \exp(\mu_\theta) -  (\frac{1}{2} - \exp(\sigma_\theta^2)) \exp(\mu_\theta)^2 = 0 \\
&\iff (y - y \exp(\frac{\sigma_\theta^2}{2}))  -  (\frac{1}{2} - \exp(\sigma_\theta^2)) \exp(\mu_\theta) =0 \\
&\iff \exp(2\mu_\theta)  = \frac{y(1-\exp(\frac{\sigma_\theta^2}{2})}{\frac{1}{2} - \exp(\sigma_\theta^2)^2}.
\end{aligned}
\end{equation}
We note that the fraction on the right side will be nonnnegative, since both the enumerator and the denominator are nonpositive due to the fact that $\sigma_\theta^2 \in (0,\infty)$.  Thus, the above has no solution for all $y<0$ and hence we cannot simultaneously maximize $\ln p(y,\mu_\theta)$ and $E(\mu_\theta,\sigma_\theta^2)$.

\section{Auxiliary Material \& Notation} \label{sec:aux}

\subsection{Asymptotic Statistics}

\subsubsection{Probabilistic Landau Notation}\label{sec:onotation}

When studying limits, it may be convenient to characterize some of the involved terms only in terms of their limiting behaviour, without paying too much attention to their particular form. Landau notation enables this type of characterization. We will use symbols $o_\mathbb{P}(1)$ and $\mathcal{O}_{\mathbb{P}}(1)$.
These are used in the following way: Consider a sequence of random variables $x^{(n)}$. Then
$$ x^{(n)} = o_\mathbb{P}(1) \ \iff \  x^{(n)} \overset{\mathbb{P}}{\longrightarrow} 0 $$
and 
$$ x^{(n)} = O_\mathbb{P}(1) \ \iff  \ \forall \varepsilon \ \exists  \ M_\varepsilon, N_\varepsilon \ \text{s.t.} \ \forall n\geq N_\varepsilon  \ \mathbb{P}\left[ \left|x^{(n)}\right| > M_\varepsilon\right] < \varepsilon .$$
The latter is referred to as $x^{(n)}$ being bounded in probability, that is one can find a compact set such that all $x^{(n)}$ are in this compact set with arbitrarily high probability.

\subsubsection{Important Asymptotical Results in Probability Theory}
Here we state some well-known results in asymptotic probability theory that are used in our derivations. We will only state this results, proofs may be found in \cite {van2000asymptotic}.
\begin{theorem}[Continuous mapping theorem]
Suppose $x^{(n)}$ is a sequence of random variables, that converges to a random variable $x$. Let $g$ be a continuous function. Then $g(x^{(n)})$ converges to $g(x)$.
\end{theorem}
Here, the convergence may be in distribution, in probability or almost surely. This result will come handy when determining the probabilistic limits of point estimators. As a consequence of this theorem, one has 
\begin{theorem}[Slutsky's lemma]
Suppose $x^{(n)}$,$y^{(n)}$ are sequences of random variables, such that $x^{(n)} \overset{d}{\longrightarrow} x$, where $x$ is some random variable, and $y^{(n)} \overset{\mathbb{P}}{\longrightarrow} c$, where c is some constant. Then $$ x^{(n)} + y^{(n)} \overset{d}{\longrightarrow} x + c$$, $$ x^{(n)} y^{(n)} \overset{d}{\longrightarrow} xc$$ and $$ x^{(n)}/y^{(n)} \overset{d}{\longrightarrow} x/c ; \ c\neq 0.$$
If $x$ is constant, the convergence also holds in probability.
\end{theorem}
Arguably the best known asymptotic result in statistics is the law of large numbers, giving the convergence of the sample mean to the expected value of the underlying distribution. Formally, we have
\begin{theorem}[Strong law of large numbers]
Let $x_1,...,x_n$ be independent, identically distributed (vector-valued) random variables. Suppose $\mathbb{E}\left[\left|x_1\right|\right]<\infty$. Let $ \overline{x}^{(n)} : = \frac{\sum\limits_{i=1}^{n} x_i}{n}$  Then $$ \overline{x}^{(n)} \overset{\text{a.s.}}{\longrightarrow} \mathbb{E}\left[x_1\right],$$
\end{theorem}
a proof can be found in \cite{etemadi1981elementary}.
Of great importance are central-limit theorems \citep{van2000asymptotic}. Here, we will make use of the following version 
\begin{theorem}[Central limit theorem]
Let $x_1,...,x_n$ be  independent, identically distributed (vector-valued) random variables with mean $\mu \in \mathbb{R}^k$ and  covariance $\Sigma \in \mathbb{R}^{k \times k}$.  Let $ \overline{x}^{(n)} : = \frac{\sum\limits_{i=1}^{n} x_i}{n}$. Then we have $$\sqrt{n}\left( \overline{x}^{(n)}- \mu \right) \overset{d}{\longrightarrow} \mathcal{N}(0,\Sigma)$$
\end{theorem}

\subsubsection{Asymptotics of M- and Z-Estimators}
Furthermore, we will use two theorems that provide consistency of estimators that are only implicitly defined as the maximizers (M-estimators) or roots (Z-estimators) of certain equations. For a proof of the following theorems, see \cite{van2000asymptotic}.
\begin{theorem}\label{M-estimator}
Consider a random function $M^{(n)}(\theta,y)$ and a deterministic function $M(\theta)$ such that $$ \underset{\theta \in \Theta}{\sup} \left\Vert M^{(n)}(\theta,y) - M(\theta) \right\Vert \overset{\mathbb{P}}{\longrightarrow} $$ and for all $\varepsilon >0$, 
$$ \underset{\theta : d(\theta,\theta_0) \geq \varepsilon}{\sup} M(\theta) < M(\theta_0).$$
Let $\hat{\theta}^{(n)}$ be an estimator, such that $$M^{(n)}(\hat{\theta}^{(n)},y) \geq M^{(n)}(\theta_0,y) - R^{(n)},$$
with $R^{(n)} = o_{\mathbb{P}}(1).$ Then we have consistency, $$\hat{\theta}^{(n)} \overset{\mathbb{P}}{\longrightarrow} \theta_0.$$
\end{theorem}
The second theorem is very similar, and follows from the first, but for convenience and to align later notation, we state it here as a separate theorem.
\begin{theorem}\label{Z-estimator}
Consider a random function $\Psi^{(n)}(\theta,y)$ and a deterministic function $\Psi(\theta)$ with $\Psi(\theta_0) = 0$,  such that $$ \underset{\theta \in \Theta}{\sup} \left\Vert \Psi^{(n)}(\theta,y) - \Psi(\theta)\right\Vert \overset{\mathbb{P}}{\longrightarrow} $$ and for all $\varepsilon >0$, 
$$ \underset{\theta : d(\theta,\theta_0) \geq \varepsilon}{\inf} \left\Vert \Psi(\theta)\right\Vert >  0.$$
Let $\hat{\theta}^{(n)}$ be an estimator, such that $$\Psi^{(n)}(\hat{\theta}^{(n)},y) = R^{(n)},$$
with $R^{(n)} = o_{\mathbb{P}}(1).$ Then we have consistency, $$\hat{\theta}^{(n)}  \overset{\mathbb{P}}{\longrightarrow} \theta_0.$$
\end{theorem}

\subsection{Multivariate Calculus} 

\subsubsection{Notation}\label{sec:mvnotation}

In the following, we will explain notation related to multivariate derivatives that will be used in the main text. Consider a differentiable function 
\begin{equation}
\begin{aligned}
f: \mathbb{R}^m &\to \mathbb{R}^n \\
x = (x_1,...,x_m) &\mapsto \left(f_1(x),...,f_n(x)\right).
\end{aligned}
\end{equation}
In this situation, the Jacobian of $f$ at $a$ is defined as
\begin{equation}
J_f(a) := \begin{pmatrix} \frac{\partial f_1}{\partial x_1} (a) & \cdots &  \frac{\partial f_1}{\partial x_m}(a) \\ \vdots & \ddots &  \vdots \\ \frac{\partial f_n}{\partial x_1} (a) & \cdots &  \frac{\partial f_n}{\partial x_m}(a) \end{pmatrix}.
\end{equation}
In the case that $n=1$ and $f$ is two times differentiable we define the Hessian matrix at $a$ in the following way
\begin{equation}
H_f(a): = \begin{pmatrix} \frac{\partial^2 f}{\partial x_1 \partial x_1} (a) & \cdots & \frac{\partial^2 f}{\partial x_1 \partial x_m} (a) \\ \vdots & \ddots & \vdots \\ \frac{\partial^2 f}{\partial x_m \partial x_1} (a) & ... & \frac{\partial^2 f}{\partial x_m \partial x_m} (a) \end{pmatrix}.
\end{equation}
Note that for the two-times continuously differentiable functions we consider here we have that $H_f(a) = H_f(a)^T$, by the Schwarz theorem. 
In this work, we also repeatedly are in the situation where we want to refer to quadratic subparts of the Hessian. Consider a partition 
\begin{equation}
\{1,...,m\} = \bigcup\limits_{j=1}^{l} \{ i_{j,1},...,i_{j,d_j} \}
\end{equation} 
and correspondingly denote
\begin{equation}
\boldsymbol{x}_j = (x_{i_{j,1}},...,x_{i_{j,d_j}}) 
\end{equation}
then we will denote the quadratic subpart of the Hessian comprising only coordinates in the $j$-th set as
\begin{equation}
H_f^{\boldsymbol{x}_j}(a) = \begin{pmatrix} \frac{\partial^2 f}{\partial x_{i_{j,1}} \partial x_{i_{j,1}}} (a) & \cdots & \frac{\partial^2 f}{\partial x_1 \partial x_{i_{j,d_j}}} (a) \\ \vdots & \ddots & \vdots \\ \frac{\partial^2 f}{\partial x_{i_{j,d_j}} \partial x_{i_{j,1}}} (a) & ... & \frac{\partial^2 f}{\partial x_{i_{j,d_j}} \partial x_{i_{j,d_j}}} (a) \end{pmatrix}.
\end{equation}
In the main text we will omit using bold-faced symbols, that were used here to distinguish between coordinate subscripts and partition subscripts.  If not stated otherwise, subscripts of model parameters will always refer to partition membership, not necessarily to single coordinates.
When it serves the purpose of clarity, we will also use the more explicit notation in terms of partial-derivatives. That is, we have
\begin{equation}
\frac{\partial f}{\partial {\boldsymbol{x}_j} }(a) = \frac{\partial f(x)}{\partial {\boldsymbol{x}_j} }\bigg \rvert_{a} = \left(\frac{\partial f}{\partial x_{i_{j,1}}}(a),...,\frac{\partial f}{\partial x_{i_{j,d_j}}}(a) \right)^T
\end{equation} 
and 
\begin{equation}
\frac{\partial^2 f}{\partial {\boldsymbol{x}_j} \partial {\boldsymbol{x}_j}^T  }(a) = \frac{\partial^2 f(x)}{\partial {\boldsymbol{x}_j} \partial {\boldsymbol{x}_j}^T  }\bigg \rvert_{a} =H_f^{\boldsymbol{x}_j}(a).
\end{equation} 

\subsubsection{Matrix Derivatives}

In a great part of this work we deal with the general case of variational Laplace in a multivariate setting. In this section we state results 
In particular, derivations of variational Laplace involve differentiating expressions with respect to a matrix. Here, we state two important results that are necessary for these derivations. Let $X$ be any matrix. Then we have
\begin{equation}\label{derivtrace}
\frac{\partial \operatorname{tr}(XA)}{\partial X} = A^T,
\end{equation}
for any matrix $A$ such that  the dimensions allow the above expression; see also \cite{petersen}, equation 100.
Furthermore,for an invertible Matrix $X$ we have that
\begin{equation}\label{derivlog}
\frac{\partial \ln |\det(X)|}{\partial X} = (X^{-1})^T =(X^T)^{-1}
\end{equation}
see also \cite{petersen}, equation 57. We note that these formulae are not assuming any additional structure (e.g. symmetry) of the matrix $X$.
That is, when one uses these expressions to find a maximal covariance matrix, the result has to be subsequently inspected whether it yields a valid covariance matrix.

\subsection{Probability Theory}
For the derivations of variational Laplace, we are also repeatedly in the situation where we want to compute the expectation of quadratic forms of a random vector.The following result gives a general formula on how to compute such expectations.
 Let $x$ be a random vector (in $\mathbb{R}^p$) with expectation $\mu$ and covariance matrix $\Sigma$. Furthermore, let
$A$ be a symmetric $p\times p$-matrix. Then we have
\begin{equation}\label{expectquad}
\mathbb{E}[x^T A x] = \mu^TA\mu + \operatorname{tr}(A\Sigma)
\end{equation}
which follows from the linearity of the trace operator and the symmetry of $A$; see also \cite{petersen}, equation 328.

We will also use the fact that the differential entropy of a Gaussian with mean $\mu \in \mathbb{R}^p$ and covariance matrix $\Sigma \in \mathbb{R}^{p \times p}$ evaluates to 
\begin{equation}\label{Gaussianentropy} \frac{1}{2}  \ln \det \Sigma  + \frac{p}{2} +\frac{p}{2} \ln (2\pi),\end{equation}
consult \citet{murphy2012machine} for reference.

\subsection{Lamberts $W$-Function}\label{sec:lambertw}
Consider the function \begin{equation}\begin{aligned}g: \mathbb{R}_{\geq0} \to  \mathbb{R}_{\geq0}\\ x \mapsto x \exp x.\end{aligned}\end{equation}
As this is a bijection, one can define a unique inverse of this function by 
\begin{equation}\begin{aligned}g^{-1} : \mathbb{R}_{\geq0} \to  \mathbb{R}_{\geq0} \\ g(g^{-1}(x)) = x.\end{aligned}\end{equation}
We denote this inverse  $W$, that is we have $$ W(x) \exp(W(x)) = x .$$
We will refer to this function as Lambert's $W$ function; note that we are treating a special case of positive real numbers here, whereas in the literature the name Lambert's $W$ also commonly refers to the multivalued function that extends the above consideration to complex numbers.
One important property of $W$ is that it helps us identify the zeros of functions that include exponential and linear terms. Indeed, we have the following lemma
\begin{lemma}\label{lambertlemma}
Let $a > 0, \ b \geq 0$ and $c \in \mathbb{R}$. Then, the unique solution to the equation $e^x(ax + c) = b$  is given by $$ x = W(\frac{b e^{\frac{c}{a}}}{a}) -\frac{c}{a}.$$
\end{lemma}
To see this, let us first define an auxiliary variable $$u = x + \frac{c}{a}.$$
Thus, the equation reads $$ e^{u-\frac{c}{a}} (au) = b.$$
We rewrite this as $$  u e^u  = b \frac{e^{\frac{c}{a}}}{a},$$
and note that by our assumption, the right hand side is nonnegative. Thus, by definition, the solution is given by
$$ u = W(b \frac{e^{\frac{c}{a}}}{a}).$$
Resubstitution  $x = u + \frac{c}{a}$ then yields 
$$ x = W(b \frac{e^{\frac{c}{a}}}{a}) - \frac{c}{a}$$
as desired.

\section{Disambiguation of Variational Laplace in the Nonlinear Transform Model}\label{sec:disambiguation}
In this section, we want to give the reader a disambiguation of variational Laplace and closely related methods that are frequently cited alongside in the cognitive neuroscience community. Furthermore, we want to highlight certain specific situations where variational Laplace is used in a way slightly deviating from the update equations presented here. 
\subsection{Taylor Approximation of the Transform Function}
In the specific situation of the model we are studying here, notions very similar to variational Laplace can be obtained by a Taylor approximation of the function $f$ instead of the whole log-joint probability.  To see this, we write out the full $\operatorname{ELBO}$ in this model, up to constants, as
\begin{equation}
\begin{aligned}
& \operatorname{ELBO}(\mu_\theta,\Sigma_\theta,\mu_\lambda,\Sigma_\lambda) = \\
&\int \int  \ln p(y,\theta,\lambda) q_{\mu_\theta,\Sigma_\theta}(\theta) q_{\mu_\lambda,\Sigma_\lambda}(\lambda) d\theta d\lambda   -  \int \ln q_{\mu_\theta,\Sigma_\theta}(\theta)  q_{\mu_\theta,\Sigma_\theta}(\theta) d\theta - \int \ln q_{\mu_\lambda,\Sigma_\lambda}(\lambda) q_{\mu_\lambda,\Sigma_\lambda}(\lambda) d\lambda  \\
&-\frac{1}{2}  \int \int  (y-f(\theta))^T Q(\lambda)^{-1} (y-f(\theta))  q_{\mu_\theta,\Sigma_\theta}(\theta) q_{\mu_\lambda,\Sigma_\lambda}(\lambda) d\theta d\lambda   \\
&- \frac{1}{2} \int (\theta-m_\theta)^T  S_\theta^{-1} (\theta-m_\theta) q_{\mu_\theta,\Sigma_\theta}(\theta)d\theta\\
&-\frac{1}{2} \int \left( (\lambda - m_\lambda)^T S_\lambda^{-1} (\lambda-m_\lambda)  + \ln \det(Q(\lambda) \right) q_{\mu_\lambda,\Sigma_\lambda}(\lambda)  d\lambda \\
&+ \mathbb{H} [q_{\mu_\theta,\Sigma_\theta}(\theta)] + \mathbb{H} [ q_{\mu_\lambda,\Sigma_\lambda}(\lambda)] + C.
\end{aligned}
\end{equation}
The differential entropies evaluate to 
\begin{equation}
\begin{aligned}
&\mathbb{H} [q_{\mu_\theta,\Sigma_\theta}(\theta)]   = \frac{1}{2} \ln \det( \Sigma_\theta) + C\\
&\mathbb{H} [ q_{\mu_\lambda,\Sigma_\lambda}(\lambda)]=  \frac{1}{2} \ln \det( \Sigma_\theta) + C. \\
\end{aligned}
\end{equation}
Also, the expectations of the quadratic forms can be carried out directly, using \eqref{expectquad}:
\begin{equation}
\begin{aligned}
&- \frac{1}{2} \int (\theta-m_\theta)^T  S_\theta^{-1} (\theta-m_\theta) q_{\mu_\theta,\Sigma_\theta}(\theta)d\theta =   -\frac{1}{2} \left( (\mu_\theta-m_\theta)^T  S_\theta^{-1} (\mu_\theta-m_\theta)  + \operatorname{tr}( S_\theta^{-1} \Sigma_\theta)\right) \\
&- \frac{1}{2} \int (\lambda-m_\lambda)^T  S_\lambda^{-1} (\lambda-m_\lambda) q_{\mu_\lambda,\Sigma_\lambda}(\lambda)d\lambda =   -\frac{1}{2} \left( (\mu_\lambda-m_\lambda)^T  S_\lambda^{-1} (\mu_\lambda-m_\lambda)  + \operatorname{tr}( S_\lambda^{-1} \Sigma_\lambda)\right). \\
\end{aligned}
\end{equation}
The mixed term can then be approximated by a Taylor series up to first order, such that
\begin{equation}
\begin{aligned}
&-\frac{1}{2}  \int \int  (y-f(\theta))^T Q(\lambda)^{-1} (y-f(\theta))  q_{\mu_\theta,\Sigma_\theta}(\theta) q_{\mu_\lambda,\Sigma_\lambda}(\lambda) d\theta d\lambda \\
&= -\frac{1}{2}  \int   (y-f(\theta))^T \mathbb{E}_{q_{\mu_\lambda,\Sigma_\lambda}} [Q(\lambda)^{-1}] (y-f(\theta))  q_{\mu_\theta,\Sigma_\theta}(\theta)  d\theta\\
 &\approx -\frac{1}{2} \int \left(y-(f(\mu_\theta) + J_f(\mu_\theta)^T(\theta-\mu_\theta)\right)^T \mathbb{E}_{q_{\mu_\lambda,\Sigma_\lambda}} [Q(\lambda)^{-1}]   \left(y-(f(\mu_\theta) + J_f(\mu_\theta)^T(\theta-\mu_\theta)\right)  q_{\mu_\theta,\Sigma_\theta}(\theta),
\end{aligned}
\end{equation}
where the expansion point, as before, is chosen to be $\mu_\theta$.
This integral can now be evaluated explicitly: 
\begin{equation}
\begin{aligned}
&\int \left(y-(f(\mu_\theta) + J_f(\mu_\theta)^T(\theta-\mu_\theta)\right)^T\mathbb{E}_{q_{\mu_\lambda,\Sigma_\lambda}} [Q(\lambda)^{-1}]   \left(y-(f(\mu_\theta) + J_f(\mu_\theta)^T(\theta-\mu_\theta)\right)  q_{\mu_\theta,\Sigma_\theta}(\theta) \\
&= \int y^T   \mathbb{E}_{q_{\mu_\lambda,\Sigma_\lambda}} [Q(\lambda)^{-1}]  y  - 2 y^T  \mathbb{E}_{q_{\mu_\lambda,\Sigma_\lambda}} [Q(\lambda)^{-1}] f(\mu_\theta) + f(\mu_\theta)^T  \mathbb{E}_{q_{\mu_\lambda,\Sigma_\lambda}} [Q(\lambda)^{-1}]  f(\mu_\theta) q_{\mu_\theta,\Sigma_\theta} (\theta) d\theta  \\
&+ \int \left(J_f(\mu_\theta)^T(\theta-\mu_\theta)\right)^T  \mathbb{E}_{q_{\mu_\lambda,\Sigma_\lambda}} [Q(\lambda)^{-1}] \left(J_f(\mu_\theta)^T(\theta-\mu_\theta)\right) q_{\mu_\theta,\Sigma_\theta}(\theta)d\theta\\
&- 2  \int (\theta-\mu_\theta) ^T  q_{\mu_\theta,\Sigma_\theta}(\theta)d\theta  J_f(\mu_\theta) \mathbb{E}_{q_{\mu_\lambda,\Sigma_\lambda}} [Q(\lambda)^{-1}]  y  \  \\
&- 2  \int (\theta-\mu_\theta) ^T  q_{\mu_\theta,\Sigma_\theta}(\theta)d\theta  J_f(\mu_\theta) \mathbb{E}_{q_{\mu_\lambda,\Sigma_\lambda}} [Q(\lambda)^{-1}]  f(\mu_\theta)   \\
&= \left (y-f(\mu_\theta)\right)^T   \mathbb{E}_{q_{\mu_\lambda,\Sigma_\lambda}} [Q(\lambda)^{-1}] \left (y-f(\mu_\theta)\right) + \operatorname{tr} [ J_f(\mu_\theta) \mathbb{E}_{q_{\mu_\lambda,\Sigma_\lambda}} [Q(\lambda)^{-1}] J_f(\mu_\theta)^T \Sigma_\theta]
\end{aligned}
\end{equation}
where the last line follows from the fact that the integrals in last two terms terms evaluate to $0$ due to the choice of the expansion point for the Taylor series being equal to the mean, and using the evaluation of the integral in the second line with the standard formula for the expectation of quadratic forms \eqref{expectquad}.
Hence, the (approximated) $\operatorname{ELBO}$  can be written as
\begin{equation}
\begin{aligned}
&\operatorname{ELBO}(\mu_\theta,\Sigma_\theta,\mu_\lambda,\Sigma_\lambda) =\\
& -\frac{1}{2}\left( \left (y-f(\mu_\theta)\right)^T   \mathbb{E}_{q_{\mu_\lambda,\Sigma_\lambda}} [Q(\lambda)^{-1}] \left (y-f(\mu_\theta)\right) + \operatorname{tr} [ J_f(\mu_\theta) \mathbb{E}_{q_{\mu_\lambda,\Sigma_\lambda}} [Q(\lambda)^{-1}] J_f(\mu_\theta)^T \Sigma_\theta]\right) \\
&-\frac{1}{2} \left( (\mu_\theta-m_\theta)^T  S_\theta^{-1} (\mu_\theta-m_\theta)  + \operatorname{tr}( S_\theta^{-1} \Sigma_\theta)\right) \\
&-\frac{1}{2} \left( (\mu_\lambda-m_\lambda)^T  S_\lambda^{-1} (\mu_\lambda-m_\lambda)  + \operatorname{tr}( S_\lambda^{-1} \Sigma_\lambda)\right) \\
& -\frac{1}{2}  \mathbb{E}_{ q_{\mu_\lambda,\Sigma_\lambda}} [\ln \det(Q(\lambda)^{-1})] 
 + \frac{1}{2}  \ln \det (\Sigma_\theta) 
 + \frac{1}{2} \ln \det (\Sigma_\lambda). \\
\end{aligned}
\end{equation}
If this expression now ought to be optimized with respect to the variational covariance parameter, we can proceed again by calculating the matrix-derivative with respect to $\Sigma_\theta$.
This yields
\begin{equation}
\frac{\partial}{\partial \Sigma_\theta}   \operatorname{ELBO}(\mu_\theta,\Sigma_\theta,\mu_\lambda,\Sigma_\lambda) =  \frac{1}{2} \Sigma_\theta^{-1} -\frac{1}{2} (J_f(\mu_\theta) \mathbb{E}_{q_{\mu_\lambda,\Sigma_\lambda}} [Q(\lambda)^{-1}] J_f(\mu_\theta)^T + S_\theta^{-1}),
\end{equation}
which is $0$ iff
\begin{equation}
\Sigma_\theta = \left(J_f(\mu_\theta) \mathbb{E}_{q_{\mu_\lambda,\Sigma_\lambda}} [Q(\lambda)^{-1}]J_f(\mu_\theta)^T + S_\theta^{-1}\right)^{-1}.
\end{equation}
Note that in this line of derivation, no issue with positive definite-ness arises.
Furthermore, optimization with respect to $\mu_\theta$ corresponds to maximizing
\begin{equation}
\begin{aligned}
\widetilde{I}(\mu_\theta) : =  &-\frac{1}{2} \left (y-f(\mu_\theta)\right)^T   \mathbb{E}_{q_{\mu_\lambda,\Sigma_\lambda}} [Q(\lambda)^{-1}] \left (y-f(\mu_\theta)\right) \\
 &-\frac{1}{2} \operatorname{tr} [ J_f(\mu_\theta) \mathbb{E}_{q_{\mu_\lambda,\Sigma_\lambda}} [Q(\lambda)^{-1}] J_f(\mu_\theta)^T \Sigma_\theta] \\
&-\frac{1}{2} (\mu_\theta-m_\theta)^T  S_\theta^{-1} (\mu_\theta-m_\theta).
\end{aligned}
\end{equation}
We see that this is similar to the function $I(\mu_\theta)$ that we introduced while discussing variational Laplace. 
Especially, if one proceeds to optimize this function via gradient-based methods and adheres to the proposition of ignoring higher-order-derivatives of $f$, the trace-term here will not matter, as its derivatives involve higher-order derivatives of $f$.
One will thus be optimizing 
\begin{equation}
-\frac{1}{2} \left (y-f(\mu_\theta)\right)^T   \mathbb{E}_{q_{\mu_\lambda,\Sigma_\lambda}} [Q(\lambda)^{-1}] \left (y-f(\mu_\theta)\right) -\frac{1}{2} (\mu_\theta-m_\theta)^T  S_\theta^{-1} (\mu_\theta-m_\theta)
\end{equation}
with respect to $\mu_\theta$. Hence, the only difference to the variational Laplace scheme lies in influence of the fixed variational density$q_\lambda$.
 Here, the expectation of $Q(\lambda)^{-1}$ is taken under this density, while in variational Laplace  $Q(\lambda)^{-1}$ is evaluated at the expected value of $\lambda$ with respect to this density.

\subsection{Taylor Expansion at Last Iteration Step}
The following approach is also depending on a Taylor expansion of the transform $f$. Although not explicitly called variational Laplace by the authors, the method is similar in nature and cited in similar occasions\citep{chappell2008variational}.
Here, one considers the update of $q_\theta(\theta)$ with $q_\lambda(\lambda)$ given. According to the fundamental Lemma of variational Inference, the optimal form of this density is implicitly given by
\begin{equation}
\begin{aligned}
\ln q_\theta(\theta) &=  \int \ln\left( p(y|\theta,\lambda) p(\theta) p(\lambda)\right) q(\lambda) d\lambda + C \\
		&=  -\frac{1}{2} (y-f(\theta))^T \mathbb{E}_{q_\lambda}[Q(\lambda)]^{-1}  (y-f(\theta)) \\
		&- \frac{1}{2} (\theta-m_\theta) S_\theta^{-1} (\theta-m_\theta) \\
		&- \frac{1}{2}  \mathbb{E}_{q_\lambda} [\ln \det Q(\lambda)]  
		+ \mathbb{E}_{q_\lambda} [\ln p(\lambda)] + C .
\end{aligned}
\end{equation}
The first insight, as is typical for the free-form approach to variational inference, is that all terms not including $\theta$ can be neglected. This is because these terms will be normalized out to render $q_\theta$ a proper density and do not contribute to determining the functional form of the density.
Thus, one only has to consider the terms
\begin{equation}\label{Gaussianterms}
 -\frac{1}{2} (y-f(\theta))^T \mathbb{E}_{q_\lambda}[Q(\lambda)^{-1}]   (y-f(\theta)) - \frac{1}{2} (\theta-m_\theta) S_\theta^{-1} (\theta-m_\theta).
\end{equation}
By a linearization of $f$ via a Taylor approximation, the above equations will result in a Gaussian form for $q_\theta$. Thus, the method iteratively employs a Taylor approximation of $f$ centered on mean of the Gaussian that resulted from the previous approximation. 
That is, assume $\mu_\theta$ is given from the previous iteration. 
Then, one can approximate
\begin{equation}
f(\theta) \approx f(\mu_\theta) + J_f(\mu_\theta)^T(\theta-\mu_\theta)
\end{equation}
Plugging this in into \ref{Gaussianterms} yields
\begin{equation}
\begin{aligned}
 &-\frac{1}{2} \left( y-(f(\mu_\theta) + J_f(\mu_\theta)^T(\theta-\mu_\theta)\right)^T \mathbb{E}_{q_\lambda}[Q(\lambda)^{-1}]   \left(y-(f(\mu_\theta) + J_f(\mu_\theta)^T(\theta-\mu_\theta)\right) \\ &- \frac{1}{2} (\theta-m_\theta) S_\theta^{-1} (\theta-m_\theta).
\end{aligned}
\end{equation}
Now, one can rewrite this as 
\begin{equation}
\begin{aligned}
&-\frac{1}{2}   (y-f(\mu_\theta))^T  \mathbb{E}_{q_\lambda}[Q(\lambda)^{-1}]  (y-f(\mu_\theta)) \\
&- (y-f(\mu_\theta))^T  \mathbb{E}_{q_\lambda}[Q(\lambda)^{-1}]  J_f(\mu_\theta)^T \mu_\theta \\ 
&+  \theta^T J_f(\mu_\theta) \mathbb{E}_{q_\lambda}[Q(\lambda)^{-1}]   (y-f(\mu_\theta)) \\
&- \frac{1}{2}  \theta^T J_f(\mu_\theta)\mathbb{E}_{q_\lambda}[Q(\lambda)^{-1}]J_f(\mu_\theta)^T \theta \\
& + \theta^T  J_f(\mu_\theta) \mathbb{E}_{q_\lambda}[Q(\lambda)^{-1}]  J_f(\mu_\theta)^T  \mu_\theta \\
& -\frac{1}{2} \mu_\theta^T  J_f(\mu_\theta)  \mathbb{E}_{q_\lambda}[Q(\lambda)^{-1}] J_f(\mu_\theta)^T \mu_\theta \\
& -\frac{1}{2} \theta^T S_\theta^{-1} \theta 
 + \theta^T S_\theta^{-1} \mu_\theta
 -\frac{1}{2} \mu_\theta^T S_\theta^{-1} \mu_\theta.
\end{aligned}
\end{equation}
Again, all terms that do not depend on $\theta$ will not influence the functional form of $q_\theta$. Dropping these terms, the above can be subsumed into
\begin{equation} \label{completesquare1}
\begin{aligned}
&-\frac{1}{2} \theta^T \left(  J_f(\mu_\theta)\mathbb{E}_{q_\lambda}[Q(\lambda)^{-1}]J_f(\mu_\theta)^T+ S_\theta^{-1} \right) \theta \\
&+ \theta^T \left( J_f(\mu_\theta) \mathbb{E}_{q_\lambda}[Q(\lambda)^{-1}] \left(y -f(\mu_\theta) + J_f(\mu_\theta)^T \mu_\theta \right) + S_\theta^{-1} \mu_\theta \right).
\end{aligned}
\end{equation}
By completing the square, one thus concludes that the functional form  of $q_\theta$ under this approximation is a Gaussian
\begin{equation}\label{completesquare2}
\begin{aligned}
q_\theta(\theta) &= \mathcal{N}(\theta;\mu_\theta^*,\Sigma_\theta^*) \\
\mu_\theta^* &= {\Sigma_\theta^*}^{-1} \left(J_f(\mu_\theta) \mathbb{E}_{q_\lambda}[Q(\lambda)^{-1}] \left(y -f(\mu_\theta) + J_f(\mu_\theta)^T \mu_\theta \right) + S_\theta^{-1} \mu_\theta \right) \\
\Sigma_\theta^*  &= \  \left(J_f(\mu_\theta)\mathbb{E}_{q_\lambda}[Q(\lambda)^{-1}]J_f(\mu_\theta) + S_\theta^{-1}\right)^{-1}.
\end{aligned}
\end{equation}
This approximate density can then be used to update $q_\lambda$. To make this possible, it might be needed to likewise approximate $f$ up to first order. Then, at the next iteration, $\mu_\theta^*$ takes the role of the expansion point. 
We note that in this approach, in contrast to previously presented approaches, the optimization of $\mu_\theta$ depends on the current estimate for this value explicitly, while in the other schemes it is  implicitly taken into account due to the current value of $\mu_\lambda$, that will depend on $\mu_\theta$ through the previous iteration.

\subsection{Linear Models}\label{sec:linearlaplace}
A special case of the discussed model arises when the transform function is actually linear. That is, if one has $f(\theta) = X\theta$ for some matrix $X$. In this case, one can partly evaluate the $\operatorname{ELBO}$ directly, without utilizing a Taylor approximation. 
This is because in this case, by the fundamental Lemma of variational inference, the optimal variational density of $\theta$ will actually be a Gaussian. To see this, assume as usually that $q_\lambda$ is fixed. Then, by the fundamental Lemma, one has that
\begin{equation}
\ln q_\theta(\theta) =  -\frac{1}{2} (y-X\theta)^T \mathbb{E}_{q_\lambda}[Q(\lambda)^{-1}]   (y-X\theta) - \frac{1}{2} (\theta-m_\theta) S_\theta^{-1} (\theta-m_\theta) + C,
\end{equation}
where we have ignored all terms devoid of $\lambda$ as in \eqref{Gaussianterms}. These terms now lead to the form of the $q_\theta$ being Gaussian, with the same considerations as in \eqref{completesquare1}  and \eqref{completesquare2}. Hence, $q_\theta$ is given by
\begin{equation}
\begin{aligned}
q_\theta(\theta) &= \mathcal{N}(\theta;\mu_\theta^*,\Sigma_\theta^*) \\
\mu_\theta^* &= {\Sigma_\theta^*}^{-1} \left(X^T \mathbb{E}_{q_\lambda}[Q(\lambda)^{-1}]y + S_\theta^{-1} \mu_\theta \right) \\
\Sigma_\theta^* &= \left(X^T\mathbb{E}_{q_\lambda}[Q(\lambda)^{-1}]X + S_\theta^{-1}\right)^{-1}.
\end{aligned}
\end{equation}
Thus, if $\mathbb{E}_{q_\lambda}[Q(\lambda)^{-1}]$ can be evaluated in a straightforward manner, then one has a closed-form update for the variational density on $\theta$.
This may for example be the case in when $q_\lambda$ is restricted to be Gaussian  and the inverse covariance matrix
takes the form $$ Q(\lambda)^{-1} = \sum_{l} \exp(\lambda_l) \Phi_l,$$ for some set of fixed positive definite matrizes $\Phi_l$. In variational Laplace, one could thus proceed to update $q_\lambda$ with these optimal values directly.
However, applications in the literature typically seem to use a similar parameterization for the covariance matrix itself $$Q(\lambda) = \sum_{l} \exp(\lambda_l) \Phi_l,$$
where the expectation may not be as straightforward to evaluate \citep{friston2008Bayesian,friston2019variational,lopez2014algorithmic}.  Also consider the appendix of \citet{friston2007variational} for a discussion of when the respective parameterizations are appropriate.
We suppose that in the latter case the covariance matrix at the variational expectation parameter $\mu_\lambda^*$ is used, such that 
\begin{equation}
\begin{aligned}
q_\theta(\theta) &= \mathcal{N}(\theta;\mu_\theta^*,\Sigma_\theta^*) \\
\mu_\theta^* &= {\Sigma_\theta^*}^{-1} \left(X^T Q(\mu_\lambda^*)^{-1}y + S_\theta^{-1} \mu_\theta \right) \\
\Sigma_\theta^* &= \left(X^T Q(\mu_\lambda^*) X + S_\theta^{-1}\right)^{-1}.
\end{aligned}
\end{equation}
although in the aforementioned work, no algorithmic description of this subpart of the variational Laplace equations is given. \cite{starke2017} derive similar update equations explicitly by dropping the Hessian term with respect to $\lambda$ in the update of $\mu_\theta$.
Note that in "full" variational Laplace, as given by equations \eqref{varlap} and following, one would have to identify the maximum of 
\begin{equation}
\begin{aligned}
I(\mu_\theta) = &-\frac{1}{2} \left(y-X \mu_\theta \right)^T Q\left(\mu_\lambda^*\right)^{-1} \left(y-X \mu_\theta \right) \\
		    & -\frac{1}{2} \left(\mu_\theta - m_\theta\right)^T S_\theta^{-1}\left(\mu_\theta - m_\theta\right) \\
		    & + \frac{1}{2} \operatorname{tr} \left(\Sigma_\lambda^*   \frac{\partial^2}{\partial^2 \lambda} \left(-\frac{1}{2} \left(y-X \mu_\theta \right)^T Q(\lambda)^{-1}\left(y-X \mu_\theta \right) \bigg\rvert_{\lambda = \mu_\lambda^*}\right)\right),
\end{aligned}
\end{equation}
which only coincides with the above derivations if the third term is ignored.

\subsection{Flexible Use of Variational Laplace}
Generally, the variational Laplace approach offers some degree of flexibility. It may be advantageous to combine the approach with the free-form-approach: By the fundamental lemma, 
\begin{equation}
q_{\theta_j}(\theta_j)  \propto \exp\left( \int \prod_{k \neq j} q_{\theta_k} (\theta_{k}) \ln{p(y,\theta)} d \theta_{/j}\right).
\end{equation}
Thus, if the choice of the prior and the model is such that for a subset of the parameters, the above free-form scheme would yield an update scheme preserving the functional form of $q_{\theta_j}$ over all iterations, these densities may be updated by the free-form-approach.
The expectations that are needed to obtain may still need an approximation, that may be obtained via variational Laplace.
Let us consider an example in the nonlinear transform model to clarify these statements \citep{daunizeau2017variational}. One has that
\begin{equation}
\ln q_\lambda(\lambda) =  -\frac{1}{2}\mathbb{E}_{q_\theta} \left[(y-f(\theta))^T Q(\lambda)^{-1}   (y-f(\theta)\right]  + \frac{1}{2} \ln \det \left(Q(\lambda)^{-1}\right) + \ln p(\lambda)  + C
\end{equation}
by the fundamental Lemma. If $Q(\lambda)^{-1} = \lambda \Phi$ for some fixed $\Phi$, one can pull $\lambda$ out of the expression and ignore constant terms to get
\begin{equation}
\ln q_\lambda(\lambda) = -\frac{\lambda}{2}\mathbb{E}_{q_\theta} \left[(y-f(\theta))^T \Phi   (y-f(\theta)\right]  +\frac{d }{2} \ln \lambda + \ln p(\lambda) + C.
\end{equation}
If $p(\lambda)$ is specified as a gamma density with parameters $\alpha,\beta$ one has
\begin{equation}
\ln q_\lambda(\lambda) = -\frac{\lambda}{2}\mathbb{E}_{q_\theta} \left[(y-f(\theta))^T \Phi   (y-f(\theta)\right]  +\frac{d }{2} \ln \lambda  + \alpha-1 \ln \lambda  - \beta \lambda + C
\end{equation}
and one sees, that the optimal density will again be given by a gamma density with the modified parameters
\begin{equation}
\begin{aligned}
& \alpha^* = \alpha + \frac{d }{2} \\
& \beta^* = \beta + \frac{1}{2} \mathbb{E}_{q_\theta} \left[(y-f(\theta))^T \Phi   (y-f(\theta)\right].
\end{aligned}
\end{equation}
The remaining expectation can then be approximated by the same arguments as in variational Laplace, dropping second derivatives, as 
\begin{equation}
\mathbb{E}_{q_\theta} \left[(y-f(\theta))^T \Phi   (y-f(\theta)\right] \approx    \left(y-f(\mu_\theta^*)\right) \Phi (y-f(\mu_\theta^*) +  \operatorname{tr} \left(\Sigma_\theta^*   J_f(\mu_\theta^*)^T \Phi J_f\left(\mu_\theta^*\right)\right),
\end{equation}
where $\mu_\theta^*$ and $\Sigma^*$ are the expectation and covariance parameter of the previously updated density $q_\theta$. This density in turn, as before, can be updated by variational Laplace, where the "variational energy" that has to be optimized now reads
\begin{equation}
I(\mu_\theta) = -\frac{1}{2}  \left(y-f(\theta)\right)^T \mathbb{E}_{q_\lambda}[\lambda] \Phi   \left(y-f(\theta)\right) - \frac{1}{2} (\mu_\theta - m_\theta)^T S_\theta^{-1} (\mu_\theta - m_\theta),
\end{equation}
where the expectation that is needed here can be evaluated explicitly due to the closed form of $q_\lambda$. Similar approaches combining free-form and variational Laplace have been employed for inference in stochastic nonlinear dynamic causal models \citep{daunizeau2009variational} and  
classification studies \citep{brodersen2013variational}.
\end{appendices}
\section*{References}
\bibliography{masterbib}
\end{document}